\documentclass{amsart}
\usepackage{graphicx}
\usepackage{amssymb}
\usepackage{amsfonts}
\swapnumbers
\newtheorem{theorem}{Theorem}[section]
\newtheorem{lemma}[theorem]{Lemma}

\newtheorem{proposition}[theorem]{Proposition}
\theoremstyle{definition}

\newtheorem{remark}[theorem]{Remark}

\numberwithin{equation}{section}
 \theoremstyle{plain}
 
 \numberwithin{equation}{section} %% Comment out for sequentially-numbered
 \numberwithin{figure}{section} %% Comment out for sequentially-numbered
 \theoremstyle{plain}
 \theoremstyle{remark}
 \newtheorem*{acknowledgement*}{Acknowledgement}
\newcommand{\cB}{{\mathcal B}}

\newcommand{\cD}{{\mathcal D}}

\newcommand{\cF}{{\mathcal F}}
\newcommand{\cG}{{\mathcal G}}
\newcommand{\cH}{{\mathcal H}}
\newcommand{\cI}{{\mathcal I}}
\newcommand{\cJ}{{\mathcal J}}

\newcommand{\cL}{{\mathcal L}}

\newcommand{\cP}{{\mathcal P}}

\newcommand{\cW}{{\mathcal W}}
\newcommand{\cX}{{\mathcal X}}

\newcommand{\vt}{{\vartheta}}
\newcommand{\Om}{{\Omega}}
\newcommand{\om}{{\omega}}
\newcommand{\ve}{{\varepsilon}}
\newcommand{\del}{{\delta}}

\newcommand{\gam}{{\gamma}}
\newcommand{\Gam}{{\Gamma}}

\newcommand{\vr}{{\varrho}}
\newcommand{\Sig}{{\Sigma}}
\newcommand{\sig}{{\sigma}}
\newcommand{\al}{{\alpha}}
\newcommand{\be}{{\beta}}
\newcommand{\ka}{{\kappa}}

\newcommand{\vp}{{\varpi}}

\newcommand{\vrho}{{\varrho}}

\newcommand{\vs}{{\varsigma}}

\newcommand{\bbC}{{\mathbb C}}

\newcommand{\bbJ}{{\mathbb J}}
\newcommand{\bbK}{{\mathbb K}}
\newcommand{\bbL}{{\mathbb L}}

\newcommand{\bbR}{{\mathbb R}}
\newcommand{\bbS}{{\mathbb S}}

\newcommand{\bbZ}{{\mathbb Z}}
\newcommand{\bbI}{{\mathbb I}}
\newcommand{\bbW}{{\mathbb W}}
\newcommand{\bbU}{{\mathbb U}}
\newcommand{\bbX}{{\mathbb X}}
\newcommand{\bbV}{{\mathbb V}}
\newcommand{\bfA}{{\bf A}}
\newcommand{\bfB}{{\bf B}}
\newcommand{\bfM}{{\bf M}}

\begin{document}
\title[]{Limit theorems for signatures}%
 \vskip 0.1cm
 \author{ Yuri Kifer\\
\vskip 0.1cm
 Institute  of Mathematics\\
Hebrew University\\
Jerusalem, Israel}%
\address{
Institute of Mathematics, The Hebrew University, Jerusalem 91904, Israel}
\email{ kifer@math.huji.ac.il}%

\thanks{ }
\dedicatory{  }
 \date{\today}
\begin{abstract}\noindent
We obtain strong moment invariance principles for normalized multiple iterated sums and integrals
of the form $\bbS_N^{(\nu)}(t)=N^{-\nu/2}\sum_{0\leq k_1<...<k_\nu\leq Nt}\xi(k_1)\otimes\cdots\otimes\xi(k_\nu)$,
$t\in[0,T]$ and $\bbS_N^{(\nu)}(t)=N^{-\nu/2}\int_{0\leq s_1\leq...\leq s_\nu\leq Nt}\xi(s_1)\otimes\cdots\otimes\xi(s_\nu)ds_1\cdots ds_\nu$, where $\{\xi(k)\}_{-\infty<k<\infty}$ and $\{\xi(s)\}_{-\infty<s<\infty}$ are centered stationary vector processes with some weak
dependence properties. We show, in particular, that (in both cases) the distribution of $\bbS^{(\nu)}_N$ is $O(N^{-\del})$-close, $\del>0$ in the Prokhorov and the Wasserstein metrics to the distribution of certain stochastic
processes $\bbW_N^{(\nu)}$ constructed recursively starting from $W_N=\bbW_N^{(1)}$ which is a Brownian motion
with covariances. This is done by constructing a coupling between $\bbS_N^{(1)}$ and $\bbW_N^{(1)}$, estimating
directly the moment variational norm of $\bbS_N^{(\nu)}-\bbW_N^{(\nu)}$ for $\nu=1,2$ and extending  these estimates to  $\nu>2$ relying on arguments borrowed 
from the rough paths theory.
 In the continuous time we work both under direct weak dependence assumptions and also within the suspension setup
which is more appropriate for applications in dynamical systems.
\end{abstract}
%\footnotetext[1]{}
\maketitle
\markboth{Yu.Kifer}{Limit theorems for signatures}
\renewcommand{\theequation}{\arabic{section}.\arabic{equation}}
\pagenumbering{arabic}

\section{Introduction}\label{sec1}\setcounter{equation}{0}

Let $\{\xi(k)\}_{-\infty<k<\infty}$ and $\{\xi(t)\}_{-\infty<t<\infty}$ be discrete and continuous time $d$-dimensional
stationary processes such that for $s=0$ (and so for all $s$),
\begin{equation}\label{1.1}
E\xi(s)=0.
\end{equation}
The sequences of multiple iterated sums
\[
\Sig^{(\nu)}(v)=\sum_{0\leq k_1<...<k_\nu< [v]}\xi(k_1)\otimes\cdots\otimes\xi(k_\nu),
\]
in the discrete time, and of multiple iterated integrals
\[
\Sig^{(\nu)}(v)=\int_{0\leq u_1\leq...\leq u_\nu\leq v}\xi(u_1)\otimes\cdots\otimes\xi(u_\nu)du_1\cdots du_\nu,
\]
in the continuous time, were called signatures in recent papers related to the rough paths theory, data sciences and machine
learning (see, for instance, \cite{HL}, \cite{DR}, \cite{DET} and references there). Observe that for $\nu=1$ we have above usual sums and integrals.

In this paper we will study limit theorems for the normalized iterated sums and integrals $\bbS_N^{(\nu)}(t)=
N^{-\nu/2}\Sig^{(\nu)}(Nt),\, t\in[0,T]$. Under certain weak dependence conditions on the process $\xi$ it was shown
in \cite{FK} that there exists a Brownian motion with covariances $\cW$ such that the $\nu$-th term of the, so called,
Lyons' extension $\bbW_N^{(\nu)}$ constructed recursively (see (\ref{rec2}) in Section \ref{sec2})
with the rescaled Brownian motion $W_N(s)=N^{-1/2}\cW(Ns)$ and a certain drift term satisfy
\begin{equation}\label{1.2}
\sup_{0\leq t\leq T}|\bbS_N^{(\nu)}(t)-\bbW^{(\nu)}_N(t)|=O(N^{-\ve})\,\,\,\mbox{almost surely (a.s.)}
\end{equation}
for some $\ve>0$ which does not depend on $N$ and, in fact, this was proved in \cite{FK} for the $p$-variational and not just for the supremum norm. This is the strong (or a.s.) invariance principle for iterated sums and integrals
which of course  implies, in particular, the weak convergence of distributions of $\bbS_N^{(\nu)}$ to the distribution of $\bbW_N^{(\nu)}$ (which does not depend on $N$). The latter does not seem to be stated before in the full generality
though for $\nu=1$ and $\nu=2$ this was proved before for certain classes of processes ( see, for instance,
\cite{CFKMZ} and references there). Of course, convergence in distribution does not provide by itself any rate
 estimates, and so it has mostly theoretical value.

Observe that the eventual almost sure estimates mentioned above do not yield any moment estimates valid for every $N$.
In this paper we show, in particular, that a Brownian motion $\cW$ can be chosen so that
\begin{equation}\label{1.3}
E\sup_{0\leq t\leq T}|\bbS_N^{(\nu)}(t)-\bbW^{(\nu)}_N(t)|^{4M/\nu}=O(N^{-\ve}),\,\nu=1,2,...,4M.
\end{equation}
for some $\ve>0$ which does not depend on $N$ and, again, this will be proved with respect to the variational and
not just for the supremum norm. It is easy to see that (\ref{1.3}) provides for each $N\geq 1$ also estimates in the
Prokhorov and Kantorovich--Rubinstein (or Wasserstein) metrics between the distributions of $\bbS_N^{(\nu)}$
 and $\bbW^{(\nu)}_N$. We stress that these estimates are valid for each $N\geq 1$ while the a.s.
approximation estimate (\ref{1.2}) is valid eventually or, in other words, the constant which can be put in place
of $O$ in (\ref{1.2}) is, in fact, a finite random variable with no hint on its integrability properties
while in (\ref{1.3}) we can replace $O$ by a true constant.
We observe that even for independent sequences $\xi(k)$ our results are new though, of course, the estimates are
easier then.

In the continuous time case we consider two setups. The first one is standard in probability when we impose
mixing and approximation conditions with respect to a family of $\sig$-algebras indexed by two continuous
time parameters on the same
probability space on which our continuous time process $\xi$ is defined. Since this setup does not have many
applications to processes generated by continuous time dynamical systems (flows) we consider also another setup,
called suspension, when mixing and approximation conditions are imposed on a discrete time process defined 
on a base probability space and the continuous time process $\xi$ moves deterministically for the time determined by certain ceiling function and then jumps to the base according to the above discrete time process. This construction
is adapted to applications for certain important classes of dynamical systems, i.e. when $\xi(n)=g\circ F^n$ or
 $\xi(t)=g\circ F^t$ where $F$ is a measure preserving transformation or a continuous time measure preserving flow. Unlike \cite{FK} and several other related papers we work under quite general dependence (mixing) and moment
 conditions and not under specific $\alpha$- or $\phi$-mixing assumptions.
 
We observe that the iterated sums and integrals $\Sig^{(\nu)}$ defined above have a visual resemblance with $U$-statistics having special (degenerate) 
 product type kernels. The kernels of $U$-statistics are supposed to be symmetric (see, for instance  \cite{DK}, \cite{DP}, \cite{KY} and references
 there) or $U$-statistics themselves are written in the symmetric form (see, for instance, \cite{DDP})
 while in our case for $x,y\in\bbR^d$ with $d>1$, in general, $h_{ij}(x,y)=x_iy_j\ne y_ix_j=h_{ij}(x,y)$. Only when $x_1,x_2,...,x_\nu$ are one dimensional
 $h(x_1,...,x_\nu)=x_1x_2\cdots x_\nu$ becomes a symmetric degenerate kernel of an $U$-statistics. Furthermore, in most papers $U$-statistics are one dimensional objects 
 and so are limit theorems for them while we obtain limit theorems for the whole tensor products defined above viewed as $d^\nu$-dimensional objects and such results
 cannot be obtained relying on limit theorems for each of their components separately. In \cite{Bor} weak limit theorems for Banach valued $U$-statistics were obtained
 while almost sure approximation results appear there without details of proofs and only for nondegenerate kernels where the limit is always a Gaussian process.
 % Moreover, the kernels of $U$-statistics are functions while
% when $d>1$ our tensor products in the language of $U$-statistics should be viewed as having $d^\nu$-dimensional vector function kernels with coordinate functions
 %being (typically) non symmetric kernels of usual $U$-statistics and limit theorems for such objects do not seem to be considered before. 
 Most of the papers on $U$-statistics having degenerated kernels
 deal with weak limit theorems for them (see \cite{BV}, \cite{LN} and references there) with representations of limit processes somewhat different from ours. For $U$-statistics
 with bivariate degenerate one dimensional kernels almost sure approximation type limit theorems were obtained in \cite{DDP} (with independent random variables) and in \cite{KY}
 (with weakly dependent stationary sequences). Our motivation and methods are different from papers on $U$-statistics and we obtain estimates in stronger variational 
 norms. In order to extend our estimates from sums and integrals to multiple iterated sums and integrals we rely on Chen relations (see Theorem \ref{hoelder} and
  Proposition \ref{hoelder2}) which hold true for iterated sums and integrals but not for general $U$-statistics. Consequently, the limit theorems we obtain in this 
  paper are of different types than in the works of $U$-statistics.

 The structure of this paper is the following. In the next section we provide necessary definitions and give
   precise statements of our results. Section \ref{sec3} is devoted to necessary estimates both of general nature
   and more specific to our problem, as well, as characteristic functions approximations needed
    for the strong approximation theorem which concludes that section. In Section \ref{sec4} we obtain
    main estimates for iterated sums with respect to the variational norms while Section \ref{sec5} is devoted to the
    "straightforward" continuous time setup. In Section 6 we deal with the suspension continuous time setup. In Section
    \ref{sec7} we extend the results to multiple iterated sums and integrals employing the rough paths machinery.

The author thanks Peter Friz for useful discussions of topics related to the present paper, especially, concerning the,
so called, Lyons' extension, i.e. extensions of estimates to multiple iterated sums and integrals..

\section{Preliminaries and main results}\label{sec2}\setcounter{equation}{0}
\subsection{Discrete time case}\label{subsec2.1}
We start with the discrete time setup which consists of a complete probability space
$(\Om,\cF,P)$, a stationary sequence of $d$-dimensional centered random vectors $\xi(n)=(\xi_1(n),...,\xi_d(n))$,
 $-\infty<n<\infty$ and
a two parameter family of countably generated $\sig$-algebras
$\cF_{m,n}\subset\cF,\,-\infty\leq m\leq n\leq\infty$ such that
$\cF_{mn}\subset\cF_{m'n'}\subset\cF$ if $m'\leq m\leq n
\leq n'$ where $\cF_{m\infty}=\cup_{n:\, n\geq m}\cF_{mn}$ and $\cF_{-\infty n}=\cup_{m:\, m\leq n}\cF_{mn}$.
It is often convenient to measure the dependence between two sub
$\sig$-algebras $\cG,\cH\subset\cF$ via the quantities
\begin{equation}\label{2.1}
\varpi_{b,a}(\cG,\cH)=\sup\{\| E(g|\cG)-Eg\|_a:\, g\,\,\mbox{is}\,\,
\cH-\mbox{measurable and}\,\,\| g\|_b\leq 1\},
\end{equation}
where the supremum is taken over real functions and $\|\cdot\|_c$ is the
$L^c(\Om,\cF,P)$-norm. Then more familiar $\al,\rho,\phi$ and $\psi$-mixing
(dependence) coefficients can be expressed via the formulas (see \cite{Bra},
Ch. 4 ),
\begin{eqnarray*}
&\al(\cG,\cH)=\frac 14\varpi_{\infty,1}(\cG,\cH),\,\,\rho(\cG,\cH)=\varpi_{2,2}
(\cG,\cH), \phi(\cG,\cH)=\frac 12\varpi_{\infty,\infty}(\cG,\cH)\,\,\,\mbox{and}\\
&\psi(\cG,\cH)=\varpi_{1,\infty}(\cG,\cH).
\end{eqnarray*}
We set also
\begin{equation}\label{2.2}
\varpi_{b,a}(n)=\sup_{k\geq 0}\varpi_{b,a}(\cF_{-\infty,k},\cF_{k+n,\infty})
\end{equation}
and accordingly
\[
\al(n)=\frac{1}{4}\varpi_{\infty,1}(n),\,\rho(n)=\varpi_{2,2}(n),\,
\phi(n)=\frac 12\varpi_{\infty,\infty}(n),\, \psi(n)=\varpi_{1,\infty}(n).
\]
Our setup includes also the  approximation rate
\begin{equation}\label{2.3}
\beta (a,l)=\sup_{k\geq 0}\|\xi(k)-E(\xi(k)|\cF_{k-l,k+l})\|_a,\,\, a\geq 1.
\end{equation}
We will assume that for some $ 1\leq L\leq\infty$, $M$ large enough and $K=\max(2L,8M)$,
\begin{equation}\label{2.4}
E|\xi(0)|^{K}<\infty,\,\,\,\mbox{and}\,\,\,\sum_{k=0}^\infty\sum_{l=k+1}^\infty(\sqrt{\sup_{m\geq l}\be(K,m)}+\varpi_{L,8M}(l))<\infty.
\end{equation}
Observe that
\[
k\sum_{l=k}^\infty\varpi_{L,8M}(l)\leq\sum_{m=1}^k\sum_{l=m}^\infty\varpi_{L,8M}(l)\leq C<\infty,
\]
where $C>0$ does not depend on $k$ by (\ref{2.4}), and so for any $k\geq 1$,
\begin{equation}\label{2.4+}
\varpi_{L,8M}(k)\leq\sum_{l=k}^\infty\varpi_{L,8M}(l)\leq Ck^{-1}.
\end{equation}

In order to formulate our results we have to introduce also the "increments" of multiple iterated sums under consideration
\begin{equation}\label{2.5}
\Sig^{(\nu)}(u,v)=\sum_{[u]\leq k_1<...<k_\nu< [v]}\xi(k_1)\otimes\cdots\otimes\xi(k_\nu),\, 0\leq u\leq v
\end{equation}
which means that $\Sig^{(\nu)}(u,v)=\{\Sig^{i_1,...,i_\nu}(u,v),\, 1\leq i_1,...,i_\nu\leq d\}$ where
in the coordinate-wise form
\[
\Sig^{i_1,...,i_\nu}(u,v)=\sum_{[u]\leq k_1<...<k_\nu< [v]}\xi_{i_1}(k_1)\cdots\xi_{i_\nu}(k_\nu).
\]
We set also for any $0\leq s\leq t\leq T$,
\[
\bbS_N^{(\nu)}(s,t)=N^{-\nu/2}\Sig^{(\nu)}(sN,tN),\,\,\bbS_N(s,t)=\bbS_N^{(2)}(s,t)\quad\mbox{and}\quad S_N(s,t)=\bbS_N^{(1)}(s,t).
\]
When $u=s=0$ we will just write
\[
\Sig^{(\nu)}(v)=\Sig^{(\nu)}(0,v),\,\bbS_N^{(\nu)}(t)=\bbS_N^{(\nu)}(0,t),\,\bbS_N(t)=\bbS_N(0,t)\,\,\mbox{and}\,\,
S_N(t)=S_N(0,t).
\]
%We set also
%\begin{eqnarray*}
%&\Sig(u,v)=\Sig^{(1)}(u,v)=\sum_{[u]\leq k<[v]}\xi(k)\,\,\mbox{and}\,\, S_N(s,t)=\bbS_N^{(1)}(s,t)=N^{-1/2}\Sig(sN,tN),\\
%&\Sig^{(1)}(v)=\sum_{0\leq k< [v]}\xi(k)\,\,\,\mbox{and}\,\,\,S_N(t)=\bbS_N^{(1)}(t)=N^{-1/2}\Sig^{(1)}(tN).
%\end{eqnarray*}

Next, introduce also the covariance matrix $\vs=(\vs_{ij})$ defined by
\begin{equation}\label{2.6}
\vs_{ij}=\lim_{k\to\infty}\frac 1k\sum_{m=0}^k\sum_{n=0}^k\vs_{ij}(n-m),\,\,
\mbox{where}\,\, \vs_{ij}(n-m)=E(\xi_i(m)\xi_j(n))
\end{equation}
taking into account that the limit here exist under conditions of our theorem below (see Section \ref{sec3}).
Let $\cW$ be a $d$-dimensional Brownian motion with the covariance matrix $\vs$ (at the time 1) and introduce
the rescaled Brownian motion $\bbW^{(1)}_N(t)=W_N(t)=N^{-1/2}\cW(Nt),\,  t\in[0,T],\, N\geq 1$. We set also
$\bbW_N^{(1)}(s,t)=W_N(s,t)=W_N(t)-W_N(s),\, t\geq s\geq 0$. Next, we introduce
\begin{equation}\label{rec1}
\bbW_N^{(2)}(s,t)=\bbW_N(s,t)=\int_s^tW_N(s,v)\otimes dW_N(v)+(t-s)\Gam
\end{equation}
which can be written in the coordinate-wise form as
\[
\bbW_N^{ij}(s,t)=\int_s^tW_N^i(s,v)dW^j_N(v)+(t-s)\Gam^{ij}\,\,\,\mbox{where}\,\,\,\Gam^{ij}=\sum_{l=1}^\infty
E(\xi_i(0)\xi_j(l))
\]
and the latter series converges absolutely as we will see in Section \ref{sec3}. Again, we set $\bbW_N^{(2)}(t)=
\bbW_N(t)=\bbW_N(0,t)$. For $n>2$ we define recursively,
\begin{equation}\label{rec2}
\bbW_N^{(n)}(s,t)=\int_s^t\bbW_N^{(n-1)}(s,v)\otimes dW_N(v)+\int_s^t\bbW_N^{(n-2)}(s,v)\otimes\Gam dv.
\end{equation}
Both here and the above the stochastic integrals are understood in the It\^ o sense. Coordinate-wise this relation
can be written in the form
\begin{eqnarray*}
&\bbW_N^{i_1,...,i_n}(s,t)=\int_s^t\bbW_N^{i_1,...,i_{n-1}}(s,v)dW^{i_n}_N(v)\\
&+\int_s^t\bbW_N^{i_1,...,i_{n-2}}(s,v)\Gam^{i_{n-1}i_n} dv.
\end{eqnarray*}
As before, we write also $\bbW_N^{(n)}(t)=\bbW_N^{(n)}(0,t)$.

We consider tensor products $(\bbR^d)^{\otimes\nu}=\bbR^d\otimes\cdots\otimes\bbR^d$ of $\bbR^d$ taken $\nu$ times which consists of $\nu d$-dimensional
vectors $X$ with one-dimensional coordinates $\{ X^{i_1,...,i_\nu},\, 1\leq i_k\leq d,\, k=1,2,...,\nu\}$. For each $X\in(\bbR^d)^{\otimes\nu}$ we define the norm
\begin{equation}\label{2.8+}
\| X\|=\sum_{1\leq i_1,...,i_\nu\leq d}|X^{i_1,...,i_\nu}|.
\end{equation}
Let  $\gam(t),\, t\in[0,T]$ be a path on $(\bbR^d)^{\otimes\nu}$, i.e. $\gam$ is a map $\gam:\,[0,T]\to(\bbR^d)^{\otimes\nu}$, and $p\geq 1$ be a number. The
 $p$-variation norm of $\gam$ on  an interval $[U,V]\subset [0, T],\, U<V$ is defined by
 \begin{equation}\label{2.7}
 \|\gam\|_{p,[U,V]}=\big(\sup_\cP\sum_{[s,t]\in\cP}\|\gam(s,t)\|^p\big)^{1/p}
 \end{equation}
 where the supremum is taken over all partitions $\cP=\{ U=t_0<t_1<...<t_n=V\}$ of $[U,V]$ and the sum is taken over
 the corresponding subintervals $[t_i,t_{i+1}],\, i=0,1,...,n-1$ of the partition while $\gam(s,t)$ is taken according
 to the definitions above depending on the process under consideration. Actually, the notation (\ref{2.7}) will be used here for any positive $p$.
  We will prove

%For multiple iterated stochastic integrals we also define the "increments" by
%\[
%\bbW^{(\nu)}(s,t)=\int_{s\leq u_1\leq...\leq u_\nu\leq t}dW(u_1)\otimes\cdots\otimes dW(u_\nu),
%\]
%where $W$ is a multivariate Brownian motion with covariances, and we write as above $\bbW^{(\nu)}(t)=\bbW^{(\nu)}(0,t)$.
%These notations affect the following definition of $p$-variation norms. For any path $\gam(t),\, t\in[0,T]$ in a Euclidean %space having left and right limits and $p\geq 1$ the $p$-variation norm of
%$\gam$ on  an interval $[U,V],\, U<V$ is given by
% \begin{equation}\label{2.6}
% \|\gam\|_{p,[U,V]}=\big(\sup_\cP\sum_{[s,t]\in\cP}|\gam(s,t)|^p\big)^{1/p}
% \end{equation}
 %where the supremum is taken over all partitions $\cP=\{ U=t_0<t_1<...<t_n=V\}$ of $[U,V]$ and the sum is taken over
% the corresponding subintervals $[t_i,t_{i+1}],\, i=0,1,...,n-1$ of the partition while $\gam(s,t)$ is taken according
% to the definitions above depending on the process under consideration. In order to formulate our first result we have to %introduce also the covariance matrix $\vs=(\vs_{ij})$ defined by
%\begin{equation}\label{2.7}
%\vs_{ij}=\lim_{k\to\infty}\frac 1k\sum_{m=0}^k\sum_{n=0}^k\vs_{ij}(n-m),\,\,
%\mbox{where}\,\, \vs_{ij}(n-m)=E(\xi_i(m)\xi_j(n))
%\end{equation}
%taking into account that the limit here exist under conditions of our theorem below (see Section \ref{sec3}).

\begin{theorem}\label{thm2.1}
Let the conditions of (\ref{2.4}) hold true with integers $L\geq 1$ and $M>\frac p{p-2}$ for some $p\in(2,3)$. Then the
stationary sequence of random vectors $\xi(n),\,-\infty<n<\infty$ can be redefined preserving its distributions on
a sufficiently rich probability space which contains  also a $d$-dimensional Brownian motion $\cW$ with the covariance
matrix $\vs$ (at the time 1) so that for any integer $N\geq 1$ the processes $\bbW_N^{(\nu)},\,
\nu=1,2,...$ constructed as above with the rescaled Brownian motion $W_N(t)=N^{-1/2}\cW(Nt),\,  t\in[0,T]$
satisfy,
 \begin{equation}\label{2.8}
 E\| \bbS^{(\nu)}_N-\bbW^{(\nu)}_N\|^{4M/\nu}_{p/\nu,[0,T]}\leq C(M)N^{-\ve_\nu},\,\,\nu=1,2,...,4M
 \end{equation}
 where the constants $\ve_\nu=\ve_\nu(M)>0$ and $C(M)>0$ do not depend on $N$.
 \end{theorem}

Note that if we replace $M$ in (\ref{2.8}) by any $\tilde M$ between 1 and $M$ then by the H\" older inequality (\ref{2.8})
will still remain true with the right hand side $(C(M)N^{-\ve}\nu)^{\tilde M/M}$, and so, in fact, it suffices to prove
(\ref{2.8}) only for all $M$ large enough. In order to understand our assumptions observe that $\varpi_{b,a}$
is clearly non-increasing in $b$ and non-decreasing in $a$. Hence, for any pair $a,b\geq 1$,
\[
\varpi_{b,a}(n)\leq\psi(n).
\]
Furthermore, by the real version of the Riesz--Thorin interpolation
theorem or the Riesz convexity theorem (see \cite{Ga}, Section 9.3
and \cite{DS}, Section VI.10.11) whenever $\theta\in[0,1],\, 1\leq
a_0,a_1,b_0,b_1\leq\infty$ and
\[
\frac 1a=\frac {1-\theta}{a_0}+\frac \theta{a_1},\,\,\frac 1b=\frac
{1-\theta}{b_0}+\frac \theta{b_1}
\]
then
\begin{equation}\label{2.9}
\varpi_{b,a}(n)\le 2(\varpi_{b_0,a_0}(n))^{1-\theta}
(\varpi_{b_1,a_1}(n))^\theta.
\end{equation}
In particular,  using the obvious bound $\varpi_{b_1,a_1}(n)\leq 2$
valid for any $b_1\geq a_1$ we obtain from (\ref{2.9}) for pairs
$(\infty,1)$, $(2,2)$ and $(\infty,\infty)$ that for all $b\geq a\geq 1$,
\begin{eqnarray}\label{2.10}
&\varpi_{b,a}(n)\le 4(2\alpha(n))^{\frac{1}{a}-\frac{1}{b}},\\
&\varpi_{b,a}(n)\le 2^{1+\frac 1a-\frac 1b}(\rho(n))^{1-\frac 1a+\frac 1b}\mbox{and}\,\,\varpi_{b,a}(n)\le 2^{1+\frac 1a}(\phi(n))^{1-\frac 1a}.
\nonumber\end{eqnarray}
We observe also that by the H\" older inequality for $b\geq a\geq 1$
and $\alpha\in(0,a/b)$,
\begin{equation}\label{2.11}
\beta(b,l)\le 2^{1-\alpha}  [\beta(a,l)]^\alpha \|\xi(0)\|^{1-\al}_{\frac{ab(1-\al)}
{a-b\al}}.
\end{equation}
 Thus, we can formulate assumption (\ref{2.4}) in terms of more familiar $\alpha,\,\rho,\,\phi,$
and $\psi$--mixing coefficients and with various moment conditions. It follows also from (\ref{2.9})
 that if $\varpi_{b,a}(n)\to 0$ as $n\to\infty$ for some $b>a\geq 1$ then
\begin{equation}\label{2.12}
\varpi_{b,a}(n)\to 0\,\,\mbox{as}\,\, n\to\infty\,\,\mbox{for all}\,\, b> a\geq 1,
\end{equation}
and so (\ref{2.12}) follows from (\ref{2.4}).

Observe that the estimate (\ref{2.8}) in the $p$-variation norm is stronger than an estimate just in the
supremum norm. In order to prove Theorem \ref{thm2.1} we will first derive directly (\ref{2.8}) for $\nu=1$
and $\nu=2$ relying, in particular, on the strong approximation theorem. Since the latter result did not seem
to appear before under our moment and mixing conditions we will provide the details which cannot be found in the
earlier literature. Observe that our mixing assumptions in (\ref{2.4}) together with the inequality (\ref{2.10})
 allow to obtain the strong approximation theorem under more general conditions than the ones appeared before.
Having (\ref{2.8}) for $\nu=1,2$ we will employ the machinery which comes from the rough paths theory and which allows
to extend (\ref{2.8}) directly from $\nu=1,2$ to $\nu\geq 3$.

Recall, that the Prokhorov distance $\pi$ between two probability measures $\mu$ and $\tilde\mu$ on a
metric space $\cX$ with a distance function $d$ is defined by
\begin{eqnarray*}
&\pi(\mu,\tilde\mu)=\inf\{\ka>0:\,\mu(U)\leq\tilde\mu(U^\ka)+\ka\,\,\mbox{and}\,\,\tilde\mu(U)\leq\mu(U^\ka)+\ka\\
&\mbox{for any Borel set $U$ on $\cX$}\}
\end{eqnarray*}
where $U^\ka=\{ x\in\cX:\, d(x,y)<\ka\,\,$ for some $y\in U\subset\cX\}$ is the $\ka$-neighborhood
of $U$.  Recall also that the $L^q$ Wasserstein (or Kantorovich--Rubinstein) distance between
 two probability measures $\mu$ and $\tilde\mu$ on $\cX$ is defined by
 \[
 w_q(\mu,\tilde\mu)=\inf\{ (Ed^q(Q,R))^{1/q}:\,\cL(Q)=\mu\,\,\mbox{and}\,\, \cL(R)=\tilde\mu\}
 \]
 where the infimum is taken over all random points (variables) $Q$ and $R$ in $\cX$ with their distributions
  $\cL(Q)$ and $\cL(R)$ equal $\mu$ and $\tilde\mu$ respectively. From Theorem \ref{thm2.1} we obtain immediately that
 \begin{equation}\label{2.13}
 w_{4M/\nu}(\cL(\bbS_N^{(\nu)}),\,\cL(\bbW_N^{(\nu)}))\leq C^{\nu/4M}(M)N^{-\frac {\nu\ve_\nu}{4M}}.
 \end{equation}
 It is known (see, for instance, Theorem 2 in \cite{GS}) that $(\pi(\mu,\tilde\mu))^2\leq w_1(\mu,\tilde\mu)\leq w_{4M/\nu}(\mu,\tilde\mu)$,
 and so we obtain also an estimate of the Prokhorov distance between the distributions in (\ref{2.13}) taking
 the square root in the right hand side of (\ref{2.13}). Taking into account the inequality
 \begin{eqnarray*}
 &|P\{ X\leq x\}-P\{ Y\leq x\}|\leq P\{|X-Y|>\del\}+P\{|Y-x|\leq\del\}\leq\del^{-2}E|X-Y|^2\\
 &+P\{|Y-x|\leq\del\},
 \end{eqnarray*}
 valid for any pair of random variables $X,Y$ and a number $x$, we obtain from (\ref{2.8}) also Berry--Esseen type estimates
  choosing $X=\bbS_N^{(\nu)}(T)$, $Y=\bbW_N^{(\nu)}(T)$ and $\del=N^{-\frac 13\ve_\nu}$.
   Now observe that $\cW$ and $W_n(t)=n^{-1/2}\cW(nt),\, t\in[0,T]$ have the same distribution since both processes
 are Brownian motions (with the same covariation matrix). It is not difficult to see relying on the inductive construction
 of $\bbW_1^{(\nu)}$ and $\bbW^{(\nu)}_n$ in (\ref{rec2}) that for each $\nu\geq 1$ the latter pair of processes have 
 the same distribution, as well (see Section 8 in \cite{Ki25}). It follows that (\ref{2.13}) remains true if we replace
  there $\bbW_N^{(\nu)}$ by $\bbW_1^{(\nu)}$. In this form (\ref{2.13}) becomes an extension of the central limit theorem
  from usual sums to iterated sums (and iterated integrals described below) with speed convergence estimates.

  Important classes of processes satisfying our conditions come from
dynamical systems. Let $F$ be a $C^2$ Axiom A diffeomorphism (in
particular, Anosov) in a neighborhood $\Om$ of an attractor or let $F$ be
an expanding $C^2$ endomorphism of a compact Riemannian manifold $\Om$ (see
\cite{Bow}), $g$ be either a H\" older continuous vector function or a
vector function which is constant on elements of a Markov partition and let $\xi(n)=
\xi(n,\om)=g(F^n\om)$. Here the probability space is $(\Om,\cB,P)$ where $P$ is a Gibbs
 invariant measure corresponding to some H\"older continuous function and $\cB$ is the Borel $\sig$-field.
  Let $\zeta$ be a finite Markov partition for $F$ then we can take $\cF_{kl}$
 to be the finite $\sig$-algebra generated by the partition $\cap_{i=k}^lF^i\zeta$.
 In fact, we can take here not only H\" older continuous $g$'s but also indicators
of sets from $\cF_{kl}$. The conditions of Theorems \ref{thm2.1} allow all such functions
since the dependence of H\" older continuous functions on $m$-tails, i.e. on events measurable
with respect to $\cF_{-\infty,-m}$ or $\cF_{m,\infty}$, decays exponentially fast in $m$ and
the condition (\ref{2.4}) is much weaker than that. A related class of dynamical systems
corresponds to $F$ being a topologically mixing subshift of finite type which means that $F$
is the left shift on a subspace $\Om$ of the space of one (or two) sided
sequences $\om=(\om_i,\, i\geq 0), \om_i=1,...,l_0$ such that $\om\in\Om$
if $\pi_{\om_i\om_{i+1}}=1$ for all $i\geq 0$ where $\Pi=(\pi_{ij})$
is an $l_0\times l_0$ matrix with $0$ and $1$ entries and such that $\Pi^n$
for some $n$ is a matrix with positive entries. Again, we have to take in this
case $g$ to be a H\" older continuous bounded function on the sequence space above,
 $P$ to be a Gibbs invariant measure corresponding to some H\" older continuous function and to define
$\cF_{kl}$ as the finite $\sig$-algebra generated by cylinder sets
with fixed coordinates having numbers from $k$ to $l$. The
exponentially fast $\psi$-mixing is well known in the above cases (see \cite{Bow}) and this property
is much stronger than what we assume in (\ref{2.4}). Among other
dynamical systems with exponentially fast $\psi$-mixing we can mention also the Gauss map
$Fx=\{1/x\}$ (where $\{\cdot\}$ denotes the fractional part) of the
unit interval with respect to the Gauss measure $G$ and more general transformations generated
by $f$-expansions (see \cite{Hei}). Gibbs-Markov maps which are known to be exponentially fast
$\phi$-mixing (see, for instance, \cite{MN}) can be also taken as $F$ in Theorem \ref{thm2.1}
with $\xi(n)=g\circ F^n$ as above.

\subsection{Straightforward continuous time setup}\label{subsec2.2}

Our direct continuous time setup consists of a Lebesgue integrable $d$-dimensional stationary process
$\xi(t),\, t\geq 0$ on a probability space $(\Om,\cF,P)$ satisfying (\ref{1.1}) and of a family of
$\sig$-algebras $\cF_{st}\subset\cF,\,-\infty\leq s\leq t\leq\infty$ such
that $\cF_{st}\subset\cF_{s't'}$ if $s'\leq s$ and $t'\geq t$. For all $t\geq 0$ we set
\begin{equation}\label{2.14}
\varpi_{b,a}(t)=\sup_{s\geq 0}\varpi_{b,a}(\cF_{-\infty,s},\cF_{s+t,\infty})
\end{equation}
and
\begin{equation}\label{2.15}
\beta (a,t)=\sup_{s\geq 0}\|\xi(s)-E(\xi(s)|\cF_{s-t,s+t})\|_a.
\end{equation}
where $\varpi_{b,a}(\cG,\cH)$ is defined by (\ref{2.1}).  We continue to
impose the assumption (\ref{2.4}) on the  decay rates  of
$\varpi_{b,a}(t)$ and $\beta (a,t)$. Although they only involve integer
values of $t$, it will  suffice since these are non-increasing functions of $t$.

Next, we introduce the covariance matrix $\vs=(\vs_{ij})$ defined by
\begin{equation}\label{2.16}
\vs_{ij}=\lim_{t\to\infty}\frac 1t\int_{0}^t\int_{0}^t\vs_{ij}(u-v)dudv,\,\,
\mbox{where}\,\, \vs_{ij}(u-v)=E(\xi_i(u)\xi_j(v))
\end{equation}
and the limit here exists under our conditions in the same way as in the discrete time setup.
In order to formulate our results we define the "increments" of multiple iterated integrals
\[
\Sig^{(\nu)}(u,v)=\int_{u\leq u_1\leq...\leq u_\nu\leq v}\xi(u_1)\otimes\cdots\otimes\xi(u_\nu)du_1\cdots du_\nu,\, u\leq v
\]
which coordinate-wise have the form
\[
\Sig^{(\nu)}(u,v)=\{\Sig^{i_1,...,i_\nu}(u,v),\, 1\leq i_1,...,i_\nu\leq d\}
\]
with
\[
\Sig^{i_1,...,i_\nu}(u,v)=\int_{u\leq u_1\leq...\leq u_\nu\leq v}\xi_{i_1}(u_1)\xi_{i_2}(u_2)\cdots\xi_{i_\nu}(u_\nu)du_1\cdots du_\nu
\]
and all integrals here and below are supposed to exist.
We use here the same letter $\Sig$ as in the discrete time case which should not lead to a confusion.
As in the discrete time case we set also
\[
\bbS_N^{(\nu)}(s,t)=N^{-\nu/2}\Sig^{(\nu)}(sN,tN),\,\,\bbS_N^{i_1,...,i_\nu}(s,t)=
N^{-\nu/2}\Sig^{i_1,...,i_\nu}(sN,tN),
\]
$\bbS_N(s,t)=\bbS_N^{(2)}(s,t)$ and $S_N(s,t)=\bbS_N^{(1)}(s,t)$.   When $u=s=0$ we will write
\[
\Sig^{(\nu)}(v)=\Sig^{(\nu)}(0,v)\quad\mbox{and}\quad \bbS_N^{(\nu)}(t)=\bbS_N^{(\nu)}(0,t).
\]
Next, we introduce the matrix
\[
\Gam=(\Gam^{ij}),\,\,\Gam^{ij}=\int_1^\infty du\int_0^1E(\xi_i(v)\xi_j(u))dv+\int_0^1du\int_0^uE(\xi_i(v)\xi_j(u))dv
\]
Then we have
\begin{theorem}\label{thm2.2}
Let the conditions of (\ref{2.4}) hold true with integers $L\geq 1$ and $M>\frac p{p-2}$ for some $p\in(2,3)$ where $\varpi$
and $\be$ are given by (\ref{2.14}) and (\ref{2.15}). Then the vector stationary process
  $\xi(t),\,-\infty<t<\infty$ can be redefined preserving its distributions on a sufficiently rich probability space which
   contains also a $d$-dimensional Brownian motion $\cW$ with the covariance matrix $\vs$ (at the time 1) so that for
  $\bbW_N^{(\nu)}$, constructed as in Theorem \ref{thm2.1} with  the rescaled Brownian motion
 $W_N(t)=N^{-1/2}\cW(Nt),\,  t\in[0,T]$ and the matrix $\Gam$ introduced above, and for any $N\geq 1$ the estimate (\ref{2.8})
 remains true for $\bbS_N^{(\nu)}$ with $\nu=1,2,...,4M$ defined above in the present continuous time setup.
 \end{theorem}

\subsection{Continuous time suspension setup}\label{subsec2.3}

Here we start with a complete probability space $(\Om,\cF,P)$, a
$P$-preserving invertible transformation $\vt:\,\Om\to\Om$ and
a two parameter family of countably generated $\sig$-algebras
$\cF_{m,n}\subset\cF,\,-\infty\leq m\leq n\leq\infty$ such that
$\cF_{mn}\subset\cF_{m'n'}\subset\cF$ if $m'\leq m\leq n
\leq n'$ where $\cF_{m\infty}=\cup_{n:\, n\geq m}\cF_{mn}$ and
$\cF_{-\infty n}=\cup_{m:\, m\leq n}\cF_{mn}$. The setup includes
also a (roof or ceiling) function $\tau:\,\Om\to (0,\infty)$ such that
for some $\hat L>0$,
\begin{equation}\label{2.17}
\hat L^{-1}\leq\tau\leq\hat L.
\end{equation}
Next, we consider the probability space $(\hat\Om,\hat\cF,\hat P)$ such that $\hat\Om=\{\hat\om=
(\om,t):\,\om\in\Om,\, 0\leq t\leq\tau(\om)\},\, (\om,\tau(\om))=(\vt\om,0)\}$, $\hat\cF$ is the
restriction to $\hat\Om$ of $\cF\times\cB_{[0,\hat L]}$, where $\cB_{[0,\hat L]}$ is the Borel
$\sig$-algebra on $[0,\hat L]$ completed by the Lebesgue zero sets, and for any $\Gam\in\hat\cF$,
\[
\hat P(\Gam)=\bar\tau^{-1}\int\bbI_\Gam(\om,t)dtdP(\om)\,\,\mbox{where}\,\,\bar\tau=\int\tau dP=E\tau
\]
and $E$ denotes the expectation on the space $(\Om,\cF,P)$.

Finally, we introduce a Lebesgue integrable vector valued stochastic process $\xi(t)=\xi(t,(\om,s))$, $-\infty<t<\infty,\, 0\leq
s\leq\tau(\om)$ on $\hat\Om$ satisfying
\begin{eqnarray*}
&\int\xi(t)d\hat P=0,\,\,\xi(t,(\om,s))=\xi(t+s,(\om,0))=\xi(0,(\om,t+s))\,\,\mbox{if}\,\, 0\leq t+s<\tau(\om)\\
&\mbox{and}\,\,\xi(t,(\om,s))=\xi(0,(\vt^k\om,u))\,\,\mbox{if}\,\, t+s=u+\sum_{j=0}^{k-1}\tau(\vt^j\om)\,\,\mbox{and}\\
&0\leq u<\tau(\vt^k\om).
\end{eqnarray*}
This construction is called in dynamical systems a suspension and it is a standard fact that $\xi$ is a
stationary process on the probability space $(\hat\Om,\hat\cF,\hat P)$. We will use this several times in Section \ref{sec6}.

Nevertheless, in what follows we will write $\xi(t)=\xi(t,\om)$ for $\xi(t,(\om,0))$ and will view it as a vector stochastic process on
the basic probability space $(\Om,\cF,P)$. When we will use stationarity of the process $\xi$ on the space $(\hat\Om,\hat\cF,\hat P)$ in Section \ref{sec6}, 
we will consider integrals with respect to the measure $\hat P$. Otherwise, all moment estimates here will be with respect to the expectation $E$ on the
 space $(\Om,\cF,P)$. In this setup we define
\[
\Sig^{(\nu)}(u,v)=\int_{u\bar\tau\leq u_1\leq...\leq u_\nu\leq v\bar\tau}\xi(u_1)\otimes\cdots\otimes\xi(u_\nu)du_1\cdots du_\nu,\, u\leq v,
\]
\[
\Sig^{i_1,...,i_\nu}(u,v)=\int_{u\bar\tau\leq u_1\leq...\leq u_\nu\leq v\bar\tau}\xi_{i_1}(u_1)\xi_{i_2}(u_2)\cdots\xi_{i_\nu}(u_\nu)du_1\cdots du_\nu,
\]
\[
\bbS_N^{(\nu)}(s,t)=N^{-\nu/2}\Sig^{(\nu)}(sN,tN),\,\,\,\bbS_N^{i_1,...,i_\nu}(s,t)=N^{-\nu/2}\Sig^{i_1,...,i_\nu}(sN,tN)
\]
and, again, $\bbS_N(s,t)=\bbS_N^{(2)}(s,t)$, $S_N(s,t)=\bbS_N^{(1)}(s,t)$,
\[
S_N(s,t)=\bbS_N^{(1)}(s,t),\,\,\Sig^{(\nu)}(v)=\Sig^{(\nu)}(0,v)\quad\mbox{and}\quad \bbS_N^{(\nu)}(t)=\bbS_N^{(\nu)}(0,t).
\]

Set $\eta(\om)=\int_0^{\tau(\om)}\xi(s,\om)ds$, $\eta(m)=\eta\circ\vt^m$ and
\begin{equation}\label{2.18}
\be(a,l)=\sup_m\max\big(\|\tau\circ\vt^m-E(\tau\circ\vt^m|\cF_{m-l,m+l})\|_a,\,\|\eta(m)-E(\eta(m)|\cF_{m-l,m+l})\|_a\big).
\end{equation}
We define $\varpi_{b,a}(n)$ by (\ref{2.2}) with respect to the $\sig$-algebras $\cF_{kl}$ appearing here.
Observe also that $\eta(k)=\eta\circ\vt^k$ is a stationary sequence of random vectors and we introduce also the covariance matrix
 \begin{equation}\label{2.19}
\vs_{ij}=\lim_{n\to\infty}\frac 1n\sum_{k,l=0}^nE(\eta_i(k)\eta_j(l))
\end{equation}
where the limit exists under our conditions in the same way as in (\ref{2.6}). We introduce also the matrix
\[
\Gam=(\Gam^{ij}),\,\,\Gam^{ij}=\sum_{l=1}^\infty E(\eta_i(0)\eta_j(l))+E\int_0^{\tau(\om)}\xi_j(s,\om)ds\int_0^s\xi_i(u,\om)du.
\]
The following is our limit theorem in the present setup.

\begin{theorem}\label{thm2.3}
Assume that $E\eta=E\int_0^\tau\xi(s)ds=0$ and $E\int_0^\tau |\xi(s)|^{K}ds<\infty$. The latter replaces the moment
condition on $\xi(0)$ in (\ref{2.4}) while other conditions there are supposed to hold true with integers $L\geq 1$ and
$M>\frac p{p-2}$ for some $p\in(2,3)$ where $\varpi$ is given by (\ref{2.2}) for $\sig$-algebras $\cF_{mn}$
appearing in this subsection and $\be$ is defined by (\ref{2.18}). Then the process $\xi(t),\,-\infty<t<\infty$  can be
redefined preserving its distributions on a sufficiently rich probability space which contains also a $d$-dimensional Brownian motion $\cW$ with the covariance matrix $\vs$ (at the time 1) so that for $\bbW_N^{(\nu)}$, constructed
as in Theorem \ref{thm2.1} with the rescaled Brownian motion $W_N(t)=N^{-1/2}\cW(Nt),\,  t\in[0,T]$ and the matrix $\Gam$
introduced above, and for any $N\geq 1$
 the estimate (\ref{2.8}) remains true for $\bbS_N^{(\nu)}$ with $\nu=1,...,4M$ defined above.
 \end{theorem}

\begin{remark}\label{rem2.4}
Theorems \ref{thm2.1}--\ref{thm2.3} can be extended with essentially the same proof to the setup where in place of one discrete
or continuous time process $\xi(k),\, k\in\bbZ$ or $\xi(t),\, t\in\bbR$ we have sequences of $\bbR^d$-valued jointly stationary
processes $\xi^{(i)}(k),\, k\in\bbZ$ or $\xi^{(i)}(t),\, t\in\bbR,\, i=1,2,...$ all satisfying the conditions of the corresponding 
Theorems \ref{thm2.1}--\ref{thm2.3} with respect to the same family of $\sig$-algebras $\cF_{mn}$ or $\cF_{st}$. In this setup we
will have iterated sums and integrals of the form
\[
\bbS_N^{(\nu)}(s,t)=N^{-\nu/2}\sum_{[sN]\leq k_1<...<k_\nu<[tN]}\xi^{(1)}(k_1)\otimes\xi^{(2)}(k_2)\otimes\cdots\otimes\xi^{(\nu)}(k_\nu)
\]
and
\[
\bbS_N^{(\nu)}(s,t)=N^{-\nu/2}\int_{sN\leq u_1\leq...\leq u_\nu<tN}\xi^{(1)}(u_1)\otimes\xi^{(2)}(u_2)\otimes\cdots\otimes\xi^{(\nu)}(u_\nu)du_1...du_\nu.
\]
The limiting processes $\bbW_N^{(\nu)}$ should be constructed now not by one Brownian motion $W_N$ having a covariance matrix $\vs$ but by
a sequence of Brownian motions $W_N^{(i)},\, i=1,2,...$ with covariance matrices $\vs^{(i)}$ so that $\bbW_N^{(1)}=W_N^{(1)}$,
\[
\bbW_N^{(2)}(s,t)=\int_s^tW_N^{(1)}(s,v)\otimes dW_N^{(2)}(v)+(t-s)\Gam^{(2)},
\]
where $\Gam^{(i)}=\sum_{l=1}^\infty E(\xi^{(i-1)}(0)\otimes\xi^{(i)}(l))$, and recursively,
\[
\bbW_N^{(\nu)}(s,t)=\int_s^tW_N^{(\nu-1)}(s,v)\otimes dW_N^{(\nu)}(v)+\int_s^t\bbW_N^{(\nu-2)}(s,v)\otimes\Gam^{(\nu)}dv.
\]
In particular, this will work when the processes $\xi^{(i)}$ are constructed by one appropriate discrete or continuous time
dynamical system $F^n$ or $F^t$ but with different functions $g^{(i)}$, i.e. $\xi^{(i)}(n)=g^{(i)}\circ F^n$ or $\xi^{(i)}(t)
=g^{(i)}\circ F^t$.
\end{remark}

\begin{remark}\label{rem2.5}
By the definition $W_N(t)=N^{-1/2}\cW(Nt)$ and changing time in the stochastic integral
\[
\bbW_N(t)=\int_0^tW_N(v)\otimes dW_N(v)+t\Gam=N^{-1}(\int_0^{Nt}\cW(u)\otimes d\cW(u)+Nt\Gam)=N^{-1}\bbW_1(Nt).
\]
Continuing by induction we have
\begin{eqnarray*}
&\bbW_N^{(n)}(t)=\int_0^t\bbW_N^{(n-1)}(v)\otimes dW_N(v)+\int_0^t\bbW_N^{(n-2)}(v)\otimes\Gam dv\\
&=N^{-n/2}(\int_0^{Nt}\bbW_1^{(n-1)}(u)\otimes d\cW(u)+\int_0^{Nt}\bbW_1^{(n-2)}(u)\otimes\Gam du)=N^{-n/2}\bbW_1^{(n)}(Nt).
\end{eqnarray*}
Hence, we can replace $\bbW_N^{(\nu)}$ in (\ref{2.8}) (as well as in Theorems \ref{thm2.2} and \ref{thm2.3}) by 
$\hat\bbW_N^{(\nu)}$ where $\hat\bbW_N^{(\nu)}(t)=N^{-n/2}\bbW_1^{(\nu)}(Nt)$ which could be a helpful clarification
since $\bbW_1^{(\nu)}$ does not depend on $N$ and the dependence on $N$ appears here only in normalization and time
stretch coefficients.
\end{remark}

\section{Auxiliary estimates}\label{sec3}\setcounter{equation}{0}
\subsection{General estimates}\label{subsec3.1}
We will need the following estimate which follows from Lemmas 3.4 and 3.7 from \cite{Ki20}.
\begin{lemma}\label{lem3.1} Let $\eta_1,\eta_2,...,\eta_N$ be random $d$-dimensional vectors and
$\cH_1\subset\cH_2\subset...\subset\cH_N$ be a filtration of $\sig$-algebras such that $\eta_m$ is
$\cH_m$-measurable for each $m=1,2,...,N$. Assume also that $E|\eta_m|^{B}<\infty$ for some integer $B\geq 2$
and each $m=1,...,N$. Set $S_m=\sum_{j=1}^m\eta_j$. Then
\begin{equation}\label{3.1}
E\max_{1\leq m\leq N}|S_m|^{B}\leq 2^{B-1}A_{B}^{B}(3(B)!d^{B/2}N^{B/2}+N)
\end{equation}
where $A_{B}=\sup_{i\geq 1}\sum_{j\geq i}\| E(\eta_j|\cH_i)\|_{B}$ .
\end{lemma}

In fact, the proof of Lemma 3.4 in \cite{Ki20} is based on Lemma 3.2.5 from \cite{HK} and the latter is formulated for
even itegers $B$. But looking at the details of the proof there it is easy to see that it is valid for all integers $B\geq 2$.
The following important result appears as Lemma 3.4 of \cite{FK} under the uniform boundedness and the $\phi$-mixing
assumptions. We will give its proof under our more general conditions which will follow a similar pattern but will
require additional H\" older inequality estimates.

\begin{lemma}\label{lem3.2}
Let $\eta(k),\,\zeta_k(l),\, k=0,1,...,\, l=0,...,k$ be sequences of random variables on a probability space
$(\Om,\cF,P)$ such that for all $k,l,n\geq 0$ and some $L,M\geq 1$ and $K=\max(2L,4M)$ satisfying (\ref{2.4}),
\begin{eqnarray}\label{3.2}
&\quad\|\eta(k)-E(\eta(k)|\cF_{k-n,k+n})\|_{K},\,\|\zeta_k(l)-E(\zeta_k(l)|\cF_{l-n,l+n})\|_{K}\\
&\leq\be(K,n),\,\,\|\eta(k)\|_{K},\,\|\zeta_k(l)\|_{K}\leq\gam_{K}<\infty\,\,\mbox{and}\,\,E\eta(k)=E\zeta_k(l)=0,\,\forall k,l\nonumber
\end{eqnarray}
where $\be(K,\cdot)$ and $\vp_{L,4M}(\cdot)$ satisfy (\ref{2.4}) and the $\sig$-algebras $\cF_{kl}$ are the same as in
Section \ref{sec2}. Then for any $N\geq 1$,
\begin{equation}\label{3.3}
E\max_{1\leq n\leq N}(\sum_{k=0}^n\eta(k))^{4M}\leq C^{\eta}(M)N^{2M}
\end{equation}
and
\begin{eqnarray}\label{3.4}
&E\max_{m\leq n\leq N}\big(\sum_{k=m}^n\sum_{j=\ell(k)}^{k-1}(\eta(k)\zeta_k(j)-E(\eta(k)\zeta_k(j)))\big)^{2M}\\
&\leq C^{\eta,\zeta}(M)(N-m)^{M}\max_{m\leq k\leq N}(k-\ell(k))^{M}\nonumber
\end{eqnarray}
where $0\leq\ell(k)<k$ is an integer valued function (may be constant) and $C^{\eta,\zeta}(M)>0$ depends only on
$\be,\gam,\vp$ and $M$ but it does not depend on $N,m,\ell$ or on the sequences $\eta$ and $\zeta$ themselves.
\end{lemma}
\begin{proof} Set
\[
\eta^{(i)}(k)=E(\eta(k)|\cF_{k-i,k+i})\,\,\,\mbox{and}\,\,\,\zeta_k^{(i)}(l)=E(\zeta_k(l)|\cF_{l-i,l+i}).
\]
Then
\[
\eta(k)=\lim_{i\to\infty}\eta^{(2^i)}(k)=\eta^{(1)}(k)+\sum_{i=1}^\infty(\eta^{(2^i)}(k)-\eta^{(2^{i-1})}(k))
\]
and
\[
\zeta_k(l)=\lim_{i\to\infty}\zeta_k^{(2^i)}(l)=\zeta_k^{(1)}(l)+\sum_{i=1}^\infty(\zeta_k^{(2^i)}(l)-\zeta_k^{(2^{i-1})}(l)).
\]
where convergence is in the $L^{K}$-sense since
\begin{equation}\label{3.5}
\|\eta^{(2^i)}(k)-\eta^{(2^{i-1})}(k)\|_{K},\,\|\zeta^{(2^i)}_k(l)-\zeta_k^{(2^{i-1})}(l)\|_{K}\leq 2(\be(K,2^i)+
\be(K,2^{i-1}))
\end{equation}
and (\ref{2.4}) holds true. For $i,j=0,1,2,...$ denote
 \begin{eqnarray*}
 &\vrho_{i,j}(k,l)=(\eta^{(2^i)}(k)-\eta^{(2^{i-1})}(k))(\zeta_k^{(2^j)}(l)-\zeta_k^{(2^{j-1})}(l))\\
 &-E\big((\eta^{(2^i)}(k)-\eta^{(2^{i-1})}(k))(\zeta_k^{(2^j)}(l)-\zeta_k^{(2^{j-1})}(l))\big)
 \end{eqnarray*}
 and $Q_{i,j}(k)=\sum_{l=\ell(k)}^{k-1}\vrho_{i,j}(k,l)$ where we set for convenience
 $\eta^{(2^{-1})}(k)=\zeta_k^{(2^{-1})}(l)=0$ for all $k,l\geq 0$.
 In view of these notations, in order to obtain (\ref{3.4}) we have to estimate $\|\max_{m\leq n\leq N}|\bbR(m,n)|\|_{2M}$
  where
 \[
 \bbR(m,n)=\sum_{i,j=0}^\infty\bbR_{i,j}(m,n)\,\,\mbox{and}\,\,\bbR_{i,j}(m,n)=\sum_{k=m}^nQ_{ij}(k).
 \]

Next, introduce $\cG_k^{i,j}=\cF_{-\infty,k+\max(2^i,2^j)}$ and observe that $Q_{i,j}(k)$ is $\cG_k^{i,j}$-measurable.
This enables us to apply Lemma \ref{lem3.1} to the sums $\bbR_{i,j}(m,n)$. First, write
\begin{equation}\label{3.6}
Q_{i,j}(k)=Q^{(1)}_{i,j,n}(k)+Q^{(2)}_{i,j,n}(k)
\end{equation}
where
\begin{eqnarray*}
&Q_{i,j,n}^{(1)}(k)=\sum_{\ell(k)\leq l<\frac {k+n}2+2\max(2^i,2^j),\, l<k}\vrho_{i,j}(k,l)\\
&=(\eta^{(2^i)}(k)-\eta^{(2^{i-1})}(k))\sum_{\ell(k)\leq l<\frac {k+n}2+2\max(2^i,2^j),l<k}
&(\zeta_k^{(2^j)}(l)-\zeta_k^{(2^{j-1})}(l))\\
&-E\big((\eta^{(2^i)}(k)-\eta^{(2^{i-1})}(k))\sum_{\ell(k)\leq l<\frac {k+n}2+2\max(2^i,2^j),l<k}
&(\zeta_k^{(2^j)}(l)-\zeta_k^{(2^{j-1})}(l))\big)\\
&\mbox{and}\,\,\, Q_{i,j,n}^{(2)}(k)=\sum_{\frac {k+n}2+2\max(2^i,2^j)\leq l<k}\vrho_{i,j}(k,l).
\end{eqnarray*}
%\[
% \mbox{and}\,\,\, Q_{i,j,n}^{(2)}(k)=\sum_{\frac {k+n}2+2\max(2^i,2^j)\leq l<k}\vrho_{i,j}(l,k).
% \]

 For $k-n\geq 8\max(2^i,2^j)$ we can write by the definition of $\vp_{L,4M}$ and the Cauchy--Schwarz inequality that
 \begin{eqnarray}\label{3.7}
 &\,\,\,\| E(Q_{i,j,n}^{(1)}(k)|\cG_n^{i,j}\|_{2M}\leq 2\| E\big(E(\eta^{(2^i)}(k)-\eta^{(2^{i-1})}(k)|\cF_{-\infty,
 \frac {k+n}2+3\max(2^i,2^j)})\\
 &\times\sum_{\ell(k)\leq l<\frac {k+n}2+2\max(2^i,2^j),\, l<k}(\zeta_k^{(2^j)}(l)-\zeta_k^{(2^{j-1})}(l))|\cG_n^{i,j}\big)\|_{2M}
 \nonumber\\
 &\leq 2\|E(\eta^{(2^i)}(k)-\eta^{(2^{i-1})}(k)|\cF_{-\infty,\frac {k+n}2+3\max(2^i,2^j)})\|_{4M}\nonumber\\
 &\times\|\sum_{\ell(k)\leq l<\frac {k+n}2+2\max(2^i,2^j),\, l<k}(\zeta_k^{(2^j)}(l)-\zeta_k^{(2^{j-1})}(l))\|_{4M}\nonumber\\
 &\leq 2\vp_{L,4M}(\frac {k-n}2-4\max(2^i,2^j))\|\eta^{(2^i)}(k)-\eta^{(2^{i-1})}(k)\|_{L}\nonumber\\
 &\times\|\sum_{\ell(k)\leq l<\frac {k+n}2+2\max(2^i,2^j), l<k}(\zeta_k^{(2^j)}(l)-\zeta_k^{(2^{j-1})}(l))\|_{4M}\nonumber
 \end{eqnarray}
 since $k-2^i-\frac {k+n}2-3\max(2^i,2^j)\geq\frac {k-n}2-4\max(2^i,2^j)$.
 If $0\leq k-n<8\max(2^i,2^j)$ then we obtain using (\ref{3.5}) and the Cauchy--Schwarz inequality that
 \begin{eqnarray}\label{3.8}
 &\| E(Q^{(1)}_{i,j,n}(k)|\cG_n^{i,j})\|_{2M}\leq 2\|\eta^{(2^i)}(k)-\eta^{(2^{i-1})}(k)\|_{4M}\\
 &\times\|\sum_{\ell(k)\leq l<\frac {k+n}2+2\max(2^i,2^j),\, l<k}(\zeta_k^{(2^j)}(l)-\zeta_k^{(2^{j-1})}(l))\|_{4M}\nonumber\\
 &\leq 4(\be(4M,2^i,)+\be(4M,2^{i-1}))\nonumber\\
 &\|\sum_{\ell(k)\leq l<\frac {k+n}2+2\max(2^i,2^j),l<k}(\zeta_k^{(2^j)}(l)-\zeta_k^{(2^{j-1})}(l))\|_{4M}.\nonumber
 \end{eqnarray}

 In order to bound the $L^{4M}$-norm of the last sum we will use Lemma \ref{lem3.1} setting $\cH_a=\cF_{-\infty,a+2^j}$.
 Taking into account the definition of $\vp$ we estimate for $l\geq a+2^{j+1}$,
 \[
 \| E(\zeta_k^{(2^j)}(l)-\zeta_k^{(2^{j-1})}(l)|\cH_a)\|_{4M}\leq 2\vp_{L,4M}(l-a-2^{j+1})(\be(L,2^j)+\be(L,2^{j-1})).
 \]
 For $a\leq l<a+2^{j+1}$ we just use the obvious estimate
 \[
 \| E(\zeta_k^{2^j)}(l)-\zeta_k^{(2^{j-1})}(l)|\cH_a)\|_{4M}\leq 2(\be(4M,2^j)+\be(4M,2^{j-1})).
 \]
 Hence,
 \begin{eqnarray*}
 &A^\zeta_{4M}=\sup_{a\geq 0}\sum_{l\geq a}\| E(\zeta_k^{2^j)}(l)-\zeta_k^{(2^{j-1})}(l)|\cH_a)\|_{4M}\\
 &\leq 2(\be(K,2^j)+\be(K,2^{j-1}))(\sum_{l=0}^\infty\vp_{L,4M}(l)+2^{j+1}),
 \end{eqnarray*}
 and so by Lemma \ref{lem3.1},
 \begin{eqnarray*}
 &\|\sum_{\ell(k)\leq l<\frac {k+n}2+2\max(2^i,2^j),\, l<k}(\zeta_k^{(2^j)}(l)-\zeta_k^{(2^{j-1})}(l))\|_{4M}\\
 &\leq(3(4M)!)^{1/4M}\sqrt dA^\zeta_{4M}\max_{m\leq k\leq n}\sqrt {k-l(k)}.
 \end{eqnarray*}
 This gives
 \begin{eqnarray}\label{3.9}
 &\|\sum_{k\geq n}|E(Q^{(1)}_{i,j,n}(k)|\cG_n^{i,j})\|_{2M}\leq 64(\be(K,2^i)\\
 &+\be(K,2^{i-1}))(\be(K,2^j)+\be(K,2^{j-1}))(3(4M)!)^{1/4M}\nonumber\\
 &\times\sqrt d\max(2^i,2^j)(1+\sum_{l=0}^\infty\vp_{L,4M}(l))^2\max_{m\leq s\leq n}\sqrt {s-l(s)}.\nonumber
 \end{eqnarray}

Next, similarly to the above, for $k-n\geq 8\max(2^i,2^j)$,
 \begin{eqnarray*}
 &\| E(Q^{(2)}_{i,j,n}(k)|\cG_n^{i,j})\|_{2M}\leq \sum_{\frac {k+n}2+2\max(2^i,2^j)\leq l<k}\\
 &\|  E(\vrho_{i,j}(l,k)|\cG_n^{i,j})\|_{2M}\\
 &\leq\vp_{L,2M}(\frac {k-n}2-4\max(2^i,2^j)) \sum_{\frac {k+n}2+2\max(2^i,2^j)\leq l<k}\|\vrho_{i,j}(l,k)\|_L\\
 &\leq\vp_{L,2M}(\frac {k-n}2-4\max(2^i,2^j))\|\eta^{(2^i)}(k)-\eta^{(2^{i-1})}(k)\|_{2L}\sum_{\frac {k+n}2+2\max(2^i,2^j)\leq l<k}\\
 &\|\zeta_k^{(2^j)}(l)-\zeta_k^{(2^{j-1})}(l))\|_{2L}\\
 &\leq 4(\be(K,2^i)+\be(K,2^{i-1}))(\be(K,2^j)\\
 &+\be(K,2^{j-1}))\vp_{L,2M}(\frac {k-n}2-4\max(2^i,2^j))|\frac {k-n}2-2\max(2^i,2^j)|
 \end{eqnarray*}
 since $E\vrho_{i,j}(l,k)=0$, each $\vrho_{i,j}(l,k)$ here is $\cF_{\frac {k+n}2-3\max(2^i,2^j)}$-measurable and $\frac {k+n}2-3\max(2^i,2^j)
 -n-\max(2^i,2^j)=\frac {k-n}2-4\max(2^i,2^j)$.
 For $k-n< 8\max(2^i,2^j)$ we use the simpler estimate
 \begin{eqnarray*}
 &\| E(Q^{(2)}_{i,j,n}(k)|\cG_n^{i,j})\|_{2M}\leq\sum_{\frac {k+n}2+2\max(2^i,2^j)\leq l<k}\|\vrho_{i,j}(l,k)\|_{2M}\\
 &\leq\|\eta^{(2^i)}(k)-\eta^{(2^{i-1})}(k)\|_{4M}\sum_{\frac {k+n}2+2\max(2^i,2^j)\leq l<k}\nonumber\\
 &\|\zeta_k^{(2^j)}(l)-\zeta_k^{(2^{j-1})}(l))\|_{4M}\\
 &\leq 48\max(2^i,2^j)(\be(K,2^i)+\be(K,2^{i-1}))(\be(K,2^j)+\be(K,2^{j-1}))
 \end{eqnarray*}
 since the sum here contains less than $\frac {k-n}2-2\max(2^i,2^j)\leq 6\max(2^i,2^j)$ terms. It follows that
 \begin{eqnarray}\label{3.10}
 &\|\sum_{k\geq n}|E(Q^{(2)}_{i,j,n}(k)|\cG_n^{i,j})\|_{2M}\\
 &\leq 8\max(2^i,2^j)(\be(K,2^i)+\be(K,2^{i-1}))(\be(K,2^j)\nonumber\\
 &+\be(K,2^{j-1}))\big(\sum_{l\geq 0}l\vp_{L,2M}(l)+24\max(2^i,2^j)\big).
 \nonumber\end{eqnarray}

Hence, by Lemma \ref{lem3.1},
\begin{equation}\label{3.11}
\|\max_{m\leq k\leq N}|\sum_{k=m}^nQ_{i,j}(k)|\|_{2M}\leq 2A_{2M}^{i,j}\big(3(2M)!d^{M}(N-m)^{M}+N-m\big)^{1/2M}
\end{equation}
where by the above
\begin{eqnarray*}
&A^{i,j}_{2M}=\sup_{0\leq n\leq N}\sum_{k\geq n}\| E(Q_{i,j}(k)|\cG_n^{i,j})\|_{2M}\\
&\leq 16\max(2^i,2^j)(\be(K,2^i)+\be(K,2^{i-1}))(\be(K,2^j)+\be(K,2^{j-1}))\\
&\times\big(4(3(4M)!)^{1/4M}\sqrt d(1+\sum_{l=0}^\infty\vp_{L,4M}(l))^2\max_{m\leq s\leq n}\sqrt {s-\ell(s)}\\
&+\sum_{l=0}^\infty l\vp_{L,4M}(l)+24\max(2^i,2^j)\big)
\end{eqnarray*}
and we observe that by (\ref{2.4}),
\begin{equation*}
\infty>\sum_{k=0}^\infty\sum_{l=k+1}^\infty\vp_{L,4M}(l)\geq\sum_{k=0}^\infty k\vp_{L,4M}(2k)\geq\frac 16\sum_{k=0}^\infty k\vp_{L,4M}(k)
\end{equation*}
and
\begin{equation*}
\infty>\sum_{k=0}^\infty\sum_{l=k+1}^\infty\be(K,l)\geq\sum_{k=0}^\infty k\be(K,2k)\geq\frac 16\sum_{k=0}^\infty k\be(K,k).
\end{equation*}
These together with
\[
\|\max_{m\leq n\leq N}|\bbR(m,n)|\|_{2M}\leq\sum_{i,j=0}^\infty\|\max_{m\leq n\leq N}|\sum_{k=m}^nQ_{i,j}(k)|\|_{2M}
\]
yield (\ref{3.4}) while (\ref{3.3}) is obtained by simplifying the proof above just by disregarding the sequence $\zeta_k(l)$
and taking $2M$ in place of $M$ in the proof which allows to estimate $L^{4M}$-norm of the sum of $\eta(k)$'s under our conditions.
\end{proof}

\subsection{Kolmogorov--Chentsov type theorem for multiplicative functionals}\label{subsec3.1+}
In the estimates for variational norms we will use the following extended version of the Kolmogorov theorem on the
 H\" older continuity of sample paths.

\begin{theorem}\label{hoelder}
For $1\leq\ell<\infty$ let $\bbX^{(\nu)}=\bbX^{(\nu)}(s,t),\,\nu=1,2,...,\ell;\, s\leq t$ be a two parameter
stochastic process which is a multiplicative functional in the sense of \cite{Lyo} and \cite{LQ}, i.e.
$\bbX^{(1)}(s,t)=\bbX^{(1)}(0,t)-\bbX^{(1)}(0,s)$ and for any $\nu=2,...,\ell$ and $0\leq s\leq u\leq t\leq T$,
\begin{equation}\label{ho1}
\bbX^{(\nu)}(s,t)=\bbX^{(\nu)}(s,u)+\sum_{k=1}^{\nu-1}\bbX^{(k)}(s,u)\otimes\bbX^{(\nu-k)}(u,t)+\bbX^{(\nu)}(u,t)
\end{equation}
where $\bbX^{(1)}(s,t)$ is a $d$-dimensional vector (or a vector in a Banach space) and $\otimes$ is a tensor
product in the corresponding space. Let $\bfM\geq\ell$ be an integer, $\be>1/\bfM$ and assume that whenever $0\leq s\leq t
\leq T<\infty$ and $1\leq\nu\leq\ell$,
\begin{equation}\label{ho2}
E\|\bbX^{(\nu)}(s,t)\|^\bfM\leq\bbC_{\ell,\bfM}|t-s|^{\nu\bfM\be}
\end{equation}
for some constants $\bbC_{\ell,\bfM}<\infty$, where $\|\cdot\|$ is the norm in a corresponding space satisfying
$\| a\otimes b\|\leq \| a\|\| b\|$. Then for all $\al\in[0,\be-\frac 1\bfM)$ there exists a modification of
$(\bbX^{(\nu)},\,\nu=1,...,\ell)$, which will have the same notation, and random variables $\bbK_\al^{(\nu)},
\,\nu=1,...,\ell$ such that
\begin{eqnarray}\label{ho3}
&E(\bbK_\al^{(\nu)})^{\bfM/\nu}\leq\bbC_{\ell,\bfM}C_{\al,\be,\bfM}<\infty\\
&\mbox{and}\quad\|\bbX^{(\nu)}(s,t)\|\leq\bbK_\al^{(\nu)}|t-s|^{\nu\al}
\quad\mbox{a.s.}\nonumber
\end{eqnarray}
where $C_{\al.\be,\bfM}>0$ depends only on $\al,\be$ and $\bfM$.
\end{theorem}
\begin{proof}
For $\ell=2$ this assertion was proved as Theorem 3.1 in \cite{FH}. We will proceed by a similar (actually, standard)
method except for the last step. Without loss of generality take $T=1$ (returning back to the original $T$ by time rescaling). Let $D_n$ be the set of integer multiples of $2^{-n}$ in $[0,1)$ and $D=\cup_{n\geq 1}D_n$. Set
$\bbL_n^{(\nu)}=\max_{t\in D_n}\|\bbX^{(\nu)}(t,t+2^{-n})\|$. By (\ref{ho2}),
\begin{equation}\label{ho4}
E(\bbL_n^{(\nu)})^\bfM\leq\sum_{t\in D_n}E\|\bbX^{(\nu)}(t,t+2^{-n})\|^\bfM\leq\bbC_{\ell,\bfM}2^{-n(\nu\bfM\be-1)}.
\end{equation}
Fix $s<t$ in $D$ and choose $m$ so that $2^{-(m+1)}<t-s\leq 2^{-m}$. The interval $[s,t)$ can be represented
as the finite disjoint union of intervals of the form $[\tau_l,\tau_{l+1})$ with $\tau_l,\tau_{l+1}\in D_n$ and $\tau_{l+1}-\tau_{l}=2^{-n} $
 for $n\geq m+1$ where no three intervals have the same length and $s=\tau_0<\tau_1<...<\tau_R=t$. Observe that by (\ref{ho1}) for any $r=1,2,...,R$,
\begin{eqnarray}\label{ho5}
&\bbX^{(\nu)}(s,\tau_r)=\sum_{q=0}^{r-1}\bbX^{(\nu)}(\tau_q,\tau_{q+1})+\sum_{q=1}^{r-1}\sum_{k=1}^{\nu-1}\\
&\bbX^{(k)}(s,\tau_q)\otimes\bbX^{(\nu-k)}(\tau_q,\tau_{q+1}).\nonumber
\end{eqnarray}
Set $A^{(k)}(s,t)=\max_{1\leq r\leq R}\|\bbX^{(k)}(s,\tau_r)\|$ and $B^{(k)}_m=\sum_{q\geq m+1}\bbL_q^{(k)}$.
Then by (\ref{ho5}),
\[
A^{(\nu)}(s,t)\leq 2B^{(\nu)}_m+2\sum_{k=1}^{\nu-1}A^{(k)}(s,t)B_m^{(\nu-k)}.
\]
Set
\[
\bfA^{(k)}(s,t)=\frac {A^{(k)}(s,t)}{|t-s|^{k\al}}\,\,\mbox{and}\,\, \bfB_m=\sup_{1\leq k\leq \ell}(2^{k\al(m+1)}B^{(k)}_m).
\]
Then dividing the above inequality by $|t-s|^{\nu\al}$ we obtain
\[
\bfA^{(\nu)}(s,t)\leq 2\bfB_m+2\bfB_m\sum_{k=1}^{\nu-1}\bfA^{(k)}(s,t).
\]
It follows by the discrete Gronwall inequality (see \cite{Cla}) that
\[
\bfA^{(\nu)}(s,t)\leq 2\bfB_m(1+2\bfB_m)^{\nu-1},
\]
and so
\[
\|\bbX^{(\nu)}(s,t)\|\leq\bbK_\al^{(\nu)}|t-s|^{\nu\al}\,\,\mbox{where}\,\,\bbK_\al^{(\nu)}=2\bfB_m(1+2\bfB_m)^{\nu-1}.
\]
Then
\[
E(\bbK_\al^{(\nu)})^{\frac \bfM{\nu}}\leq 2^{\bfM}(E\bfB_m^{\bfM/\nu}+2^\bfM E\bfB_m^\bfM)<\infty
\]
and by (\ref{ho4}),
\begin{eqnarray*}
&\|\bfB_m\|_{L^\bfM}\leq\sum_{1\leq k<\ell+1}2^{k\al(m+1)}\| B^{(k)}_m\|_{L^\bfM}
\leq\sum_{1\leq k\leq \ell}2^{k\al(m+1)}\sum_{q\geq m+1}\|\bbL_q^{(k)}\|_{L^\bfM}\\
&\leq \bbC^{1/\bfM}_{\ell ,\bfM}\sum_{1\leq k\leq\ell} 2^{-(m+1)(k(\be-\al)-\frac 1\bfM)}(1-2^{-(k\be-\frac 1\bfM)})^{-1}
<\infty.
\end{eqnarray*}
This gives (\ref{ho3}) for diadic $s,t\in[0,1]$. The extension to all real $s,t\in[0,1]$ follows by continuity since by (\ref{ho1}) and (\ref{ho3}) for any diadic sequences $t_n\to t$ and $s_n\to s$, $s,t,s_n,t_n\in[0,1]$ the sequence
$\bbX^{(\nu)}(s_n,t_n)$ is the Cauchy one. The limit of $\bbX^{(\nu)}(s_n,t_n)$ coincides with $\bbX^{(\nu)}(s,t)$
almost surely since $\bbX^{(\nu)}(s_n,t_n)\to\bbX^{(\nu)}(s,t)$ in probability in view of (\ref{ho1}) and (\ref{ho2}),
and so the modification comes into the picture.
\end{proof}

In the discrete time case we will use the following corollary from the above theorem (cf. Proposition 3.9 in \cite{CFKMZ}).
\begin{proposition}\label{hoelder2}
For each integer $N\geq 1$ let $(X_N,\bbX_N)=(X_N(s,t),\,\bbX_N(s,t))$ be a pair of two parameter processes
in $\bbR^d$ and $\bbR^d\otimes\bbR^d$, respectively, defined for each $s=k/N$ and $t=l/N$ where $0\leq k\leq l
\leq TN$, $T>0$ and $X_N(s,s)=\bbX_N(s,s)=0$. Assume that the Chen relation
\begin{equation}\label{hoe1}
\bbX_N(\frac kN,\frac mN)=\bbX_N(\frac kN,\frac lN)+X_N(\frac kN,\frac lN)\otimes X_N(\frac lN,\frac mN)
+\bbX_N(\frac lN,\frac mN)
\end{equation}
and $X_N(\frac kN,\frac lN)=X_N(0,\frac lN)-X_N(0,\frac kN)$ holds true for any $0\leq k\leq l\leq m\leq TN$.
Suppose that for all $0\leq k\leq l\leq NT$ and some $\bfM\geq 1$, $\ka>\frac 1{\bfM}$, $C(\bfM)>0$ and
$\bbC(\bfM)>0$ (which do not depend on $k,l$ and $N$),
\begin{eqnarray}\label{hoe2}
&E|X_N(\frac kN,\frac lN)|^{\bfM}\leq C(\bfM)|\frac {l-k}N|^{\bfM\ka}\,\,\mbox{and}\\
&E|\bbX_N(\frac kN,\frac lN)|^{\bfM}\leq \bbC(\bfM)|\frac {l-k}N|^{2\bfM\ka}.\nonumber
\end{eqnarray}
Then for any $\be\in(\frac 1{\bfM},\ka)$ there exist random variables $C_{\ka,\be,N}>0$ and
$\bbC_{\ka,\be,N}>0$ such that
\begin{eqnarray}\label{hoe3}
&|X_N(\frac kN,\frac lN)|\leq C_{\ka,\be,N}|\frac {l-k}N|^{\ka-\be}\,\,\mbox{and}\\
&|\bbX_N(\frac kN,\frac lN)|\leq \bbC_{\ka,\be,N}|\frac {l-k}N|^{2(\ka-\be)}\,\,\mbox{a.s.}\nonumber
\end{eqnarray}
and
\begin{eqnarray}\label{hoe4}
EC^{\bfM}_{\ka,\be,N}\leq K_{\ka,\be}(\bfM)<\infty\,\,\,\mbox{and}\,\,\, E\bbC^{\bfM/2}_{\ka,\be,N}\leq\bbK_{\ka,\be}(\bfM)<\infty
\end{eqnarray}
where $K_{\ka,\be}(\bfM)>0$ and $\bbK_{\ka,\be}(\bfM)>0$ do not depend on $N$.
\end{proposition}
\begin{proof} First, we will extend the definition of $X_N$ and $\bbX_N$ to the whole interval $[0,T]$
preserving the Chen relations similarly to Proposition 6.17 from \cite{FZ} and then will use Theorem \ref{hoelder}.
Set $\tilde X_N(t)=X_N(0,k/N)$ and $\tilde\bbX_N(t)=\bbX_N(0,k/N)$ if $t=k/N$, $k=0,1,...,N$ and for any
$t\in[\frac kN,\frac {k+1}N)$ we extend the definition by the linear interpolation
\begin{eqnarray*}
&\tilde X_N(t)=(k+1-tN)X_N(0,\frac kN)+(tN-k)X_N(0,\frac {k+1}N)\,\,\mbox{and}\\
&\tilde\bbX_N(t)=(k+1-tN)\bbX_N(0,\frac kN)+(tN-k)\bbX_N(0,\frac {k+1}N).
\end{eqnarray*}
Now set $\tilde X_N(s,t)=\tilde X_N(t)-\tilde X_N(s)$ and
\[
\tilde\bbX_N(s,t)=\tilde\bbX_N(t)-\tilde\bbX_N(s)-\tilde X_N(s)\otimes\tilde X_N(s,t).
\]
By this definition we check directly that the pair $(\tilde X_N,\tilde\bbX_N)$ satisfies the Chen relation
\begin{equation}\label{hoe5}
\tilde\bbX_N(s,t)=\tilde\bbX_N(s,u)+\tilde X_N(s,u)\otimes\tilde X_N(u,t)+\tilde\bbX_N(u,t),\quad s\leq u\leq t.
\end{equation}

Next, in the same way as in Proposition 6.17 in \cite{FZ}, we see from (\ref{hoe1}) and (\ref{hoe5}) that
\begin{eqnarray}\label{hoe6}
&E|\tilde X_N(s,t)|^{\bfM}\leq\tilde C(\bfM)|t-s|^{\bfM\ka}\,\,\mbox{and}\,\,
&E|\tilde\bbX_N(s,t)|^{\bfM}\leq\tilde\bbC(\bfM)|t-s|^{2\bfM\ka}\nonumber
\end{eqnarray}
where $\tilde C(\bfM)>0$ and $\tilde\bbC(\bfM)>0$ depend only on $C(\bfM),\,\bbC(\bfM), \bfM$ and $\ka$. Now we can use
Theorem \ref{hoelder} to conclude that there exist modifications of $\tilde X_N$ and $\tilde\bbX_N$ denoted
by the same letter such that
\begin{eqnarray}\label{hoe7}
&E\sup_{0\leq s<t\leq T}\frac {|\tilde X_N(s,t)|^{\bfM}}{|t-s|^{\bfM(\ka-\be)}}\leq K_{\ka,\be}(\bfM)\,\,\mbox{and}\\
&E\sup_{0\leq s<t\leq T}\frac {|\tilde\bbX_N(s,t)|^{\bfM}}{|t-s|^{2\bfM(\ka-\be)}}\leq\bbK_{\ka,\be}(\bfM)\nonumber
\end{eqnarray}
for some numbers $K_{\ka,\be}(\bfM)>0$ and $\bbK_{\ka,\be}(\bfM)>0$ which do not depend on $N$. Since a.s.
$\tilde X_N(k/N,\, l/N)=X_N(k/N,\, l/N)$ and $\tilde\bbX_N(k/N,\, l/N)=\bbX_N(k/N,\, l/N)$ for all
$0\leq k\leq l\leq N$ we conclude that (\ref{hoe3}) and (\ref{hoe4}) hold true, completing the proof.
\end{proof}
We observe that this result can be extended to higher order piecewise constant rough paths making them continuous
by linear interpolation and using Theorem \ref{hoelder} but we will not need this extension here. Note though that
linear interpolations lead to processes with infinite $q$-variational norms for $q<1$, and so such extensions will
not be appropriate for our goals in Section \ref{sec7}.

\subsection{The matrix $\vs$ and characteristic functions estimates}\label{subsec3.2}

In the following lemma we will justify the definition of the matrix $\vs$ and obtain the necessary convergence
estimates.
\begin{lemma}\label{lem3.3}
For each $i,j=1,...,d$ the limit
\begin{equation}\label{3.12}
\vs_{ij}=\lim_{n\to\infty}\frac 1n\sum_{k=0}^n\sum_{m=0}^nE(\xi_i(k)\xi_j(m))
\end{equation}
exists and for any $m,n\geq 0$,
\begin{eqnarray}\label{3.13}
&|n\vs_{ij}-\sum_{k=m}^{m+n}\sum_{l=m}^{m+n}E(\xi_i(k)\xi_j(l))|\\
&\leq 6\sum_{k=0}^n\sum_{l=k+1}^\infty\big((\|\xi_i(0)\|_K+\|\xi_j(0)\|_K)\be(K,l)+(\|\xi_i(0)\|_K\|\xi_j(0)\|_K)\vp_{L,4M}(l)\big).\nonumber
\end{eqnarray}
Similarly, the limit
\[
\Gam^{ij}=\lim_{n\to\infty}E\bbS_N(n)=\lim_{n\to\infty}\frac 1n\sum_{l=1}^n\sum_{m=0}^{l-1}E(\xi_i(m)\xi_j(l))=
\sum_{l=1}^\infty E((\xi_i(0)\xi_j(l))
\]
exists and
\begin{eqnarray*}
&|n\Gam^{ij}-\sum_{l=1}^{n}\sum_{m=0}^{l-1}E(\xi_i(m)\xi_j(l))|\\
&\leq 3\sum_{k=0}^n\sum_{l=k+1}^\infty\big((\|\xi_i(0)\|_K+\|\xi_j(0)\|_K)\be(K,l)+(\|\xi_i(0)\|_K\|\xi_j(0)\|_K)
\vp_{L,4M}(l)\big).\nonumber
\end{eqnarray*}
\end{lemma}
\begin{proof}
By the Cauchy--Schwarz inequality and the definition of the coefficients $\be$ and $\vp$ we have for $k>l$,
\begin{eqnarray}\label{3.14}
&|E(\xi_i(k)\xi_j(l))|\leq |E\big((\xi_i(k)-E(\xi_i(k)|\cF_{k-\frac 13|k-l|,k+\frac 13|k-l|}))\xi_j(l)\big)|\\
&+|E\big(E(\xi_i(k)|\cF_{k-\frac 13|k-l|,k+\frac 13|k-l|})(\xi_j(l)-E(\xi_j(l)|\cF_{l-\frac 13|k-l|,l+\frac 13|k-l|})\big)|
\nonumber\\
&+|E\big(E(\xi_i(k)|\cF_{k-\frac 13|k-l|,k+\frac 13|k-l|})E(\xi_j(l)|\cF_{l-\frac 13|k-l|,l+\frac 13|k-l|})\big)|\nonumber\\
&\leq\be(2M,\frac 13|k-l|)\|\xi_j(l)\|_{2M}+\|\xi_i(k)\|_{2M}\be(2M,\frac 13|k-l|)\nonumber\\
&+|E\big(E(\xi_i(k)|\cF_{k-\frac 13|k-l|,k+\frac 13|k-l|})|\cF_{-\infty,l+\frac 13|k-l|})E(\xi_j(l)|\cF_{l-\frac 13|k-l|,l+\frac 13|k-l|})\big)|\nonumber\\
&\leq \be(K,\frac 13|k-l|)\|\xi_j(0)\|_{K}+\|\xi_i(0)\|_{K}\be(K,\frac 13|k-l|)\nonumber\\
&+\vp_{L,4M}(\frac 13|k-l|)\|\xi_i(0)\|_K\|\xi_j(0)\|_{K}.
\nonumber\end{eqnarray}
Now by the stationarity,
\[
\sum_{k=m}^{m+n}\sum_{l=m}^{m+n}E(\xi_i(k)\xi_j(l))=(n+1)E(\xi_i(0)\xi_j(0))+\sum_{k=1}^n\sum_{m=1}^kE(\xi_i(m)\xi_j(0)+\xi_i(0)\xi_j(m)).
\]
It follows by (\ref{2.4}), (\ref{2.7}) and (\ref{3.14}) that the limit (\ref{3.12}) exists and
\[
\vs_{ij}=E(\xi_i(0)\xi_j(0))+\sum_{k=1}^\infty E(\xi_i(k)\xi_j(0)+\xi_i(0)\xi_j(k))
\]
and the latter series converges. Moreover,
\begin{eqnarray*}
&|n\vs_{ij}-\sum_{k=m}^{m+n}\sum_{l=m}^{m+n}E(\xi_i(k)\xi_j(l))|\\
&\leq 6(\|\xi_i(0)\|_{K}+\|\xi_j(0)\|_{K})\sum_{k=0}^n\sum_{l=k+1}^\infty\be(K,l)\\
&+6\|\xi_i(0)\|_K\|\xi_j(0)\|_K\sum_{k=0}^n\sum_{l=k+1}^\infty\vp_{L,4M}(l)
\end{eqnarray*}
and (\ref{3.13}) follows.
\end{proof}

Next, for each $n\geq 1$ introduce the characteristic function
 \[
 f_n(w)=E\exp(i\langle w,\, n^{-1/2}\sum_{k=0}^{n-1}\xi(k)\rangle),\, w\in\bbR^d
 \]
 where $\langle\cdot,\cdot\rangle$ denotes the inner product. We will need the following
 estimate.
 \begin{lemma}\label{lem3.4}
 For any $n\geq 1$,
 \begin{equation}\label{3.15}
 |f_n(w)-\exp(-\frac 12\langle\vs w,\, w\rangle)|\leq C_fn^{-\wp}
 \end{equation}
 for all $w\in\bbR^d$ with $|w|\leq n^{\wp/2}$ where the matrix $\vs$ is given
 by (\ref{2.6}) and we can take $\wp\leq\frac 1{20}$
 and a constant $C_f>0$ does not depend on $n$.
 \end{lemma}
 \begin{proof}
 First, observe that (\ref{2.4}) implies that for any $n\geq 1$,
 \begin{equation}\label{3.16}
 \be(K,n)\leq C_1n^{-4}\quad\mbox{and}\quad\vp_{L,4M}(n)\leq\vp_{L,8M}(n)\leq C_1n^{-2}
 \end{equation}
 where $C_1>0$ does not depend on $n$. Indeed, $\vp_{L,8M}(n)$ is nonnegative and does not increase
 in $n$. Hence, for some $C>0$ and any $n\geq 1$,
 \begin{eqnarray*}
 &\infty>C\geq\sum_{k=0}^\infty\sum_{l=k+1}^\infty\varpi_{L,8M}(l)\geq\sum_{k=0}^n\sum_{l=k+1}^n\varpi_{L,8M}(l)\\
 &\geq\varpi_{L,8M}(n)\sum_{k=0}^n(n-k)=\frac 12n(n+1)\varpi_{L,8M}(n),
 \end{eqnarray*}
 and so (\ref{2.4}) implies the second estimate in (\ref{3.16}). The same argument applied to $\sup_{l\geq n}\be(K,l)$
 yields the first estimate in (\ref{3.16}). Observe also that for any pair of $L^\infty$ functions $g$ and $h$ and a pair
 of $\sig$-algebras $\cG,\cH\subset\cF$ if $g$ is $\cG$-measurable and $h$ is $\cH$-measurable then
 \[
 |E(gh)-EgEh|=|E((E(g|\cH)-Eg)h)|\leq\| g\|_\infty\| h\|_\infty\vp_{L,4M}(\cH,\cG).
 \]
 Now relying on (\ref{3.16}) and the last observation, which replaces here Lemma 3.1 in \cite{Ki22} used in Lemma 3.10 there,
 the proof of the present lemma proceeds essentially in the same way as the proof of Lemma 3.10 in \cite{Ki22} replacing there
  the $\phi$-mixing coefficient $\phi(n)$ by $\vp_{L,4M}(n)$, the approximation coefficient $\rho(n)$ by $\be(K,n)$ and using
  here Lemmas \ref{lem3.2} and \ref{lem3.3} above in place of Lemmas 3.2 and 3.4 in \cite{Ki22}, respectively. Taking all these
  into account the details of the estimates in our circumstances can be easily reproduced from the proof of Lemma 3.10 of \cite{Ki22}.
 \end{proof}

 Next, set $l_1=0$, $n_k=l_k+[(k+2)^\rho],\, l_{k+1}=n_k+3[(k+2)^{\rho/2}],\, k=1,2,...$ where $\rho\geq 1$ will be chosen big enough. We view each time
  interval $[l_k,\, n_k)$ as a block and $[n_k,\, l_{k+1})$ as a much smaller (when $k$ is big) gap between the corresponding blocks.
  Set also $k_N(t)=\max\{ k:\, n_k\leq Nt\}$, $k_N=k_N(T)$ and
  $Q_k=\sum_{l_k\leq j<n_k}\xi^{(k)}(j)$ where $\xi^{(k)}(j)=E(\xi(j)|\cF_{j-[(k+1)^{\rho/2}],j+[(k+1)^{\rho/2}]})$.
  Then we have
 
 %R^{(1)}_{N,k}=\sum_{j=l_{k}}^{n_k-1}(\xi^{(m_N)}(j)
 %-\xi^{(m^{1/4}_N)}(j)),\\
 %&R^{(2)}_{N,k}=\sum_{j=n_{k-1}}^{l_k-1}\xi^{(m_N)}(j)\,\,
 %\mbox{and}\,\, Q_{N,k}(n)=\sum_{k=1}^{k_N(n)}Q_{N,k}
 % and $k_N(t)=\max\{ k:\, n_k\leq t\}$. 
 
 \begin{lemma}\label{lem3.5}
 For any $N\geq 1$ and $k\leq k_N$,
 \begin{eqnarray*}
 &E|E\big(\exp(i\langle w,\,(n_k-l_k)^{-1/2}Q_{k}\rangle)|\cF_{-\infty,n_{k-1}+[(k+1)^{\rho/2}]}\big)-g(w)|\\
 &\leq C_Q(n_k-l_k)^{-\wp}
 \end{eqnarray*}
 for all $w\in\bbR^d$ with $|w|\leq(n_k-l_k)^{\wp/2}$ where $g(w)=\exp(-\frac 12\langle\vs w,w\rangle)$, $\wp=\frac 1{20}$ and $C_Q>0$
  does not depend on $N,n_k,l_k$ and $w$.
 \end{lemma}
 \begin{proof}
 Set $F_{k}=\exp(i\langle w,(n_k-l_k)^{1/2}Q_{k}\rangle)$. Then $F_{k}$ is $\cF_{l_k-[(k+1)^{\rho/2}],\infty}$-measurable and
 since $|F_{k}|=1$ we obtain by (\ref{2.1}) and (\ref{2.2}) that
 \begin{eqnarray*}
 &E|E(F_{k}|\cF_{-\infty,n_{k-1}+[(k+1)^{\rho/2}]})-EF_{k}|\\
 &\leq\| E(F_{k}|\cF_{-\infty,n_{k-1}+[(k+1)^{\rho/2}]})-EF_{k}\|_{4M}\leq\varpi_{L,4M}([(k+1)^{\rho/2}]).
 \end{eqnarray*}
 Since $|e^{ia}-e^{ib}|\leq |a-b|$ we obtain from (\ref{2.3}), taking into account stationarity of $\xi(k)$'s, that
 \begin{eqnarray*}
 &|EF_{k}-f_{n_k-l_k}(w)|\leq |w|(n_k-l_k)^{-1/2}\sum_{j=l_k}^{n_k-1}E|\xi^{(k)}(j)-\xi(j)|\\
 &\leq |w|(k+2)^{\rho/2}\be(K,[(k+1)^{\rho/2}])
 \end{eqnarray*}
 where $f_n(w)$ is the same as in Lemma \ref{lem3.4}. These together with (\ref{2.4}), (\ref{2.4+}) and (\ref{3.16}) yields the estimate of the lemma.
 \end{proof} 

 \subsection{Strong approximations}\label{subsec3.3}
 Our strong approximations (coupling) will be based on the following result which appears as
 Theorem 4.6 of \cite{DP}. As usual, we will denote by $\sig\{\cdot\}$ a $\sig$-algebra generated by
random variables or vectors appearing inside the braces and we write $\cG\vee\cH$ for the minimal
$\sig$-algebra containing both $\sig$-algebras $\cG$ and $\cH$.
\begin{theorem}\label{thm3.6}
Let $\{ V_m,\, m\geq 1\}$ be a sequence of random vectors with values in $\bbR^d$ defined on some
probability space $(\Om,\cF,P)$ and such that $V_m$ is measurable with respect to $\cG_m$, $m=1,2,...$
where $\cG_m,\, m\geq 1$ is a filtration of sub-$\sig$-algebras of $\cF$.
Assume that the probability space is rich enough
so that there exists on it a sequence of uniformly distributed on $[0,1]$ independent random variables $U_m,\, m\geq 1$ independent of $\vee_{m\geq 0}\cG_m$. For each $m\geq 1$, let $G_m$
 be a probability distribution on $\bbR^d$ with the characteristic function
 \[
 g_m(w)=\int_{\bbR^d}\exp(i\langle w,x\rangle)G_m(dx),\,\, w\in\bbR^d.
 \]
 Suppose that for some non-negative numbers $\nu_m,\del_m$ and $K_m\geq 10^8d$,
 \begin{equation}\label{3.19}
 E\big\vert E(\exp(i\langle w,V_m\rangle)|\cG_{m-1})-g_m(w)\big\vert \leq\nu_m
 \end{equation}
 for all $w$ with $|w|\leq K_m$, and that
 \begin{equation}\label{3.20}
 G_m(\{ x:\, |x|\geq\frac 12K_m\})\leq\del_m.
 \end{equation}
 Then there exists a sequence $\{ W_m,\, m\geq 1\}$ of $\bbR^d$-valued independent random
 vectors defined on $(\Om,\cF,P)$ such that $W_m$ has the distribution $G_m$, it is
 $\sig\{ V_m,U_m\}$-measurable, $W_m$ is
 independent of $\sig\{ U_1,...,U_{m-1}\}\vee\cG_{m-1}$ (and so also of $W_1,...,W_{m-1})$ and
 \begin{equation}\label{3.21}
 P\{ |V_m-W_m|\geq\vr_m\}\leq\vr_m
 \end{equation}
 where $\vr_m=16K^{-1}_m\log K_m+2\nu_m^{1/2}K_m^d+2\del_m^{1/2}$.
 In particular, the Prokhorov distance between the distributions $\cL(V_m)$ and $\cL(W_m)$
 of $V_m$ and $W_m$, respectively, does not exceed $\vr_m$.
 \end{theorem}

 In order to apply this theorem we set
 $V_{m}=(n_{m}-l_{m})^{-1/2}Q_{m}$, $\cG_{m}=\sig\{ V_1,...,V_m\}\subset
 \cF_{-\infty,n_{m}+(m+1)^{\rho/2}}$ and $g_m=g$ defined in Lemma \ref{lem3.5}, so that
 $G_m=G$ is the mean zero $d$-dimensional Gaussian distribution with the covariance matrix $\vs$ and the characteristic function $g$.
 By Lemma \ref{lem3.5} for $|w|\leq [(k+2)^\rho]^{\wp/2}$,
\begin{equation}\label{3.22}
 E\big\vert E\big(\exp(i\langle w,V_{m}\rangle)| \cG_{m-1}\big)-g_{m}(w)\big\vert\\
 \leq C_Q(n_{m}-l_{m})^{-\wp}
 \end{equation}
 where, recall, $\wp=\frac 1{20}$. Next, we define $K_m=(n_m-l_m)^{\wp/4d}=[(m+2)^\rho]^{\wp/4d}$. 
 Theorem \ref{thm3.6} requires that $K_m\geq 10^8d$ which will hold true in our case if 
 $3^\rho\geq 10^{32d/\wp}d^{4d/\wp}+1$, and so we choose $\rho$ so large that the latter inequality will be satisfied.

Next, let $\Psi$ be a mean zero Gaussian random variable with the
covariance matrix $\vs$. Then by the Chebyshev inequality,
\begin{eqnarray}\label{3.23}
& G(\{ y\in\bbR^d:\, |y|\geq\frac 12(n_m-l_m)^{\frac \wp{4d}}\})\\
&\leq P\{|\Psi| \geq\frac 12(n_m-l_m)^{\frac \wp{4d}}\}\nonumber\\
&\leq C_Gd(n_m-l_m)^{-\frac \wp{2d}}=C_Gd[(m+2)^\rho]^{-\wp/4d}=\del_m\nonumber
\end{eqnarray}
where $C_G>0$ does not depend on $m,n_m$ and $l_m$.
Now, Theorem \ref{thm3.6} provides us with independent random vectors $\{ W_m,\, m\geq 1\}$
having the mean zero Gaussian distribution with the covariance matrix $\vs$ and such that
 \begin{eqnarray}\label{3.24}
 &\quad\vr_{m}=\vr_m(N)=4\frac \wp d(n_m-l_m)^{-\wp/4d}\log(n_m-l_m)+2C_6^{1/2}(n_m-l_m)^{-\wp/4}\\
 &+ 2C_Gd(n_m-l_m)^{-\wp/4d}\leq C_\vr(n_m-l_m)^{-\wp/8d}=C_\vr[(m+2)^\rho]^{-\wp/8d}\nonumber
 \end{eqnarray}
 where $C_\vr>0$ does not depend on $m,n_m$ and $l_m$.

Next, let $\cW(t),\, t\geq 0$ be a $d$-dimensional Brownian motion with the covariance matrix $\vs$.
Then the sequences of random vectors 
\[
\tilde W_k=(\cW(n_k)-\cW(l_k))\quad\mbox{and}\quad (n_k-l_k)^{1/2}W_k,\, k\leq k_N(T)
\]
have the same distribution. Hence, by Lemma A1 from \cite{BP} the sequences $\xi(k)$ and $W_k,\, k\geq 1$ 
can be redefined without changing their joint distributions on a richer probability space where there exists
a $d$-dimensional Brownian motion $\cW(t)$ with the covariance matrix $\vs$ such that the pairs $(V_k,W_k)$ and
\[
(V_k,\,(n_k-l_k)^{-1/2}(\cW(n_k)-\cW(l_k)),\, k\leq k_N(T)),
\]
constructed by means of the redefined processes, have the same joint distributions. Thus, we can assume from now on that
\[
V_k=(n_k-l_k)^{-1/2}Q_{k}\quad\mbox{and}\quad W_k=(n_k-l_k)^{-1/2}(\cW(n_k)-\cW(l_k)),
\]
constructed by the redefined processes, have the properties asserted in Theorem \ref{thm3.6}.

The first step in the proof of Theorem \ref{thm2.1} is the following assertion concerning the uniform norms.
\begin{proposition}\label{prop3.7}
Let, as explained above, the sequence of random vectors $\xi(n),\,-\infty<n<\infty$ be redefined preserving its distributions on a sufficiently rich
 probability space which contains also a $d$-dimensional Brownian motion $\cW$ with the covariance matrix $\vs$ (at the time 1) so that 
the sequences of random vectors $V_k=(n_k-l_k)^{-1/2}\sum_{j=l_k}^{n_k-1}\xi^{k)}(j)$ and $W_k=(n_k-l_k)^{-1/2}(\cW(n_k)-\cW(l_k)),\, k\geq 1$
have the properties described in Theorem \ref{thm3.6}. Then the normalized sums $\bbS^{(1)}_N(t),\, t\in[0,T]$ and the rescaled  Brownian motion
 $W_N(t)=N^{-1/2}\cW(Nt),\,  t\in[0,T]$ satisfy for any $N\geq 1$ the estimate,
 \begin{equation}\label{3.25}
 I_N=E\sup_{0\leq t\leq T}| \bbS^{(1)}_N(t)-W_N(t)|^{4M}\leq C_MN^{-\del_1}
 \end{equation}
 where, recall, $\bbS_N^{(1)}(t)=S_N(t)=N^{-1/2}\sum_{0\leq k<tN}\xi(k)$ and the constants $\del_1>0$, $C_M>0$ do not depend on $N$.
\end{proposition}
\begin{proof} 
Observe that
\[
\sup_{0\leq t\leq T}| S^{(1)}_N(t)-W_N(t)|\leq\sum_{i=1}^6(\max_{1\leq n\leq k_N}I^{(i)}_N(n))+I_N^{(7)}+I_N^{(8)}
\]
where
\[
I_N^{(1)}(n)=N^{-1/2}|\sum_{1\leq k\leq n}(Q_k-(\cW(n_k)-\cW(l_k)))|=N^{-1/2}|\sum_{1\leq k\leq n}((n_k-l_k)(V_k-W_k))|
\]
and
\[
I_N^{(2)}(n)=N^{-1/2}|\sum_{1\leq k\leq n}\sum_{j=l_k}^{n_k-1}(\xi(j)-\xi^{(k)}(j))|
\]
with the last term which is due to the fact that the sum in $Q_k$ is over $\xi^{(k)}(j)$'s and not over $\xi(j)$'s.
The following two terms
\[
I_N^{(3)}(n)=N^{-1/2}|\sum_{1\leq k< n}\sum_{j=n_{k}}^{l_{k+1}-1}\xi(j)|\,\,\mbox{and}\,\,I_N^{(4)}(n)=N^{-1/2}|\sum_{1\leq k< n}(\cW(l_{k+1})-\cW(n_k))|
\]
correspond to summation in gaps while the next two terms
\begin{eqnarray*}
&I_N^{(5)}(n)=N^{-1/2}\max_{1\leq k\leq n}\max_{l_{k}\leq m<l_{k+1}}|\sum_{j=l_{k}}^m\xi(j)|\\
&\mbox{and}\,\,\, I_N^{(6)}(n)=N^{-1/2}\max_{1\leq k\leq n}\sup_{l_{k}\leq t<l_{k+1}}|\cW(t)-\cW(l_{k})|
\end{eqnarray*}
correspond to the case where the supremum in $t$ of $|S_N(t)-W_N(t)|$ occurs inside of a block or of a gap and not at their ends.
The last two terms
\[
I_N^{(7)}=N^{-1/2}\max_{n_{k_N}\leq m<TN}|\sum_{j=n_{k_N}}^m\xi(j)|\,\,\mbox{and}\,\, I_N^{(8)}=N^{-1/2}\sup_{n_{k_N}\leq t\leq TN}|\cW(t)-\cW(n_{k_N})|
\]
correspond to the case  where the supremum in $t$ of $|S_N(t)-W_N(t)|$ occur after the last block but before the last gap end. 

%We start with
%\begin{eqnarray*}
%&E\max_{1\leq n\leq k_N}(I_N^{(1)}(n))^{4M}\leq E\max_{1\leq n\leq\ell_N}(I_N^{(1)}(n))^{4M}\\
%&+2^{4M}\big(E(I_N^{(1)}(\ell_N))^{4M}+E\max_{\ell_N<n\leq k_N}(\hat I^{(1)}_N(\ell_N,n))^{4M}\big)
%\end{eqnarray*}
%where $\ell_N=\max\{ k:\, n_k\leq\sqrt N\}$, $I_N^{(1)}(n)=N^{-1/2}|\sum_{1\leq k\leq n}(Q_k-(\cW(n_k)-\cW(l_k))|$ and
%$\hat I^{(1)}_N(l,n)=N^{-1/2}|\sum_{l\leq k\leq n}(Q_k-(\cW(n_k)-\cW(l_k))|$. Clearly,
%\[
%|I_N^{(1)}(n)|\leq\sum_{1\leq k\leq n}\sum_{l_k\leq j<n_k}(|\xi^{(k)}(j)|+|\cW(j+1)-\cW(j)|)
%\]
%and if $n\leq\ell_M$ then there are no more than $2\sqrt N$ terms in this double sum. It follows from (\ref{3.3}) of Lemma \ref{lem3.2}
%that
%\[
%E\max_{1\leq n\leq\ell_N}(I^{(1)}_N(n))^{4M}\leq C^{(1)}N^{-M}\,\,\mbox{and}\,\, E(I^{(1)}_N(\ell_N))^{4M}\leq C^{(1)}N^{-M}
%\]
%where a constant $C^{(1)}>0$ does not depend on $N$. 

We will estimate 
\[
E\max_{1\leq n\leq k_N}(I^{(1)}_N(n))^{4M}=N^{-2M}E\max_{1\leq n\leq k_N}\big(\sum_{1\leq k\leq n}(Q_k-(\cW(n_k)-\cW(l_k)))\big)^{4M}
\]
using Lemma \ref{lem3.1}. Recall that $Q_{k}$ is $\cG_k\subset\cF_{-\infty,n_k+[(k+1)^{\rho/2}]}$-measurable
and $\cW(n_k)-\cW(l_k)=(n_k-l_k)^{1/2}W_k$ is $\cG_k\vee\sig\{ U_1,...,U_k\}$-measurable. On the other hand, $\cW(n_k)-\cW(l_k)$ is independent
of $\cG_{k-1}\vee\sig\{ U_1,...,U_{k-1}\}$, and so for any $j<k$,
\[
E(\cW(n_k)-\cW(l_k)|\cG_j\vee\sig\{ U_1,...,U_j\})=E(\cW(n_k)-\cW(l_k))=0.
\]
Next, since $V_k$, and so $Q_{k}$, are independent of $\sig\{ U_1,...,U_j\}$ when $j<k$ and the latter $\sig$-algebra is independent of $\cG_j$, 
we obtain that
\[
E(Q_{k}|\cG_j\vee\sig\{ U_1,...,U_j\})=E(E(Q_{k}|\cG_{j\vee(k-2)})|\cG_j).
\]
Now By (\ref{2.1}), (\ref{2.2}) and (\ref{2.4+}) we obtain for $j<k$,
\begin{eqnarray*}
&\| E(Q_{k}|\cG_{j\vee(k-2)})\|_{4M}\leq\sum_{i=l_k}^{n_k-1}\| E(\xi^{(k)}(i)|\cF_{-\infty,n_j\vee n_{k-2}+[(k+1)^{\rho/2}]})\|_{4M}\\
&\leq \|\xi(0)\|_L\sum_{i=l_k}^{n_k-1}\vp_{L,4M}(i-n_j\vee n_{k-2}-2[(k+1)^{\rho/2}])\leq\hat C_M(k+1)^{-\rho/2}\\
&\mbox{if}\,\, j=k-1\,\,\,\mbox{and}\,\,\,\leq\hat C_M(k+1)^{-\rho}\,\,\mbox{if}\,\, j<k-1
\end{eqnarray*}
where $\hat C_M>0$ does not depend on $k$ and $N$.

Now, in order to use Lemma \ref{lem3.1} we have to bound $A_{4M}$ appearing there and to do this it remains to consider the case
$j=k$, i.e. to estimate $\| Q_{k}-(\cW(n_k)-\cW(l_k))\|_{4M}$ and then to combine this with the above conditional expectations estimates. 
By the Cauchy-Schwarz inequality for any $k\geq 1$,
\begin{eqnarray*}
&E|Q_{k}-(\cW(n_k)-\cW(l_k))|^{4M}=(n_k-l_k)^{2M}\big(E(|V_k-W_k|^{4M}\bbI_{|V_k-W_k|\leq\vr})\\
&+E(|V_k-W_k|^{4M}\bbI_{|V_k-W_k|>\vr})\big)\\
&\leq (n_k-l_k)^{2M}(\vr_k^{4M}+(E|V_k-W_k|^{8M})^{1/2}(P\{|V_k-W_k|>\vr_k)^{1/2}\}\\
&\leq (n_k-l_k)^{2M}(\vr_k^{4M}+\vr_k^{1/2}2^{4M}((E|V_k|^{8M})^{1/2}+(E|W_k|^{8M})^{1/2}).
\end{eqnarray*}
By Lemma \ref{lem3.3} (with $2M$ in place of $M$ in (\ref{3.3})) and the standard moment estimates for the Brownian motion,
\[
(n_k-l_k)^{2M}(E|V_k|^{8M})^{1/2}=(E|Q_{k}|^{8M})^{1/2}\leq\hat C_M(n_k-l_k)^{2M}
\]
and
\[
(n_k-l_k)^{2M}(E|W_k|^{8M})^{1/2}=(E|\cW(n_k)-\cW(l_k)|^{8M})^{1/2}\leq\hat C_M(n_k-l_k)^{2M}
\]
where $\hat C_M>0$ does not depend on $k$ and $N$. Combining the above estimates including (\ref{3.24}) we conclude that $A_{4M}$ from Lemma \ref{lem3.1}
for the sum $\sum_{1\leq k\leq n}(Q_{k}-(\cW(n_k)-\cW(l_k))$ is bounded by $C(n_{k_N}-l_{l_N})^{\frac 12-\frac \wp{64Md}}$ where $C>0$ does not depend
on $N$. Taking into account that $k_N$ is of order $N^{\frac 1{\rho+1}}$ we obtain that
\begin{eqnarray*}
&E\max_{1\leq n\leq k_N}(I_N^{(1)}(n))^{4M}\leq C^{(1)}(M)N^{-2M}(n_{k_N}-l_{k_N})^{2M-\frac \wp{16d}}k_N^{2M}\\
&\leq\tilde C^{(1)}(M)N^{-2M}k_N^{\rho(2M-\frac {\wp}{16d})+2M}\leq\hat C^{(1)}(M)N^{-\frac {\rho\wp}{16d(\rho+1)}},
\end{eqnarray*}
 where $C^{(1)}(M),\tilde C^{(1)}(M),\hat C^{(1)}(M)>0$ do not depend on $N$.

The estimates of $I_N^{(i)},\, i=2,...,8$ are a bit more straightforward. First, we write 
\begin{eqnarray*}
&E\max_{1\leq n\leq k_N}(I_N^{(2)}(n))^{4M}\leq E\max_{1\leq n\leq\ell_N}(I_N^{(2)}(n))^{4M}\\
&+2^{4M}\big(E(I_N^{(2)}(\ell_N))^{4M}+E\max_{\ell_N<n\leq k_N}(I^{(2)}_N(\ell_N,n))^{4M}\big)
\end{eqnarray*}
where $\ell_N=\max\{ k:\, n_k\leq N^{7/8}\}$, $I^{(2)}_N(l,n)=N^{-1/2}|\sum_{l\leq k\leq n}\sum_{l_k<j\leq n_k}(\xi(j)-\xi^{(k)}(j))|$
and we observe that $\ell_N$ is of order $N^{\frac 7{8(\rho+1)}}$. When $n\leq\ell_N$ there are no more than $2N^{7/8}$ terms in the sums of $I_N^{(2)}(n)$. 
Hence by (\ref{3.3}) of Lemma \ref{lem3.2},
\[
E\max_{1\leq n\leq\ell_N}(I^{(2)}_N(n))^{4M}\leq C^{(2)}(M)N^{-M/4}\,\,\mbox{and}\,\, E(I^{(2)}_N(\ell_N))^{4M}\leq C^{(2)}(M)N^{-M/4}
\]
where a constant $C^{(2)}(M)>0$ does not depend on $N$. Next,
\begin{eqnarray*}
&E\max_{\ell_N<n\leq k_N}( I_N^{(2)}(\ell_N,n))^{4M}\leq\hat C^{(2)}(M)N^{2M}\sum_{k=\ell_N}^{ k_N}\sum_{k=l_k}^{n_k}E|\xi(j)-\xi^{(k)}(j)|^{4M}\\
&\leq\hat C^{(2)}(M)N^{2M+1}(\be(4M,(\ell_N+1)^{\rho/2}))^{4M}\leq\tilde C^{(2)}(M)N^{-M/4}
\end{eqnarray*}
for some $\hat C^{(2)}(M),\,\tilde C^{(2)}(M)>0$ which do not depend on $N$ where we assume that $\rho\geq 13$ and take into account (\ref{2.4+}).

Relying on (\ref{3.3}) of Lemma \ref{lem3.2} and taking into account that sums over gaps contain the order of $N^{\frac {\rho+2}{2(\rho+1)}}$ terms we estimate
\[
E\max_{1\leq n\leq k_N}(I^{(3)}(n))^{4M}\leq C_g(M)N^{-M/2}\,\,\mbox{and}\,\, E\max_{1\leq n\leq k_N}(I^{(4)}(n))^{4M}\leq C_g(M)N^{-M/2}
\]
where $C_g(M)>0$ does not depend on $N$. Observe that an estimate of the last expression is easier since there we deal with sums of independent (Gaussian) random variables.
Using again (\ref{3.3}) of Lemma \ref{lem3.2} and the fact that $k_N$ is of order $N^{\frac 1{\rho+1}}$ we estimate
\begin{eqnarray*}
&E\max_{1\leq n\leq k_N}(I^{(5)}(n))^{4M}\leq N^{-2M}\sum_{0\leq k<k_N}E\max_{l_{k}\leq m<l_{k+1}}|\sum_{j=l_k}^m\xi(j)|^{4M}\\
&\leq\hat C_\xi(M)N^{-2M}\sum_{1\leq k\leq k_N}(k+2)^{2M\rho}\leq\tilde C_\xi(M)N^{-\frac M{\rho+1}}
\end{eqnarray*}
where $\hat C_\xi(M)>0$ and $\tilde C_\xi(M)>0$ do not depend on $N$. Similarly, by the standard estimates for the Brownian motion,
\[
E\max_{1\leq n\leq k_N}(I^{(6)}(n))^{4M}\leq C_\cW(M)N^{-\frac M{\rho+1}}
\]
where $C_\cW(M)>0$ does not depend on $N$. The remaining estimates
\[
E(I_N^{(7)})^{4M}\leq C_\xi(M)N^{-2M}k_N^{2M\rho}\leq\tilde C_\xi(M)N^{-\frac {2M}{\rho+1}}\,\,\mbox{and}\,\, E(I_N^{(8)})^{4M}\leq C_\cW(M)N^{-\frac {2M}{\rho+1}}
\]
for come constants $C_\xi(M)>0,\tilde C_\xi(M)$ and $C_\cW(M)>0$ which do not depend on $N$, follow directly from (\ref{3.3}) of Lemma \ref{lem3.2} and the estimate of $k_N$.
Collecting these estimates we arrive at (\ref{3.25}) completing the proof 
of Proposition \ref{prop3.7}.
\end{proof}

\section{Discrete time case estimates}\label{sec4}\setcounter{equation}{0}
\subsection{Variational norm estimates for sums}\label{subsec4.1}

First, we write for $0\leq s<t\leq T$,
\begin{eqnarray}\label{4.1}
&E\| S_N-W_N\|_{p,[0,T]}^{4M}=E\sup_{0=t_0<t_1<...<t_m=T}\big(\sum_{0\leq i<m}|S_N(t_{i+1})\\
&-W_N(t_{i+1})-S_N(t_{i})+W_N(t_i)|^p\big)^{4M/p}\leq 2^{\frac {4M}p-1}(J^{(1)}_N+2^{4M-1}(J_N^{(2)}+J_N^{(3)}))\nonumber
\end{eqnarray}
where $p\in(2,3)$,
\begin{eqnarray*}
&J_N^{(1)}=E\sup_{0=t_0<t_1<...<t_m= T}\big(\sum_{i:\, t_{i+1}-t_i>N^{-(1-\al)}}|S_N(t_{i+1})-W_N(t_{i+1})\\
&-S_N(t_i)+W_N(t_i)|^p\big)^{4M/p},
\end{eqnarray*}
\[
J_N^{(2)}=E\sup_{0=t_0<t_1<...<t_m= T}\big(\sum_{i:\, t_{i+1}-t_i\leq N^{-(1-\al)}}|S_N(t_{i+1})-S_N
(t_i)|^p\big)^{4M/p},
\]
\[
J_N^{(3)}=E\sup_{0=t_0<t_1<...<t_m= T}\big(\sum_{i:\, t_{i+1}-t_i\leq N^{-(1-\al)}}|W_N(t_{i+1})-W_N
(t_i)|^p\big)^{4M/p},
\]
with $\al\in (0,1)$ which will be specified later on. Since there exist no more than $[TN^{1-\al}]$ intervals
$[t_{i},t_{i+1}]$ such that $t_{i+1}-t_i>N^{-(1-\al)}$  and by Proposition \ref{prop3.7},
\begin{equation}\label{4.2}
E\sup_{0\leq t\leq T}|S_N(t)-W_N(t)|^{4M}\leq C_MN^{-\del_1},
\end{equation}
we obtain that
\begin{equation}\label{4.3}
J_N^{(1)}\leq C_M^{1/2}2^{4M}T^{4M/p}N^{\frac {4M}p(1-\al)-\del_1}
\end{equation}
where $\al$ will be chosen so that
\begin{equation}\label{4.4}
1>\al>1-\frac {p\del_1}{4M}.
\end{equation}

In order to estimate $J_N^{(2)}$ observe that if $[t_iN]=[t_{i+1}N]$ then $S_N(t_i)=S_N(t_{i+1})$ and
defining $k_j=[t_{i_j}N]$, where $0=t_{i_0}<t_{i_1}<...<t_{i_n}=T$ is the maximal subsequence of $t_0,t_1,...,t_m$ such that $[t_{i_j}N]<[t_{i_{j+1}}N],\, j=0,1,...,n-1$, we obtain
\[
J_N^{(2)}=E\max_{0=k_0<k_1<...<k_n\leq TN}\big(\sum_{j:\,k_{j+1}-k_j\leq N^\al}|S_N(\frac {k_{j+1}}N)-S_N
(\frac {k_{j}}N)|^p\big)^{4M/p}.
\]
Taking into account that $\sum_{j:\, k_{j+1}-k_j\leq N^\al}|k_{j+1}-k_j|N^{-1}\leq T$ we obtain by Proposition
 \ref{hoelder2} with $\bfM=4M$ together with Lemma \ref{lem3.2} that
\begin{eqnarray}\label{4.5}
&J_N^{(2)}=E\max_{0=k_0<k_1<...<k_n\leq TN}\big(\sum_{j:\,k_{j+1}-k_j\leq N^\al}(|S_N(\frac {k_{j+1}}N)-S_N(\frac {k_{j}}N)|^p\\
&|\frac {k_{j+1}-k_j}N|^{-p(\frac 12-\be)}|\frac {k_{j+1}-k_j}N|^{p(\frac 12-\be)})\big)^{4M/p}\nonumber\\
&\leq\max_{0=k_0<k_1<...<k_n\leq TN}\big(\sum_{j:\, k_{j+1}-k_j\leq N^\al}(\frac {|k_{j+1}-k_j|}N)^{p(\frac 12-\be)}\big)^{4M/p}\nonumber\\
&\times E\max_{k,l:\, 0\leq k<l\leq TN,\,l-k\leq N^\al}\big(|S_N(\frac lN)-S_N(\frac kN)|^{4M}(\frac {|l-k|}N)^{-2M(1-2\be)}\big)\nonumber\\
&\leq T^{4M/p}N^{-4M(1-\al)(\frac 12-\frac 1p-\be)}EC^{4M}_{\frac 12,\be,N}\leq  K_{S,\frac 12,\be}(4M)T^{4M/p}N^{-4M(1-\al)(\frac 12-\frac 1p-\be)}\nonumber
\end{eqnarray}
where we estimate
\begin{eqnarray*}
&\sum_{j:\, k_{j+1}-k_j\leq N^\al}(\frac {|k_{j+1}-k_j|}N)^{p(\frac 12-\be)}\\
&\leq N^{-(1-\al)(p(\frac 12-\be)-1)}\sum_{j:\, k_{j+1}-k_j\leq N^\al}(\frac {|k_{j+1}-k_j|}N)\leq TN^{-p(1-\al)(\frac 12-\frac 1p-\be)},\\
\end{eqnarray*}
$C_{\frac 12,\be,N}$ comes from Proposition \ref{hoelder2},
$K_{S,\frac 12,\be}(4M)>0$ does not depend on $N$ and $\be$ is chosen so that $\frac 1{4M}<\be<\frac 12-p^{-1}$,
which is possible since $p>2$ and $M>\frac p{p-2}$.

Similarly, taking into account that $\sum_{0\leq i<m}(t_{i+1}-t_i)=T$ and relying on Theorem \ref{hoelder} we obtain
that
\begin{eqnarray}\label{4.6}
&J^{(3)}_N\leq\sup_{0=t_0<t_1<...<t_m= T}\big(\sum_{i:\, t_{i+1}-t_i\leq N^{-(1-\al)}}|t_{i+1}-t_i|^{p(\frac 12-\be)}
\big)^{4M/p}\\
&\times E\sup_{u,v:\, 0\leq u<v\leq T}(|W_N(v)-W_N(u)|^{4M}|v-u|^{-2M(1-2\be)})\nonumber\\
&\leq K_{W,\frac 12,\be}(M)T^{4M/p}N^{-4M(1-\al)(\frac 12-\frac 1p-\be)}\nonumber
\end{eqnarray}
where we estimate
\[
\sum_{i:\, t_{i+1}-t_i\leq N^{-(1-\al)}}|t_{i+1}-t_i|^{p(\frac 12-\be)}\leq TN^{-(1-\al)(p(\frac 12-\be)-1)},
\]
$K_{W,\frac 12,\be}(M)>0$ does not depend on $N$ and $\be$ is chosen again to satisfy $\frac 1{4M}<\be<\frac 12-\frac 1p$.
Finally, we derive (\ref{2.8}) for $\nu=1$ from (\ref{4.1}) and (\ref{4.3})--(\ref{4.6}).

\subsection{Supremum norm estimates for iterated sums}\label{subsec4.2}

The supremum norm moment estimates for iterated sums under our conditions are carried out similarly
to Section 4.1 in \cite{FK}, so we will only indicate the main points here.
Set $m_N=[N^{1-\ka}]$ with a small $\ka>0$ which will be chosen later
on, $\nu_N(l)=\max\{ jm_N:\, jm_N<l\}$ if $l>m_N$ and
\[
R_i(k)=R_i(k,N)=\sum_{l=(k-1)m_N}^{km_N-1}\xi_i(l)\,\,\mbox{for}\,\, k=1,2,...,\iota_N(T)
\]
 where $\iota_N(t)=[[Nt]m_N^{-1}]$.
For $1\leq i,j\leq d$ define
\begin{equation*}
\bbU_N^{ij}(t)=N^{-1}\sum_{l=m_N}^{\iota_N(t)m_N-1}\xi_j(l)\sum_{k=0}^{\nu_N(l)}\xi_i(k)=N^{-1}\sum_{1<l\leq\iota_N(t)}
R_j(l)\sum_{k=0}^{(l-1)m_N-1}\xi_i(k).
\end{equation*}
Set also
\[
\bar\bbS_N^{ij}(t)=\bbS_N^{ij}(t)-t\sum_{l=1}^\infty E(\xi_i(0)\xi_j(l))
\]
where the series converges absolutely in view of (\ref{2.4}) and (\ref{3.14}).
In the same way as in the proof of Lemma 4.1 from \cite{FK} we obtain relying on Lemma \ref{lem3.2}
that for all $i,j=1,...,d$ and $N\geq 1$,
\begin{equation}\label{4.7}
E\sup_{0\leq t\leq T}|\bar\bbS_N^{ij}(t)-\bbU_N^{ij}(t)|^{2M}\leq C_{SU}(M)N^{-M\min(\ka,1-\ka)}
\end{equation}
where $C_{SU}(M)>0$ does not depend on $N$.
The proof here differs from \cite{FK} only in the use of Lemma \ref{lem3.2} in place of Lemma 3.4 there and
in estimating $|E(\xi_i(k)\xi_j(l))|$ by means of (\ref{3.14}) in place of using simpler arguments based on
boundedness and $\phi$-mixing assumptions there.

Now set
\[
S^i_N(t)=N^{-1/2}\sum_{0\leq k<[Nt]}\xi_i(k),\,\, i=1,...,d
\]
and observe that
\[
\bbU^{ij}_N(t)=\sum_{2\leq l\leq\iota_N(t)}\big(S^j_N\big(\frac {lm_N}N\big)-S^j_N\big(\frac {(l-1)m_N}N\big)\big)S^i_N
\big(\frac {(l-1)m_N}N\big).
\]
Define
\[
\bbV^{ij}_N(t)=\sum_{2\leq l\leq\iota_N(t)}\big(W^j_N\big(\frac {lm_N}N\big)-W^j_N\big(\frac {(l-1)m_N}N\big)\big)W^i_N
\big(\frac {(l-1)m_N}N\big)
\]
where $W_N=(W^1_N,...,W^d_N)$ is the $d$-dimensional Brownian motion with the covariance matrix $\vs$ (at the time 1)
appearing in (\ref{2.6}). Then
\begin{eqnarray}\label{4.8}
&\sup_{0\leq t\leq T}|\bbU^{ij}_N(t)-\bbV^{ij}_N(t)|\leq\sum_{2\leq l\leq\iota_N(T)}\big(\big(\big\vert S^j_N\big(\frac {lm_N}N\big)-W^j_N\big(\frac {lm_N}N\big)\big\vert\\
&+\big\vert S^j_N\big(\frac {(l-1)m_N}N\big)-W^j_N\big(\frac {(l-1)m_N}N\big)\big\vert\big)\big\vert
S^i_N\big(\frac {(l-1)m_N}N\big)\big\vert\nonumber\\
&+\big\vert W^j_N\big(\frac {lm_N}N\big)-W^j_N\big(\frac {(l-1)m_N}N\big)\big\vert\big\vert  S^i_N\big(\frac {(l-1)m_N}N\big)-W^i_N\big(\frac {(l-1)m_N}N\big)\big\vert\big)\nonumber\\
&\leq\iota_N(T)\big(\max_{1\leq l\leq\iota_N(T)}\vert S^j_N\big(\frac {lm_N}N\big)-W^j_N\big(\frac {lm_N}N\big)\big\vert
\max_{1\leq l\leq\iota_N(T)}\big\vert S^i_N\big(\frac {lm_N}N\big)\big\vert\nonumber\\
&+\max_{1\leq l\leq\iota_N(T)}\big\vert W^j_N\big(\frac {lm_N}N\big)-W^j_N\big(\frac {(l-1)m_N}N\big)\big\vert\nonumber\\
&\times\max_{1\leq l\leq\iota_N(T)}\big\vert  S^i_N\big(\frac {lm_N}N\big)-W^i_N\big(\frac {lm_N}N\big)\big\vert\big).\nonumber
\end{eqnarray}
Now relying on (\ref{3.3}), Proposition \ref{prop3.7}, the Cauchy--Schwarz inequality and the standard moment estimates for
the Brownian motion we obtain from (\ref{4.8}) that
\begin{equation}\label{4.9}
E\sup_{0\leq t\leq T}|\bbU^{ij}_N(t)-\bbV^{ij}_N(t)|^{2M}\leq C_{UV}(M)N^{-\frac 12(\del_1-4M\ka)}
\end{equation}
where $C_{UV}(M)>0$ does not depend on $N$, $\del_1>0$ is the same as in Proposition \ref{prop3.7} and we take $\ka<\frac {\del_1}{4M}$.

Next, relying on the standard moment estimates for stochastic integrals in the same way as in \cite{FK} we obtain that
\begin{equation}\label{4.10}
E\sup_{0\leq t\leq T}|\int_0^tW_N^i(s)dW^j_N(s)-\bbV^{ij}_N(t)|^{2M}\leq C_{WV}(M)N^{-(M-1)\ka}
\end{equation}
for some $C_{WV}(M)>0$ which does not depend on $N$. It follows from (\ref{4.7})--(\ref{4.10}) that
 \begin{equation}\label{4.11}
E \max_{1\leq i,j\leq d}\sup_{0\leq t\leq T}|\bbS^{ij}_N(t)-\bbW^{ij}_N(t)|^{2M}\leq C_{SW}(M)N^{-\del_2}
 \end{equation}
for some $C_{SW}(M)>0$ and $\del_2>0$ which do not depend on $N$.

\subsection{Variational norm estimates for iterated sums}\label{subsec4.3}

For $0\leq s<t\leq T$ and $i,j=1,...,d$ set
\[
\bar\bbS_N^{ij}(s,t)=N^{-1}\sum_{[sN]\leq k<l<[Nt]}\xi_i(k)\xi_j(l)-(t-s)\sum_{l=1}^\infty E(\xi_i(0)\xi_j(l))
\]
\[
\mbox{and}\,\,\,\bar\bbW_N^{ij}(s,t)=\int_s^t(W^i_N(u)-W^i_N(s))dW_N^j(u).
\]
Hence,
\[
\bbS_N^{ij}(s,t)-\bbW_N^{ij}(s,t)=\bar\bbS_N^{ij}(s,t)-\bar\bbW_N^{ij}(s,t).
\]
Now for $0\leq T$,
\begin{eqnarray}\label{4.12}
&E\sup_{0=t_0<t_1<...<t_m=T}(\sum_{0\leq q<m} |\bbS_N^{ij}(t_q,t_{q+1})- \bbW_N^{ij}(t_q,t_{q+1})|^{p/2})^{4M/p}\\
&\leq 2^{\frac {4M}p-1}(\bbJ_N^{(1)}+2^{2M-1}(\bbJ_N^{(2)}+\bbJ_N^{(3)}))\nonumber
\end{eqnarray}
where $p\in(2,3)$,
\begin{equation*}
\bbJ_N^{(1)}=E\sup_{0=t_0<t_1<...<t_m= T}\big(\sum_{l:\, t_{l+1}-t_l>N^{-(1-\al)}}|\bbS_N^{ij}(t_l,t_{l+1})-
\bbW_N^{ij}(t_l,t_{l+1})|^{p/2}\big)^{4M/p},
\end{equation*}
\[
\bbJ_N^{(2)}=E\sup_{0=t_0<t_1<...<t_m= T}\big(\sum_{l:\, t_{l+1}-t_l\leq N^{-(1-\al)}}|\bar\bbS_N^{ij}(t_l,t_{l+1})|
^{p/2}\big)^{4M/p}
\]
and
\[
\bbJ_N^{(3)}=E\sup_{0=t_0<t_1<...<t_m= T}\big(\sum_{l:\, t_{l+1}-t_l\leq N^{-(1-\al)}}|\bar\bbW_N(t_l,t_{l+1})|^{p/2}\big)^{4M/p}.
\]

Observe that
\[
\bbS^{ij}_N(s,t)=\bbS^{ij}_N(t)-\bbS^{ij}_N(s)-S^i_N(s)(S^j_N(t)-S^j_N(s))
\]
\[
\mbox{and}\,\,\,\bbW^{ij}_N(s,t)=\bbW^{ij}_N(t)-\bbW^{ij}_N(s)-W^i_N(s)(W^j_N(t)-W^j_N(s)).
\]
Hence, by (\ref{4.11}), Lemma \ref{lem3.2}, Proposition \ref{prop3.7} and the Cauchy-Schwarz inequality,
\begin{eqnarray}\label{4.13}
&E\sup_{0\leq u<v\leq T}|\bbS^{ij}_N(u,v)-\bbW^{ij}_N(u,v)|^{2M}\\
&\leq 4^{2M-1}E\sup_{0\leq v\leq T}|\bbS^{ij}_N(v)-\bbW^{ij}_N(v)|^{2M}\nonumber\\
&+2^{6M}(E\sup_{0\leq v\leq T}|S^j_N(v)-W^j_N(v)|^{4M})^{1/2}(E\sup_{0\leq v\leq T}|S^i_N(v)|^{4M}\nonumber\\
&+E\sup_{0\leq v\leq T}|W^i_N(v)|^{4M})^{1/2}\leq\bbC_{\bbS\bbW}(M)N^{-\del_3}\nonumber
\end{eqnarray}
where $\del_3=\min(\frac {\del_1}2,\del_2)$ and $\bbC_{\bbS\bbW}(M)>0$ does not depend on $N$.
Since there exists no more than $[TN^{1-\al}]$ intervals $[t_l,t_{l+1}]\subset[0,T]$ with $t_{l+1}-t_l>
N^{-(1-\al)}$ and taking into account (\ref{4.13}) we obtain that
\begin{equation}\label{4.14}
\bbJ_N^{(1)}\leq\bbC_{\bbS\bbW}(M)T^{4M/p}N^{\frac {4M}p(1-\al)-\del_3}.
\end{equation}

In order to estimate $\bbJ^{(2)}_N$ we introduce
\[
\hat\bbS_N^{ij}(u,v)=\bbS_N^{ij}(u,v)-E\bbS_N^{ij}(u,v)=N^{-1}\sum_{[uN]\leq k<l<[vN]}(\xi_i(k)\xi_j(l)
-E(\xi_i(k)\xi_j(l)))
\]
and observe that
\begin{equation}\label{4.15}
|\hat\bbS_N^{ij}(u,v)-\bar\bbS_N^{ij}(u,v)|\leq\tilde C(v-u)
\end{equation}
where $\tilde C=2\sum_{l=1}^{\infty}|E(\xi_i(0)\xi_j(l))|<\infty$. Hence,
\begin{eqnarray}\label{4.16}
&\bbJ^{(2)}_N\leq 2^{2M}E\sup_{0=t_0<t_1<...<t_m= T}\big(\sum_{l:\, t_{l+1}-t_l\leq N^{-(1-\al)}}(|\hat\bbS_N^{ij}(t_l,t_{l+1})|
^{p/2}\\
&+\tilde C^{p/2}(t_{l+1}-t_l)^{p/2})\big)^{4M/p}\nonumber\\
&\leq 2^{2M}(\bbJ_N^{(4)}+\tilde C^{2M}\sup_{0=t_0<t_1<...<t_m=T}(\sum_{l:\, t_{l+1}-t_l\leq N^{-(1-\al)}}(t_{l+1}-t_l)^{p/2})^{4M/p}\nonumber\\
&\leq 2^{2M}(\bbJ_N^{(4)}+\tilde C^{2M}TN^{-2M(1-\al)(1-\frac 2p)})\nonumber
\end{eqnarray}
where we set
\[
\bbJ_N^{(4)}=E\sup_{0=t_0<t_1<...<t_m=T}(\sum_{l:\, t_{l+1}-t_l\leq N^{-(1-\al)}}|\hat\bbS_N^{ij}(t_l,t_{l+1})|^{p/2})
^{4M/p}
\]
and take into account that $p>2$ and  
\[
\sum_{l:\, t_{l+1}-t_l\leq N^{-(1-\al)}}(t_{l+1}-t_l)^{p/2}\leq TN^{-\frac p2(1-\al)(1-\frac 2p)}.
\]

Next, observe as in Section \ref{subsec4.1} that if $[t_lN]=[t_{l+1}N]$ then $\hat\bbS_N^{ij}(t_l,t_{l+1})=0$,
and so it suffices to consider the maximal subsequence $0=t_{l_0}<t_{l_1}<...<t_{l_n}=T$ of $t_0,t_1,...,t_m$
such that $[t_{l_r}N]<[t_{l_{r+1}}N]$, $r=0,1,...,n-1$ and setting $k_r=[t_{l_r}N]$ we obtain
\[
\bbJ_N^{(4)}=E\max_{0\leq k_0<k_1<...<k_n\leq[TN]}(\sum_{l:\,k_{l+1}-k_l\leq N^\al}|\hat\bbS_N^{ij}
(\frac {k_l}N,\frac {k_{l+1}}N)|^{p/2})^{4M/p}.
\]
Applying (\ref{3.4}) of Lemma \ref{lem3.2} to each $\hat\bbS_N^{ij}(\frac {k_l}N,\frac {k_{l+1}}N)$ and then using
Proposition \ref{hoelder2} with $\bfM=4M$ we see that
\begin{eqnarray}\label{4.17}
&\bbJ_N^{(4)}\leq\max_{0\leq k_0<k_1<...<k_n\leq[TN]}(\sum_{l:\,k_{l+1}-k_l\leq N^\al}(\frac {k_{l+1}-k_l}N)^{\frac p2(1-\be)})^{4M/p}\\
&\times E\max_{k,l:\, 0\leq k<l\leq TN}(|\hat\bbS_N^{ij}(\frac kN,\frac lN)|^{2M}|\frac {l-k}N|^{-2M(1-\be)})\nonumber\\
&\leq K_{\bbS,\be}(M)T^{\frac {4M}p}N^{-2M(1-\al)(1-\frac 2p-\be)}\nonumber
\end{eqnarray}
where $K_{\bbS,\be}(M)>0$ does not depend on $N$ and we choose $\be<1-\frac 2p$ which is possible since $p>2$.

Taking into account that $\sum_{0\leq l\leq m}(t_{l+1}-t_l)=T$ and relying on Theorem \ref{hoelder} we obtain
similarly to the above that
\begin{eqnarray}\label{4.18}
&\bbJ_N^{(3)}\leq\sup_{u,v:\, 0\leq u<v\leq T}(\sum_{l:\, t_{l+1}-t_l\leq N^{-(1-\al)}}|t_{l+1}-t_l|^{\frac p2(1-\be)})^{4M/p}\\
&\times E\sup_{u,v:\, 0\leq u<v\leq T}(|\bar\bbW_N^{ij}(u,v)|^{2M}|v-u|^{-2M(1-\be)})\nonumber\\
&\leq K_{\bbW,\be}(M)T^{4M/p}N^{-2M(1-\al)(1-\frac 2p-\be)}\nonumber
\end{eqnarray}
where $K_{\bbW,\be}(M)>0$ does not depend on $N$. Finally, (\ref{2.8}) follows for $\nu=2$ from
(\ref{4.12})--(\ref{4.18}).

\section{Continuous time case: straightforward setup}\label{sec5}\setcounter{equation}{0}
\subsection{Discretization}\label{subsec5.1}
Set $\eta(k)=\int_{k}^{k+1}\xi(s)ds,\, k=0,1,...,[NT]-1$,
\[
Z_N(s,t)=N^{-1/2}\sum_{sN\leq k<[tN]}\eta(k)\,\,\,\mbox{and}\,\,\,\bbZ_N^{ij}(s,t)=N^{-1}\sum_{sN\leq l<k<[tN]}\eta_i(l)\eta_j(k)
\]
where the sums over empty set are considered to be zero and, again, $Z_N(t)=Z_N(0,t)$, $\bbZ^{ij}_N(t)=\bbZ_N^{ij}(0,t)$. Observe that for $a\geq 1$,
\begin{eqnarray*}
&\|\eta(k)-E(\eta(k)|\cF_{k-l,k+l})\|_a\leq\int_k^{k+1}\|\xi(s)-E(\xi(s)|\cF_{k-l,k+l})\|_ads\\
&\leq 2\int_k^{k+1}\|\xi(s)-E(\xi(s)|\cF_{k-l+1,k+l-1})\|_ads\leq 2\be(a,l-1)
\end{eqnarray*}
where $\be$ defined by (\ref{2.15}) satisfies (\ref{2.4}) by the assumption of Theorem \ref{thm2.2}. Hence,
similarly to Lemma \ref{lem3.3} the limits
\begin{eqnarray*}
&\lim_{N\to\infty}N^{-1}\sum_{0\leq l<k<[tN]}E(\eta_i(l)\eta_j(k))=t\sum_{k=1}^\infty E(\eta_i(0)\eta_j(k))\\
&=\lim_{N\to\infty}N^{-1}\int_0^{tN}du\int_0^{[u]}E(\xi_i(v)\xi_j(u))dv=t(\int_0^\infty E(\xi_i(0)\xi_j(u))du-EF_{ij})
\end{eqnarray*}
exist, where $F_{ij}=\int_0^1du\int_0^u\xi_i(v)\xi_j(u)dv$, and
\[
|N^{-1}\sum_{0\leq l<k<[tN]}E(\eta_i(l)\eta_j(k))-\sum_{1\leq k<[tN]} E(\eta_i(0)\eta_j(k))|=O(N^{-1}).
\]

Since the sequence $\eta(k),\, k\geq 0$ satisfies the conditions of Theorem \ref {thm2.1}, it follows
that the process $\xi$ can be redefined preserving its distributions on a sufficiently rich probability space which
contains also a $d$-dimensional Brownian motion $\cW$ with the covariance matrix $\vs$ so that the rescaled Brownian motion
$W_N(t)=N^{-1/2}\cW(Nt)$ satisfies
\begin{equation}\label{5.1}
E\| Z_N-W_N\|^{4M}_{p,[0,T]}\leq C_Z(M)N^{-\ve}\,\,\mbox{and}\,\,E\| \bbZ_N-\bbW^\eta_N\|^{2M}_{\frac p2,[0,T]}\leq C_\bbZ(M)N^{-\ve}
\end{equation}
where
\[
\bbW_N^\eta(s,t)=\int_s^tW_N(s,v)\otimes dW_N(v)+(t-s)\sum_{l=1}^\infty E(\eta(0)\otimes\eta(l))=\bbW_N(s,t)-(t-s)F,
\]
$ F=(F_{ij})$, $\bbW_N=\bbW_N^{(2)}$ is defined as in Section \ref{subsec2.1} but with the matrix $\Gam$ defined in Section \ref{subsec2.2}, 
$p\in(2,3)$, $\ve>0$ and $C_Z(M),\, C_\bbZ(M)>0$ do not depend on $N$.
In fact, Theorem \ref{thm2.1} gives directly only that the sequence $\eta(k),\, k\geq 0$ above can be redefined so that $Z_N$ and $\bbZ_N$
constructed by this sequence satisfy (\ref{5.1}). But if $\eta(k),\, k\geq 0$ is the above sequence and $\tilde\eta(k),\,k\geq 0$
is the redefined sequence with $Z_N$ and $\bbZ_N$ constructed by it satisfy (\ref{5.1}) then we have two pairs of processes $(\xi,\eta)$ and $(\tilde\eta,\cW)$. 
Since the second marginal of the first pair has the same distribution as the first
 marginal of the second pair, we
can apply Lemma A1 of \cite{BP} which yields three processes $(\hat\xi,\hat\eta,\hat\cW)$ such that the joint distributions of
 $(\hat\xi,\hat\eta)$ and of $(\hat\eta,\hat\cW)$ are the same as of $(\xi,\eta)$ and of $(\tilde\eta,\cW)$, respectively. Hence,
 we have (\ref{5.1}) for $\hat\xi,\hat\eta,\hat\cW$ in place of $\xi,\eta,\cW$ and we can and will assume in what follows that (\ref{5.1})
 holds true for $\xi,\eta,\cW$ with $Z,\bbZ,W_N$ and $\bbW_N$ defined above. Thus, in order to derive Theorem \ref{thm2.2} for $\nu=1$
 and $\nu=2$ in this setup it remains to estimate 
 \begin{eqnarray*}
 &E\| S_N-W_N\|^{4M}_{p,[0,T]}=E\| Z_N-W_N\|^{4M}_{p,[0,T]}\,\,\,\mbox{and}\\
 & E\|\bbS_N-\bbW^\eta_N\|^{2M}_{p/2,[0,T]}\leq 2^{4M-1}(E\|\bbS_N-\bbZ_N-EF\|^{2M}_{p/2,[0,T]}+E\|\bbZ_N-\bbW^\eta_N\|^{2M}_{p/2,[0,T]})
 \end{eqnarray*}
  where
 \begin{eqnarray*}
& S_N(t)=\bbS_N^{(1)}(t)=N^{-1/2}\int_0^{tN}\xi(s)ds\,\,\,\mbox{and}\,\,\,\bbS_N=(\bbS^{ij}_N,\,1\leq i,j\leq d)\\
 &\mbox{with}\,\,\,\bbS^{ij}_N(t)=N^{-1}\int_0^{tN}\xi_j(u)du\int_0^u\xi_i(v)dv.
 \end{eqnarray*}

 \subsection{Estimates for integrals}\label{subsec5.2}

 First, observe that
 \[
 \sup_{0\leq t\leq T}|S_N(t)-Z_N(t)|\leq N^{-1/2}\max_{0\leq k<N}\int_k^{k+1}|\xi(s)|ds,
 \]
 and so
 \begin{equation}\label{5.2}
 E\sup_{0\leq t\leq T}|S_N(t)-Z_N(t)|^{4M}\leq N^{-2M+1}E|\xi(0)|^{4M}.
 \end{equation}
 Next, let $0=t_0<t_1<...<t_m=T$ and set
 \[
 I_i=|Z_N(t_{i+1})-S_N(t_{i+1})-(Z_N(t_i)-S_N(t_i))|.
 \]
 Then
 \[
 \sum_{0\leq i<m}I_i^p\leq\cI_1+\cI_2
 \]
 where
 \[
 \cI_1=\sum_{i:\,[t_iN]<[t_{i+1}N]}I_i^p\,\,\,\mbox{and}\,\,\,\cI_2=\sum_{i:\,[t_iN]=[t_{i+1}N]}I_i^p.
 \]

 Observe that the sum in $\cI_1$ contains no more than $[NT]$ terms, and so
 \begin{equation}\label{5.3}
 \cI_1\leq NT2^p\sup_{0\leq t\leq T}|Z_N(t)-S_N(t)|^p.
 \end{equation}
 Since $\sum_{0\leq i<m}|t_{i+1}-t_i|=T$ and $I_i=N^{-1/2}|\int_{t_iN}^{t_{i+1}N}\xi(s)ds|$ if $[t_iN]=[t_{i+1}N]$, we
 obtain for any $p\in(2,3)$ that
 \begin{eqnarray}\label{5.4}
 &\cI_2=N^{-p/2}\sum_{i:\,[t_iN]=[t_{i+1}N]}|\int_{t_iN}^{t_{i+1}N}\xi(s)ds|^p\\
 &\leq N^{-p/2}\sum_{i:\,[t_iN]=[t_{i+1}N]}(t_{i+1}-t_i)^{p-1}N^{p-1}\int_{[t_iN]}^{[t_iN]+1}|\xi(s)|^pds\nonumber\\
 &\leq N^{-\frac p2+1}\max_{0\leq k<NT}\int_k^{k+1}|\xi(s)|^pds.\nonumber
 \end{eqnarray}
 It follows from (\ref{5.2})--(\ref{5.4}) that for $M>\frac p{p-2}$,
 \begin{eqnarray}\label{5.5}
 &E\| Z_N-S_N\|^{4M}_{p,[0,T]}\leq 2^{\frac {4M}p-1}(E\sup_{0=t_0<t_1<...<t_m=T}\cI_1^{4M/p}\\
 &+E\sup_{0=t_0<t_1<...<t_m=T}\cI_2^{4M/p})\leq 2^{4M(\frac 1p+1)-1}N^{-2M(1-\frac 2p)+1}E|\xi(0)|^{4M}\nonumber\\
 &+N^{-2M(1-\frac 2p)}T^{\frac {4M}p-1}E\max_{0\leq k<NT}\int_k^{k+1}|\xi(s)|^{4M}ds.
 \nonumber\end{eqnarray}
Since
\[
 E\max_{0\leq k<NT}\int_k^{k+1}|\xi(s)|^{4M}ds\leq\sum_{0\leq k<NT}\int_k^{k+1}E|\xi(s)|^{4M}ds\leq NTE|\xi(0)|^{4M},
 \]
 we obtain from (\ref{5.5}) that
 \begin{equation}\label{5.6}
 E\| S_N-Z_N\|^{4M}_{p,[0,T]}\leq C_{SZ}(M)N^{-2M(1-\frac 2p)+1}
 \end{equation}
 where $C_{SZ}(M)=2^{\frac {4M}p-1}E|\xi(0)|^{4M}(2^{4M}T^{4M}+T^{4M/p})$ which proves Theorem \ref{thm2.2} for $\nu=1$ since we
 assume that $M>\frac p{p-2}$.

 \subsection{Estimates for iterated integrals}\label{subsec5.3}

First, we write
\begin{eqnarray}\label{5.7}
&\bbS^{ij}_N(t)=\frac 1N\int_0^{tN}\xi_j(u)du\int_0^u\xi_i(v)dv=\frac 1N\int_0^{[tN]}\xi_j(u)du\int_0^{[u]}\xi_i(v)dv\\
&+\frac 1N\int_0^{[tN]}\xi_j(u)du\int_{[u]}^u\xi_i(v)dv\nonumber\\
&+\frac 1N\int_{[tN]}^{tN}\xi_j(u)du\int_0^u\xi_i(v)dv=\bbZ^{ij}_N(t)+\frac 1N\sum_{k=0}^{[tN]-1}F_{ij}\circ\vt^k+J_N(t)\nonumber
\end{eqnarray}
where $F_{ij}\circ\vt^k=\int_k^{k+1}du\int_k^u\xi_i(v)\xi_j(u)dv$ and $J_N(t)=\frac 1N\int_{[tN]}^{tN}\xi_j(u)du\int_0^{u}\xi_i(v)dv$. Observe that
$E(F_{ij}\circ\vt^k)=EF_{ij}$ for all $k\geq 0$ and
\begin{eqnarray}\label{5.8}
&J_N(t)\leq \frac 1N\int_{[tN]}^{tN}|\xi_j(u)|du\sup_{0\leq u\leq tN}|\int_0^u\xi_i(v)dv|\\
&\leq\frac 1N\max_{0\leq k\leq[tN]}\int_k^{k+1}|\xi_j(u)|du(\max_{0\leq k\leq[tN]}|\sum_{l=0}^k\eta_i(l)|\nonumber\\
&+\max_{0\leq k\leq[tN]}\int_k^{k+1}|\xi_i(v)|dv).
\nonumber\end{eqnarray}
It follows from (\ref{5.7}) and (\ref{5.8}) together with Lemma \ref{lem3.2} and the Cauchy--Schwarz inequality that for any $t\in [0,T]$,
\begin{eqnarray}\label{5.9}
&E\sup_{0\leq s\leq t}|\bbS^{ij}_N(s)-sEF_{ij}-\bbZ_N^{ij}(s)|^{2M}\\
&\leq 2^{4M}N^{-2M}\big(\sum_{0\leq k\leq[tN]}E(\int_k^{k+1}|\xi_j(u)|du)^{4M}\big)^{1/2}\nonumber\\
&\times\big(E\max_{0\leq k\leq[tN]}|\sum_{l=0}^k\eta_i(l)|^{4M}\nonumber\\
&+\sum_{0\leq k\leq[tN]}E(\int_k^{k+1}|\xi_i(v)|dv)^{4M}\big)^{1/2}\nonumber\\
&+2^{2M}N^{-2M}E\max_{0\leq k<[tN]}|\sum_{l=0}^k(F_{ij}\circ\vt^l-EF_{ij})|^{2M}\nonumber\\
&\leq C_{\bbS F\bbZ}(M)(N^{-M+\frac 12}t^{M+\frac 12}+N^{-2M+1}t+N^{-M}t^M)
\nonumber\end{eqnarray}
where $C_{\bbS F\bbZ}(M)>0$ does not depend on $N$ and $t$.

Next we observe that
\begin{eqnarray}\label{5.10}
&\bbS_N^{ij}(s,t)=\bbS_N^{ij}(t)-\bbS_N^{ij}(s)-S_N^i(s)(S^j_N(t)-S^j_N(s))\\
&\mbox{and}\,\,\,\bbZ_N^{ij}(s,t)=\bbZ_N^{ij}(t)-\bbZ_N^{ij}(s)-Z_N^i(s)(Z^j_N(t)-Z^j_N(s)).\nonumber
\end{eqnarray}
 By (\ref{3.3}) of Lemma \ref{lem3.2} together with (\ref{5.2}),
\begin{equation}\label{5.11}
E\sup_{0\leq t\leq T}|Z^i_N(t)|^{4M}\leq\tilde C_Z(M)\,\,\mbox{and}\,\, E\sup_{0\leq t\leq T}|S^i_N(t)|^{4M}\leq\tilde C_S(M)
\end{equation}
where $\tilde C_Z(M),\tilde C_S(M)>0$ do not depend on $N$ and $i$. This together with (\ref{5.2}), (\ref{5.9}), (\ref{5.10}) and the Cauchy--Schwarz
inequality imply that
\begin{eqnarray}\label{5.12}
&E\sup_{0\leq s\leq t\leq T}|\bbS^{ij}_N(s,t)-(t-s)EF_{ij}-\bbZ_N^{ij}(s,t)|^{2M}\\
&\leq 6^{2M}\big(E\sup_{0\leq t\leq T}|\bbS^{ij}_N(t)-tEF_{ij}-\bbZ_N^{ij}(t)|^{2M}\nonumber\\
&+(E\sup_{0\leq t\leq T}|S_N^i(t)|^{4M})^{1/2} (E\sup_{0\leq t\leq T}|S^j_N(t)-Z^j_N(t)|^{4M})^{1/2}\nonumber\\
&+(E\sup_{0\leq t\leq T}|Z_N^j(t)|^{4M})^{1/2} (E\sup_{0\leq t\leq T}|S^i_N(t)-Z^i_N(t)|^{4M})^{1/2}\big)\nonumber\\
&\leq\bbC_{\bbS F\bbZ}(M)N^{-M+\frac 12}\nonumber
\end{eqnarray}
where $\bbC_{\bbS F\bbZ}(M)>0$ does not depend on $N$.

\subsection{Variational norm estimates for iterated integrals}\label{subsec5.4}

Observe that by (\ref{3.4}) of Lemma \ref{lem3.2},
\begin{equation}\label{5.13}
E|\bbZ_N^{ij}(s,t)-E\bbZ_N^{ij}(s,t)|^{2M}\leq\bbC_\bbZ(M)(t-s)^{2M}
\end{equation}
where $\bbC_\bbZ(M)>0$ does not depend on $N$ and $t\geq s\geq 0$. Next, by the stationarity of the sequence $\eta(n)$,
\[
E\bbZ_N^{ij}(s,t)=N^{-1}\sum_{sN\leq k\leq[tN]}E(\eta_i(0)\eta_j(k-l)),
\]
and so
\[
|E\bbZ_N^{ij}(s,t)|\leq (t-s)\sum_{k=1}^\infty|E(\eta_i(0)\eta_j(k)|
\]
and using the same estimates as in (\ref{3.14}) we conclude that the series in the right hand side above converges. This together
with (\ref{5.13}) yields
\begin{equation}\label{5.14}
E|\bbZ^{ij}_N(s,t)|^{2M}\leq\tilde\bbC_\bbZ(M)(t-s)^{2M}
\end{equation}
where $\tilde\bbC_\bbZ(M)>0$ does not depend on $N$ and $t\geq s\geq 0$.

Since $\xi$ is a stationary process
\[
E|\bbS_N^{ij}(s,t)-(t-s)EF_{ij}|^{2M}=E|\bbS_N^{ij}(0,t-s)-(t-s)EF_{ij}|^{2M},
\]
and so by (\ref{5.9}) and (\ref{5.14}),
\begin{equation}\label{5.15}
E|\bbS_N^{ij}(s,t)-(t-s)EF_{ij}|^{2M}\leq\bbC_{\bbS F}(M)(t-s)^{2M},
\end{equation}
provided $t-s\geq 1/N$, where $\bbC_{\bbS F}(M)>0$ does not depend on $N$ and $t\geq s\geq 0$.
When $0\leq t-s<N^{-1}$ we will use the simple estimate
\[
|\bbS^{ij}_N(s,t)|\leq N^{-1}(\int_{sN}^{tN}|\xi_i(u)|du)(\int_{sN}^{tN}|\xi_j(v)|dv),
\]
and so by the H\" older inequality,
\begin{eqnarray}\label{5.16}
&\,\,\,\,\, E|\bbS^{ij}_N(s,t)|^{2M}\leq N^{-2M}(E(\int_{sN}^{tN}|\xi_i(u)|du)^{4M})^{1/2}(E(\int_{sN}^{tN}|\xi_j(v)|dv)^{4M})^{1/2}\\
&\leq N^{2M}(t-s)^{4M}(E|\xi_i(0)|^{4M})^{1/2}(E|\xi_j(0)|^{4M})^{1/2}\nonumber\\
&\leq(t-s)^{2M}(E|\xi_i(0)|^{4M})^{1/2}(E|\xi_j(0)|^{4M})^{1/2}.\nonumber
\end{eqnarray}

Now we are ready to estimate
\begin{eqnarray}\label{5.17}
&E\|\bbS_N^{ij}(s,t)-(t-s)EF_{ij}-\bbZ_N^{ij}(s,t)\|^{2M}_{p/2,[0,T]}\\
&=E\sup_{0=t_0<t_1<...<t_m=T}\big(\sum_{0\leq l<m}|\bbS_N^{ij}(t_l,t_{l+1})-(t_{l+1}-t_l)EF_{ij}\nonumber\\
&-\bbZ^{ij}_N(t_l,t_{l+1})|^{p/2}\big)^{4M/p}\nonumber\\
&\leq 3^{4M/p}(E\cJ_1^{4M/p}+2^{2M}(E\cJ_2^{4M/p}+E\cJ_3^{4M/p})\nonumber
\end{eqnarray}
where
\begin{eqnarray*}
&\cJ_1=\sup_{0=t_0<t_1<...<t_m=T}\big(\sum_{l:t_{l+1}-t_l\geq N^{-(1-\al)}}|\bbS_N^{ij}(t_l,t_{l+1})\\
&-(t_{l+1}-t_l)EF_{ij}-\bbZ^{ij}_N(t_l,t_{l+1})|^{p/2}\big),
\end{eqnarray*}
\[
\cJ_2=\sup_{0=t_0<t_1<...<t_m=T}\big(\sum_{l:t_{l+1}-t_l<N^{-(1-\al)}}|\bbS_N^{ij}(t_l,t_{l+1})-(t_{l+1}-t_l)EF_{ij}|^{p/2}\big),
\]
\[
\mbox{and}\,\,\,\cJ_3=\sup_{0=t_0<t_1<...<t_m=T}\big(\sum_{l:t_{l+1}-t_l<N^{-(1-\al)}}|\bbZ^{ij}_N(t_l,t_{l+1})|^{p/2}\big)
\]
where $\al<1$ will be specified later on.

Since there are no more than $TN^{1-\al}$ terms in the sum for $\cJ_1$ we obtain by (\ref{5.12}) that
\begin{eqnarray}\label{5.18}
&E\cJ_1^{4M/p}\leq T^{4M/p}N^{4M(1-\al)/p}E\sup_{0\leq s\leq t\leq T}|\bbS^{ij}_N(s,t)-(t-s)EF_{ij}\\
&-\bbZ_N^{ij}(s,t)|^{2M}\leq\bbC_{\bbS F\bbZ}(M)T^{4M/p}N^{-M(1-4(1-\al)p^{-1})+1}.\nonumber
\end{eqnarray}
Next, applying Theorem \ref{hoelder}, which is possible in view of (\ref{5.10}),  (\ref{5.15}) and (\ref{5.16}), and taking into account that $\sum_l(t_{l+1}-t_l)\leq T$
we obtain similarly to (\ref{4.6}) that,
\begin{eqnarray}\label{5.19}
&E\cJ_2^{4M/p}\leq E\sup_{0=t_0<t_1<...<t_m=T}\big(2^{p/2}\sum_{l:t_{l+1}-t_l<N^{-(1-\al)}}\\
&|t_{l+1}-t_l|^{p(1-\be)/2}\big(\frac {|\bbS_N^{ij}(t_l,t_{l+1})|^{p/2}}{|t_{l+1}-t_l|^{p(1-\be)/2}}+|t_{l+1}-t_l|^{p\be/2}|EF_{ij}|^{p/2}\big)\big)^{4M/p}\nonumber\\
&\leq 2^{4M}T^{4M/p}N^{-2M(1-\al)(1-\be-\frac 2p)}
\big(E\sup_{0\leq s<t\leq N^{-(1-\al)}}\frac {|\bbS_N^{ij}(t_l,t_{l+1})|^{2M}}{|t_{l+1}-t_l|^{2M(1-\be)}}\nonumber\\
&+N^{-2M(1-\al)\be}|EF_{ij}|^{2M}\big)\leq\bbC_{\be,\bbS}(M)N^{-M(1-\al)(1-\be-\frac 2p)}\nonumber
\end{eqnarray}
where $\be>0$ will be chosen later on.

Finally, applying Proposition \ref{hoelder2}, which is possible in view of (\ref{5.10}) and (\ref{5.14}), and taking into account that
$N^{-1}\sum_l|k_{l+1}-k_l|\leq T$  we obtain similarly to (\ref{4.5}) that,
 \begin{eqnarray}\label{5.20}
&E\cJ_3^{4M/p}\leq E\max_{0=k_0<k_1<...<k_m=[TN]}\sum_{l:k_{l+1}-k_l<N^{\al}}|\bbZ^{ij}_N(\frac {k_l}N,\frac {k_{l+1}}N)|\\
&\leq E\big(\max_{0=k_0<k_1<...<k_m=[TN]}(\sum_{l:k_{l+1}-k_l<N^{\al}}|\frac {k_{l+1}-k_l}N|^{p(1-\be)/2})^{4M/p})\nonumber\\
&\times\max_{0\leq k<m\leq[TN]}\frac {|\bbZ_N^{ij}(\frac kN,\frac mN)|^{2M}}{|\frac {m-k}N|^{2M(1-\be)}}\big)\leq\bbC_{\be,\bbZ}(M)N^{-2M(1-\al)(1-\be-\frac 2p)}.
\nonumber\end{eqnarray}
Choosing $\al$ and $\be$ so that $1>\al>1-\frac p4$ and $0<\be<1-\frac 2p$ we derive Theorem \ref{thm2.2} for $\nu=2$ from (\ref{5.1}) and
(\ref{5.17})--(\ref{5.20}) while its extension from $\nu=1$ and $\nu=2$ to $\nu>2$ will be provided in Section \ref{sec7}.

\section{Continuous time case: suspension construction}\label{sec6}\setcounter{equation}{0}

\subsection{Discretization}\label{subsec6.1}
Set again
\[
Z_N(s,t)=N^{-1/2}\sum_{sN\leq k<[tN]}\eta(k)\quad\mbox{and}
\]
\[
\bbZ_N^{ij}(s,t)=N^{-1}\sum_{sN\leq l<k<[tN]}\eta_i(l)\eta_j(k)
\]
with $\eta$ now defined in Section \ref{subsec2.3}. As before, we denote also $Z_N(t)=Z_N(0,t)$
and $\bbZ^{ij}_N(t)=\bbZ_N^{ij}(0,t)$. In view of (\ref{2.4}) and (\ref{2.18}) we see similarly to
Lemma \ref{lem3.3} that the limit
\[
\lim_{N\to\infty}N^{-1}\sum_{0\leq l<k<[tN]}E(\eta_i(l)\eta_j(k))=t\sum_{k=1}^\infty E(\eta_i(0)\eta_j(k))
\]
exists. Since the sequence $\eta$ satisfies the conditions of Theorem \ref{thm2.1},
we argue similarly to Section \ref{sec5} to conclude that the process $\xi$ can be redefined preserving its
distributions on a sufficiently rich probability space which contains also a $d$-dimensional Brownian motion
$\cW$ with the covariance matrix $\vs$ given by (\ref{2.19}) so that the rescaled Brownian motion $W_N(t)=
N^{-1/2}\cW(Nt)$ satisfies
\begin{equation}\label{6.1}
E\| Z_N-W_N\|^{4M}_{p,[0,T]}\leq C_Z(M)N^{-\ve}\,\,\mbox{and}\,\,E\| \bbZ_N-\bbW^\eta_N\|^{2M}_{\frac p2,[0,T]}\leq C_\bbZ(M)N^{-\ve}
\end{equation}
where
\begin{eqnarray}\label{6.1+}
&\bbW_N^\eta(s,t)=\int_s^tW(s,v)\otimes dW_N(v)+(t-s)\sum_{l=1}^\infty E(\eta(0)\otimes\eta(l))\\
&=\bbW_N(s,t)-(t-s)F,\, F=(F_{ij}),\, F_{ij}=E\int_0^{\tau(\om)}\xi_j(s,\om)ds\int_0^s\xi_i(u,\om)du,\nonumber
\end{eqnarray}
$\bbW_N=\bbW_N^{(2)}$ is defined as in Section \ref{subsec2.1} but with the matrix $\Gam$ defined in Section \ref{subsec2.3}, $p\in(2,3)$,
 $\ve>0$ and $C_Z(M),\, C_\bbZ(M)>0$ do not depend on $N$. We observe also that
\[
E\bbW_N(t)=\Gam t=\lim_{N\to\infty}E\bbS_N(t).
\]

Next, define $n(s)=n(s,\om)=0$ if $\tau(\om)>s$ and for $s\geq\tau(\om)$,
\[
n(s)=n(s,\om)=\max\{ k:\,\sum_{j=0}^{k-1}\tau\circ\vt^j(\om)\leq s\}.
\]
It was shown in Lemma 5.1 from \cite{FK} that for all $s,t\geq 0$,
\begin{equation}\label{6.2}
E|n(s\bar\tau)-s|^{4M}\leq K(M)s^{2M}\,\,\mbox{and}\,\, E\sup_{0\leq s\leq t}|n(s
\bar\tau)-s|^{4M}\leq K(M)t^{2M+1}
\end{equation}
where $K(M)>0$ does not depend on $s,t$ and, recall, $\tau$ is the ceiling function and $\bar\tau=E\tau$. In view of
the boundedness assumption (\ref{2.17}) the proof in our circumstances is the same as in \cite{FK}, only in place
of Lemma 3.4 there which relied on $\phi$-mixing we use Lemma \ref{lem3.2} here which is proved under our more general assumptions.
Now define
\begin{eqnarray*}
&U_N(s,t)=N^{-1/2}\sum_{n(s\bar\tau N)\leq k<n(t\bar\tau N)}\eta(k)\\
&\mbox{and}\,\,\bbU_N^{ij}(s,t)=N^{-1}\sum_{n(s\bar\tau N)\leq k<l<n(t\bar\tau N)}\eta_i(k)\eta_j(l)
\end{eqnarray*}
setting again $U_N(t)=U_N(0,t)$ and $\bbU^{ij}_N(t)=\bbU^{ij}(0,t)$. Next, we will compare $S_N(s,t)=\bbS_N^{(1)}(s,t)$
and $\bbS_N(s,t)=\bbS_N^{(2)}(s,t)$ with $U_N(s,t)$ and $\bbU_N(s,t)$ and then with $Z_N(s,t)$ and $\bbZ_N(s,t)$,
respectively, which in view of (\ref{6.1}) and (\ref{6.1+}) will provide the required variational norms estimates for $S_N-W_N$ and
for $\bbS_N-\bbW_N$.

\subsection{Estimates for sums and integrals}\label{subsec6.2}

Our strategy here will be similar to Section \ref{subsec5.2} above but the situation here is more complicated
in view of the need of two discretization approximations $\bbZ_N$ and $\bbU_N$ with the second one containing the
random number of terms. Introduce
\[
\hat S_N(s,t,(\om,u))=N^{-1/2}\int_{s\bar\tau N}^{t\bar\tau N}\xi(v,(\om,u))dv=N^{-1/2}
\int_{s\bar\tau N+u}^{t\bar\tau N+u}\xi(v,(\om,0))dv
\]
where $0\leq u<\tau(\om)$, Since $S_N(s,t)=S_N(s,t,\om)=\hat S_N(s,t,(\om,0))$, we have
\begin{equation}\label{6.3}
|S_N(s,t,\om)-\hat S_N(s,t,(\om,u))|\leq 2N^{-1/2}\max_{sN\leq k< tN}\int_{k\bar\tau}^{(k+1)\bar\tau+\hat L} |\xi(v,\om)|dv.
\end{equation}
Since $\xi$ is a stationary process on the probability space $(\hat\Om,\hat\cF,\hat P)$ it follows that
\[
\int|\hat S_N(s,t,(\om,u))|^{4M}d\hat P(\om,u)=\int|\hat S_N(0,t-s,(\om,u))|^{4M}d\hat P(\om,u).
\]
This together with (\ref{2.17}) and (\ref{6.3}) yields that
\begin{eqnarray}\label{6.4}
&E|S_N(s,t)|^{4M}\leq 2^{4M-1}\hat L(\int|\hat S_N(0,t-s,(\om,u))|^{4M}d\hat P(\om,u)\\
&+2^{4M}\hat LN^{-M}E\max_{sN\leq k< tN}(\int_{k\bar\tau}^{(k+1)\bar\tau+\hat L}|\xi(v)|dv)^{4M})
\nonumber\\
&\leq 2^{8M-2}\hat L^2(E|S_N(t-s)|^{4M}\nonumber\\
&+N^{-2M}(2^{4M}\hat L+2)E\max_{sN\leq k< tN}(\int_{k\bar\tau}^{(k+1)\bar\tau+\hat L}|\xi(v)|dv)^{4M}).\nonumber
\end{eqnarray}

Set
\[
\sig(s)=\sig(s,\om)=\sum_{l=0}^{n(s)-1}\tau\circ\vt^l(\om)
\]
and observe that
\begin{eqnarray*}
&S_N(t)=N^{-1/2}\int_0^{tN\bar\tau}\xi(u)du=N^{-1/2}\sum_{0\leq k<[n(tN\bar\tau]} \eta(k)\\
&+N^{-1/2}\int_{\sig(tN\bar\tau)}^{tN\bar\tau}\xi(u)du=U_N(t)+N^{-1/2}\int_{\sig(tN\bar\tau)}^{tN\bar\tau}\xi(u)du.
\end{eqnarray*}
Hence,
\begin{equation}\label{6.5}
|S_N(t)-U_N(t)|\leq N^{-1/2}\max_{0\leq k< tN}\int_{k\bar\tau}^{(k+1)\bar\tau+\hat L}|\xi(v)|dv,
\end{equation}
and so by (\ref{6.4}),
\begin{eqnarray}\label{6.6}
&E|S_N(s,t)|^{4M}\leq 2^{12M-3}\hat L^2(E|U_N(t-s)|^{4M}\\
&+2N^{-2M}(\hat L+2)E\max_{sN\leq k< tN}(\int_{k\bar\tau}^{(k+1)\bar\tau+\hat L}|\xi(v)|dv)^{4M}).\nonumber
\end{eqnarray}

In order to estimate the right hand side of (\ref{6.6}) we observe that by (\ref{2.17}),
\[
\hat L^{-2}tN\leq n(t\bar\tau N)\leq\hat L^2tN,
\]
and so
\begin{eqnarray}\label{6.7}
&\quad\,\,\,|U_N(t)-Z_N(t)|^{4M}\leq N^{-2M}|\sum_{\min([tN],n(t\bar\tau N))\leq k<\max([tN],n(t\bar\tau N))}\eta(k)|^{4M}\\
&\leq N^{-2M}\max_{0\leq m\leq Nt\hat L^2}\max_{0\leq k<|n(t\bar\tau N)-[tN]|}|\sum_{l=0}^k\eta(l+m)|^{4M}\nonumber\\
&\leq N^{-2M}\big(\sum_{0\leq m\leq Nt\hat L^2}\max_{0\leq k<(tN)^{\frac 12+\gam}}|\sum_{l=0}^k\eta(l+m)|^{4M}\nonumber\\
&+\sum_{0\leq m\leq Nt\hat L^2}\max_{(tN)^{\frac 12+\gam}\leq k<tN\hat L^2}|\sum_{l=0}^k\eta(l+m)|^{4M}
\bbI_{\{|n(t\bar\tau N)-[tN]|\geq(tN)^{\frac 12+\gam}\}}\big)\nonumber
\end{eqnarray}
where $\bbI_\Gam$ denotes the indicator of an event $\Gam$ and $\gam>0$ will be chosen later on. Applying Lemma \ref{lem3.2} we obtain
\[
E\max_{0\leq k<(tN)^{\frac 12+\gam}}|\sum_{l=0}^k\eta(l+m)|^{4M}\leq C^\eta(M)(tN)^{M(1+2\gam)}
\]
where $C^\eta(M)>0$ does not depend on $N$ and $t>0$. Taking into account that
\[
|n(t\bar\tau N)-[tN]|\leq|n(t\bar\tau N)-tN|+|tN-[tN]|\leq|n(t\bar\tau N)-tN|+(tN)^\gam
\]
when $1>\gam>0$, we derive from here together with (\ref{6.2}), the Cauchy--Schwarz and Chebyshev inequalities that
\begin{eqnarray*}
&E\max_{0\leq k<tN\hat L^2}|\sum_{l=0}^k\eta(l+m)|^{4M}\bbI_{\{|n(t\bar\tau N)-[tN]|\geq(tN)^{\frac 12+\gam}\}}\\
&\leq(E\max_{0\leq k<tN\hat L^2}|\sum_{l=0}^k\eta(l+m)|^{4M})^{1/2}(P\{|n(t\bar\tau N)-[tN]|\geq(tN)^{\frac 12+\gam}\})^{1/2}\\
&\leq\tilde C^\eta(M)(tN)^{-2M\gam}(E|n(t\bar\tau N)-[tN]|^{4M})^{1/2}\leq \tilde C^\eta(M)2^{2M}(\sqrt {K(M)}(tN)^{M(1-2\gam)}+1)
\end{eqnarray*}
for some $\tilde C^\eta(M)>0$ which does not depend on $N$ and $t>0$. It follows that
\begin{eqnarray}\label{6.8}
&E|U_N(t)-Z_N(t)|^{4M}\leq C_{UZ}(M)(t^{M(1+2\gam)+1}N^{-M(1-2\gam)+1}\\
&+t^{M(1-2\gam)+1}N^{-M(1+2\gam)+1}+N^{-2M})\nonumber
\end{eqnarray}
for some $C_{UZ}(M)>0$ which does not depend on $N$ and $t>0$.

Applying Lemma \ref{lem3.2} to sums of $\eta(k)$'s we obtain that for all $0\leq s<t<T$,
\begin{equation}\label{6.9}
E|Z_N(s,t)|^{4M}\leq C^\eta_Z(M)(t-s)^{2M}
\end{equation}
where $C^\eta_Z(M)>0$ does not depend on $N,s$ and $t$.
By the stationarity of the process $\xi$ on the probability space $(\hat\Om,\hat\cF,\hat P)$ for any $s\geq 0$,
\begin{eqnarray}\label{6.10}
&E\int_0^\tau|\xi(u)|^{4M}du=\int_{\hat\Om}|\xi(0,(\om,u))|^{4M}d\hat P(\om,u)\\
&=\int_{\hat\Om}|\xi(s,(\om,u)|^{4M}d\hat P(\om,u)\geq\int_s^{s+\hat L^{-1}}E|\xi(v)|^{4M}dv\nonumber
\end{eqnarray}
where the last inequality follows from (\ref{2.17}). Hence,
\begin{eqnarray}\label{6.11}
&\int_{k\bar\tau}^{(k+1)\bar\tau+\hat L}E|\xi(v)|^{4M}dv\\
&\leq\sum_{0\leq l\leq\hat L(\bar\tau+\hat L)}\int_{k\bar\tau+l\hat L^{-1}}^{k\bar\tau+(l+1)\hat L^{-1}}E|\xi(v)|^{4M}dv\nonumber\\
&\leq (\hat L(\bar\tau+\hat L)+1)E\int_0^\tau|\xi(v)|^{4M}dv<\infty\nonumber
\end{eqnarray}

%Observe also that
%\begin{eqnarray}\label{6.10}
%&E\max_{sN\leq k< tN}(\int_{k\bar\tau}^{(k+1)\bar\tau+\hat L}|\xi(v)|dv)^{4M}\\
%&\leq (\bar\tau+\hat L)^{4M-1}\sum_{sN\leq k< tN}\int_{k\bar\tau}^{(k+1)\bar\tau+\hat L}E|\xi(v)|^{4M}dv.
%\nonumber\end{eqnarray}
%By the stationarity of the process $\xi$ on the probability space $(\hat\Om,\hat\cF,\hat P)$,
%\begin{eqnarray}\label{6.11}
%&\int_{k\bar\tau}^{(k+1)\bar\tau+\hat L}E|\xi(v)|^{4M}dv\leq E\int_{\sig(n(k\bar\tau))}^{\sig(n((k+1)\bar\tau+\hat L)+1)}|\xi(v)|^{4M}dv\\
%&=E((n((k+1)\bar\tau +\hat L)+1-n(k\bar\tau))\int_0^\tau|\xi(v)|^{4M}dv\nonumber\\
%&\leq (\hat L(\bar\tau+\hat L)+1)E\int_0^\tau|\xi(v)|^{4M}dv<\infty\nonumber
%\end{eqnarray}
%since $ n((k+1)\bar\tau +\hat L)+1-n(k\bar\tau)\leq \hat L(\bar\tau+\hat L)+1$ by (\ref{2.17}).

Now assume that $t-s>1/N$ and take $\gam>0$ in (\ref{6.8}) so small that $M(1-2\gam)>1$. Then combining (\ref{6.6})--(\ref{6.11}) we conclude that
\begin{equation}\label{6.12}
E|S_N(s,t)|^{4M}\leq C_{S}(M)(t-s)^{2M}
\end{equation}
provided $t-s>1/N$ for some $C_S(M)>0$ which does not depend on $N$. For $|t-s|\leq N^{-1}$ we will use (\ref{6.11}) and the Jensen inequality
to obtain
\begin{eqnarray}\label{6.13}
&\quad E|S_N(s,t)|^{4M}\leq N^{-2M}E(\int_{s\bar\tau N}^{t\bar\tau N}|\xi(u)|du)^{4M}\leq N^{2M-1}(t-s)^{4M-1}\bar\tau^{4M-1}\\
&\times\int_{s\bar\tau N}^{t\bar\tau N}E|\xi(u)|^{4M}du\leq(t-s)^{2M}\bar\tau^{4M-1}\int_{s\bar\tau N}^{t\bar\tau N}E|\xi(u)|^{4M}du\leq\tilde C_S(M)(t-s)^{2M},
\nonumber\end{eqnarray}
where $\tilde C_S(M)>0$ does not depend on $N,s$ and $t$,
which together with (\ref{6.10}) and (\ref{6.11}) yields that (\ref{6.12}) holds true for all $t>s\geq 0$.

We will need also the supremum norm estimate which in view of (\ref{6.5}) and (\ref{6.11}) has the form
\begin{equation}\label{6.14}
E\sup_{0\leq t\leq T}|S_N(t)-Z_N(t)|^{4M}\leq\tilde C_{SZ}(M)(N^{-2M}+E\sup_{0\leq t\leq T}|U_N(t)-Z_N(t)|^{4M})
\end{equation}
where $C_{SZ}(M)>0$ does not depend on $N$. In order to estimate the last term in the right hand side of (\ref{6.14})
 we proceed similarly to (\ref{6.7}) with a slight modification to obtain a uniform in $t$ estimate. Namely, we write
 \begin{eqnarray*}
 &\sup_{0\leq t\leq T}|U_N(t)-Z_N(t)|^{4M}\\
 &\leq N^{-2M}\max_{0\leq m\leq NT\hat L^2}\max_{0\leq l<\sup_{0\leq t\leq T}
 |n(t\bar\tau N)-[tN]|}|\sum_{l=0}^k\eta(l+m)|^{4M}\\
 &\leq N^{-2M}\big(\sum_{0\leq m\leq NT\hat L^2}\max_{0\leq k<(TN)^{\frac 12+\gam}}|\sum_{l=0}^k\eta(l+m)|^{4M}\\
 &+(\sum_{0\leq m\leq NT\hat L^2}\max_{(TN)^{\frac 12+\gam}\leq k<NT\hat L^2}|\sum_{l=0}^k\eta(l+m)|^{4M})\\
 &\times\bbI_{\{\sup_{0\leq t\leq T}|n(t\bar\tau N)-[tN]|\geq(tN)^{\frac 12+\gam}\}}\big).
 \end{eqnarray*}
 Hence, similarly to the proof of (\ref{6.8}),
 \begin{equation}\label{6.15}
E\sup_{0\leq t\leq T}|S_N(t)-Z_N(t)|^{4M}\leq C^\eta_{SZ}(M)(N^{-M(1-2\gam)+1}+N^{-2M\gam+1}+N^{-2M})
\end{equation}
where $C^\eta_{SZ}(M)>0$ does not depend on $N$ and choosing $M>2$ we can achieve that both $M(1-2\gam)-1\geq\del M$ and
 $2M\gam-1\geq\del M$ for some $\del>0$.

Next, we estimate $E\| S_N-Z_N\|^{4M}_{p,[0,T]}$. Let $0=t_0<t_1<...<t_m=T$ and set
\[
\cI_i=|S_N(t_{i+1})-Z_N(t_{i+1})-S_N(t_i)+Z(t_i)|,\, i=0,1,...,m-1.
\]
Then
\begin{equation}\label{6.16}
\sum_{0\leq i<m}\cI_i^p\leq I_1+2^{2p-1}(I_2+I_3)
\end{equation}
where
\begin{eqnarray*}
&I_1=\sum_{i:\, t_{i+1}-t_i>N^{-(1-\al)}}\cI_i^p,\,\, I_2=\sum_{i:\, 0< t_{i+1}-t_i\leq N^{-(1-\al)}}|S_N(t_{i+1})-
S_N(t_i)|^p,\\
& I_3=\sum_{i:\, 0< t_{i+1}-t_i\leq N^{-(1-\al)}}|Z_N(t_{i+1})-Z_N(t_i)|^p
\end{eqnarray*}
and $\al\in(0,1)$ will be chosen later on. Observe that the sum in $I_1$ contains no more than $TN^{1-\al}$ terms
which together with (\ref{6.15}) yields that
\begin{eqnarray}\label{6.17}
&E\sup_{0=t_0<t_1<...<t_m=T}I_1^{4M/p}\\
&\leq(TN^{1-\al})^{\frac {4M}p-1}E\sup_{0=t_0<t_1<...<t_m=T}\max_{i:\, t_{i+1}-t_i>N^{-(1-\al)}}\cI_i^{4M}\nonumber\\
&\leq(TN^{1-\al})^{\frac {4M}p}2^{4M}\sup_{0\leq t\leq T}|S_N(t)-Z_N(t)|^{4M}\nonumber\\
&\leq C_I^{(1)}(M)(N^{-2M(\frac 12-\frac 2p(1-\al)-\gam)+1}+N^{-2M(\gam-\frac 2p(1-\al))+1})\nonumber
\end{eqnarray}
for some $C_I^{(1)}(M)>0$ which does not depend on $N$. We choose $\gam<\frac 14$, $\al>1-\frac {\gam p}2$ and
$M>\max((\gam-\frac 2p(1-\al))^{-1},\,(\frac 12-\frac 2p(1-\al)-\gam)^{-1})$ which bounds the right hand side
of (\ref{6.17}) by $C_I^{(1)}(M)N^{-\del M}$ for some $\del>0$ which does not depend on $N$.

In order to estimate $I_2$ we observe that $\sum_{0\leq i<m}(t_{i+1}-t_i)=T$, and so similarly to (\ref{4.6}),
\begin{eqnarray*}
&I_2\leq\sum_{i:\, 0< t_{i+1}-t_i\leq N^{-(1-\al)}}(t_{i+1}-t_i)\sup_{s,t:\, 0<t-s\leq N^{-(1-\al)}}\\
&((t-s)^{-1}|S_N(t)-S_N(s)|^p)\leq TN^{-(1-\al)(\frac p2-\be-1)}\sup_{s,t:\, 0<t-s\leq N^{-(1-\al)}}
\frac {|S_N(s,t)|^p}{(t-s)^{\frac p2-\be}}
\end{eqnarray*}
where $\be>0$ is chosen so small that $\frac p2-\be-1>0$. Using Theorem \ref{hoelder} we estimate expectation of the
$4M/p$-power of the last supremum to conclude that
\begin{equation}\label{6.18}
E\sup_{0=t_0<t_1<...<t_m=T}I_2^{4M/p}\leq C_I^{(2)}N^{-\del M}
\end{equation}
for some $\del,C_I^{(2)}(M)>0$ which do not depend on $N$.

Now observe that if $t-s<N^{-1}$ then $Z_N(s,t)=0$, and so we can assume in the sum for $I_3$ that $N^{-1}\leq t_{i+1}-t_i
\leq N^{-(1-\al)}$ for all $i$. So let $0=t_0<t_1<...<t_m=T$ and since  $\sum_{0\leq i<m}(t_{i+1}-t_i)=T$, we have using (\ref{6.9})
and Proposition \ref{hoelder2} (similarly to (\ref{4.5}) and (\ref{4.6})) that,
\begin{eqnarray}\label{6.19}
&E\sup_{0=t_0<t_1<...<t_m=T}I_3^{4M/p}\\
&\leq T^{4M/p}N^{-2M(1-\al)(p(1-\be)-2)}E\max_{k,l:\, 0\leq k<l\leq k+N^\al,\, l<TN}\frac
{|Z_N(kN^{-1},lN^{-1})|^{4M}}{(|l-k|/N)^{2M(1-\be)}}\nonumber\\
&\leq C^{(4)}_{I,\be}N^{-2M(1-\al)(p(1-\be)-2)}\leq C^{(4)}_{I,\be}N^{-\del M}\nonumber
\end{eqnarray}
for some $\del,C_{I,\be}^{(4)}(M)>0$ which do not depend on $N$ provided we choose $\be>0$ so small that $p(1-\be)>2$. Finally, combining
(\ref{6.1}) and (\ref{6.16})--(\ref{6.20}) we obtain Theorem \ref{thm2.3} for $\nu=1$.

\subsection{Estimates for iterated sums and integrals}\label{subsec6.3}

We will proceed somewhat similarly to Section \ref{subsec5.3} and \ref{subsec5.4} above taking into account that now we have two
discrete time approximations $\bbZ_N$ and $\bbU_N$ with the latter being the iterated sum with a random number of terms. First, we write
\begin{eqnarray}\label{6.20}
&\bbS_N^{ij}(t)=\frac 1N\int_0^{tN\bar\tau}\xi_j(u)du\int_0^u\xi_i(v)dv
=\frac 1N\int_0^{\sig(tN\bar\tau)}\xi_j(u)du\int_0^{\sig(u)}\xi_i(v)dv\\
&+\frac 1N\int_0^{\sig(tN\bar\tau)}\xi_j(u)du\int_{\sig(u)}^u\xi_i(v)dv
+\int_{\sig(tN\bar\tau)}^{tN\bar\tau}\xi_j(u)du\int_0^{u}\xi_i(v)dv\nonumber\\
&=\bbU_N(t)+\frac 1N\sum_{k=0}^{n(tN\bar\tau)-1}F_{ij}\circ\vt^k\nonumber\\
&+\frac 1N\int_{\sig(tN\bar\tau)}^{tN\bar\tau}\xi_j(u)du\int_0^u\xi_i(v)dv\nonumber
\end{eqnarray}
where $\sig$ was defined in Section \ref{subsec6.2}. Hence,
\begin{eqnarray*}
&|\bbS_N^{ij}(t)-tEF_{ij}-\bbU_N(t)|\leq |\frac 1N\max_{0\leq k<\hat LtN\bar\tau}|\sum_{l=0}^{k}(F_{ij}\circ\vt^k-EF_{ij})|\\
&+\frac 1N|n(tN\bar\tau)-tN||EF_{ij}|+\frac 1N\int_{\sig(tN\bar\tau)}^{tN\bar\tau}|\xi_j(u)|du\max_{0\leq u\leq tN\bar\tau }|\int_0^u\xi_i(v)dv|.\nonumber
\end{eqnarray*}
Taking into account that
\[
\max_{0\leq s\leq tN\bar\tau}|\int_0^s\xi_i(v)dv|\leq\max_{0\leq k\leq tN\bar\tau\hat L}|\sum_{l=0}^{k-1}\eta_i(l)|+\sup_{0\leq s\leq tN\bar\tau}
\int_{\sig(sN\bar\tau)}^{sN\bar\tau}|\xi_i(v)|dv
\]
and relying on (\ref{6.2}), (\ref{6.10}), (\ref{6.11}), (\ref{6.15}), Lemma \ref{lem3.2} and the Cauchy--Schwarz inequality we obtain similarly to (5.9) that
\begin{eqnarray}\label{6.21}
&E\sup_{0\leq s\leq t}|\bbS^{ij}_N(s)-sEF_{ij}-\bbU_N^{ij}(s)|^{2M}\\
&\leq 3^{2M}N^{-2M}\big(E\max_{0\leq k<\hat LtN\bar\tau}|\sum_{l=0}^{k}(F_{ij}\circ\vt^k-EF_{ij})|^{2M}\nonumber\\
&+E\sup_{0\leq s\leq t}|n(sN\bar\tau)-sN|^{2M}|EF_{ij}|^{2M}\nonumber\\
&+(E\sup_{0\leq s\leq t}(\int_{\sig(sN\bar\tau)}^{sN\bar\tau}|\xi_j(u)|du)^{4M})^{1/2}\nonumber\\
&\times(E(\sup_{0\leq u\leq tN\bar\tau }|\int_0^u\xi_i(v)dv|^{4M})^{1/2}\nonumber\\
&\leq C_{\bbS F\bbZ}(M)(t^MN^{-M}+N^{-M+\frac 12}t^{M+\frac 12}+N^{-2M+1}t)\nonumber
\end{eqnarray}
where $C_{\bbS F\bbZ}>0$ does not depend on $N$ and $t$.

Introduce
\[
\hat\bbS_N^{ij}(s,t)(\om,u)=N^{-1}\int_{s\bar\tau N}^{t\bar\tau N}\xi(v,(\om,u))\int_{s\bar\tau N}^v\xi(w,(\om,u))dwdv
\]
and observe that
\begin{eqnarray}\label{6.22}
&|\bbS_N^{ij}(s,t)(\om)-\hat\bbS_N^{ij}(s,t)(\om,u)|\\
&\leq N^{-1/2}\big(|S_N^i(s,t)|\max_{0\leq m<TN}\int_{m\bar\tau}^{(m+1)\bar\tau+\hat L}|\xi_j(u)|du\nonumber\\
&+|S_N^j(s,t)|\max_{0\leq m<TN}\int_{m\bar\tau}^{(m+1)\bar\tau+\hat L}|\xi_i(v)|dv\big)\nonumber\\
&+4N^{-1}\max_{0\leq k,l<TN}\int_{k\bar\tau}^{(k+1)\bar\tau+\hat L}|\xi_i(v)|dv\int_{l\bar\tau}^{(l+1)\bar\tau+\hat L}|\xi_j(u)|du.
\nonumber\end{eqnarray}
Again, using the stationarity of the process $\xi$ on the probability space $(\hat\Om,\hat\cF,\hat P)$ we see that
\[
\int|\hat\bbS_N^{ij}(s,t)(\om,u)|^{2M}d\hat P(\om,u)=\int|\hat\bbS_N^{ij}(0,t-s)(\om,u)|^{2M}d\hat P(\om,u).
\]
Hence, by (\ref{6.22}) and the Cauchy--Schwarz inequality,
\begin{eqnarray}\label{6.23}
&E|\bbS_N^{ij}(s,t)-(t-s)EF_{ij}|^{2M}\leq 2^{4M}\big( E|\bbS_N^{ij}(0,t-s)-(t-s)EF_{ij}|^{2M}\\
&+N^{-M}\big((E|S^i_N(s,t)|^{4M})^{1/2}(E\max_{0\leq m<TN}(\int_{m\bar\tau}^{(m+1)\bar\tau}|\xi_j(u)|du)^{4M})^{1/2}
\nonumber\\
&+(E|S^j_N(s,t)|^{4M})^{1/2}(E\max_{0\leq m<TN}(\int_{m\bar\tau}^{(m+1)\bar\tau}|\xi_i(v)|dv)^{4M})^{1/2}\big)\nonumber\\
&+2^{4M}N^{-2M}(E\max_{0\leq m<TN}(\int_{m\bar\tau}^{(m+1)\bar\tau}|\xi_i(v)|dv)^{4M})^{1/2}\nonumber\\
&\times(E\max_{0\leq m<TN}(\int_{m\bar\tau}^{(m+1)\bar\tau}|\xi_j(u)|du)^{4M})^{1/2}\big).\nonumber
\end{eqnarray}
Estimating the terms in the right hand side of (\ref{6.23}) by (\ref{6.10}), (\ref{6.11}) and (\ref{6.21}) we obtain
\begin{eqnarray}\label{6.24}
&\quad\quad E|\bbS_N^{ij}(s,t)-(t-s)EF_{ij}|^{2M}\leq\bbC_{\bbS F}(M)(E|\bbZ_N^{ij}(s,t)|^{2M}+N^{-M+\frac 12}(t-s)^M\\
&+N^{-2M+1})\leq\bbC_{\bbS F}(M)(\bbC_\bbZ(M)N^{-M}(t-s)^M+N^{-M+\frac 12}(t-s)^M+N^{-2M+1})\nonumber
\end{eqnarray}
where we take into account that by Lemma \ref{lem3.2} similarly to (\ref{5.14}),
\begin{equation}\label{6.25}
E|\bbZ_N^{ij}(s,t)|^{2M}\leq \bbC_\bbZ(M)N^{-M}(t-s)^M
\end{equation}
with both constants $\bbC_{\bbS F}(M),\bbC_\bbZ(M)>0$ not dependent of $N$ and $t\geq s$. We will use (\ref{6.24}) when
$t-s\geq 1/N$ while for $0<t-s<1/N$ we will use the simple estimate
\begin{equation*}
|\bbS_N^{ij}(s,t)|\leq N^{-1}(\int_{s\bar\tau N}^{t\bar\tau N}|\xi_i(u)|du)(\int_{s\bar\tau N}^{t\bar\tau N}|\xi_j(v)|dv)
\end{equation*}
which yields relying on (\ref{6.11}) that in this case
\begin{equation}\label{6.26}
E|\bbS^{ij}_N(s,t)|^{2M}\leq\bbC_\bbS(M)N^{2M}(t-s)^{4M}\leq\bbC_\bbS(M)(t-s)^{2M}.
\end{equation}
Combining (\ref{6.24}) and (\ref{6.26}) we obtain that for any $0\leq s\leq t$,
\begin{equation}\label{6.27}
E|\bbS^{ij}_N(s,t)|^{2M}\leq\tilde\bbC_\bbS(M)(t-s)^{2M-1}
\end{equation}
where $\tilde\bbC_\bbS(M)>0$ does not depend on $N$ and $t\geq s$.

Next, it is easy to see from the definitions of $\bbU_N^{ij}$ and $\bbZ_N^{ij}$ that
\begin{eqnarray}\label{6.28}
&|\bbU_N^{ij}(t)-\bbZ_N^{ij}(t)|\leq N^{-1}|\sum_{\min([tN],n(t\bar\tau N))\leq k<\max([tN],n(t\bar\tau N))}\eta_j(k)\\
&\times\sum_{l=0}^{k-1}\eta_i(l)|\nonumber\\
&\leq |Z_N^j(N^{-1}\min([tN],n(t\bar\tau N)),\,N^{-1}\max([tN],n(t\bar\tau N)))|\nonumber\\
&\times|Z^i_N(N^{-1}\min([tN],n(t\bar\tau N))|\nonumber\\
&+|\bbZ_N^{ij}(N^{-1}\min([tN],n(t\bar\tau N)),\,N^{-1}\max([tN],n(t\bar\tau N)))|
\nonumber\end{eqnarray}
and we will estimate below moments of the terms in the right hand side of (\ref{6.28}).

Observe that $tN\bar\tau-\sig(tN\bar\tau)\leq\hat L$, and so
\begin{eqnarray}\label{6.29}
&\sup_{0\leq t\leq T}(\int_{\sig(tN\bar\tau)}^{tN\bar\tau}|\xi_i(v)|dv)^{4M}\leq\hat L^{4M-1}\sup_{0\leq t\leq T}
\int_{\sig(tN\bar\tau)}^{tN\bar\tau}|\xi_i(v)|^{4M}dv\\
&\leq \hat L^{4M-1}\max_{0\leq k\leq TN}\int_{k\hat L}^{(k+2)\hat L}|\xi_i(v)|^{4M}dv\nonumber\\
&\leq\hat L^{4M-1}\sum_{0\leq k\leq TN}\int_{k\hat L}^{(k+2)\hat L}|\xi_i(v)|^{4M}dv.\nonumber
\end{eqnarray}
By the H\" older inequality,
\begin{eqnarray}\label{6.30}
&\big(E\int_{\sig(tN\bar\tau)}^{tN\bar\tau}\int_{\sig(tN\bar\tau)}^{tN\bar\tau}|\xi_i(v)\xi_j(u)|dudv\big)^{2M}\\
&\leq\hat L^{4M-2}E\big((\int_{\sig(tN\bar\tau)}^{tN\bar\tau}|\xi_i(v)|^{2M}dv)(\int_{\sig(tN\bar\tau)}^{tN\bar\tau}
|\xi_j(u)|^{2M}du)\big)\nonumber\\
&\leq\hat L^{4M}(E\max_{0\leq k<TN}\int_{k\hat L}^{(k+2)\hat L}|\xi_i(v)|^{4M}dv)^{1/2}\nonumber\\
&\times(E\max_{0\leq k<TN}\int_{k\hat L}^{(k+2)\hat L}|\xi_j(u)|^{4M}du)^{1/2}\nonumber\\
&\leq\hat L^{4M}(\sum_{0\leq k<TN}\int_{k\hat L}^{(k+2)\hat L}E|\xi_i(v)|^{4M}dv)^{1/2}\nonumber\\
&\times(\sum_{0\leq k<TN}\int_{k\hat L}^{(k+2)\hat L}E|\xi_j(u)|^{4M}du)^{1/2}\nonumber.
\end{eqnarray}
We have also
\begin{eqnarray}\label{6.31}
&E\sup_{0\leq t\leq T}(|S_N^i(t)|\int_{\sig(tN\bar\tau)}^{tN\bar\tau}|\xi_j(u)|du)^{2M}\\
&\leq (E\sup_{0\leq t\leq T}|S_N^i(t)|^{4M})^{1/2}(E\sup_{0\leq t\leq T}(\int_{\sig(tN\bar\tau)}^{tN\bar\tau}
|\xi_j(u)|du)^{4M})^{1/2}\leq\bbC_{S\xi}(M)N^{1/2}\nonumber
\end{eqnarray}
The last factor here is estimated as in (\ref{6.14}), (\ref{6.30}) and (\ref{6.31}) while for the first factor
we write
\begin{equation}\label{6.32}
E\sup_{0\leq t\leq T}|S^i_N(t)|^{4M}\leq 2^{4M-1}(E\sup_{0\leq t\leq T}|S^i_N(t)-Z_N^i(t)|^{4M}+
E\sup_{0\leq t\leq T}|Z^i_N(t)|^{4M}).
\end{equation}
By (\ref{3.3}) of Lemma \ref{lem3.2},
\begin{equation}\label{6.33}
E\sup_{s\leq u\leq t}|Z^i_N(s,u)|^{4M}\leq\tilde C_Z(M)(t-s)^{2M}
\end{equation}
where $\tilde C_Z(M)>0$ does not depend on $N$ and $i$. Similarly to (\ref{6.17}) we obtain also that
\begin{equation}\label{6.34}
E\sup_{0\leq t\leq T}|S^i_N(t)-Z_N^i(t)|^{4M}\leq\tilde C_{SZ}(M)N^{-\del M}
\end{equation}
where $\tilde C_{SZ}(M)>0$ and $\del>0$ do not depend on $N$ and $M$ is taken appropriately large.

Next, observe that by (\ref{3.4}) of Lemma \ref{lem3.2},
\begin{equation}\label{6.35}
E|\bbZ^{ij}_N(s,t)|^{2M}\leq\tilde C_{\bbZ}(M)(t-s)^{2M}
\end{equation}
for some $\tilde C_{\bbZ}(M)>0$ which does not depend on $N$.
Applying (\ref{3.4}) of Lemma \ref{lem3.2} together with (\ref{6.2}), (\ref{6.25}) and the Cauchy-Schwarz inequality we obtain that
\begin{eqnarray}\label{6.36}
&E\sup_{0\leq t\leq T}|\bbZ_N^{ij}(N^{-1}\min([tN],n(t\bar\tau N)),\,N^{-1}\max([tN],n(t\bar\tau N)))|^{2M}\\
&\leq\sum_{0\leq k\leq NT\hat L^2}E(\max_{0\leq l<(TN)^{\frac 12+\gam}}|\bbZ_N^{ij}(kN^{-1},(k+l)N^{-1})|^{2M})\nonumber\\
&+\sum_{0\leq k\leq NT\hat L^2}E(\max_{(TN)^{\frac 12+\gam}\leq l\leq NT\hat L^2}|\bbZ_N^{ij}(kN^{-1},(k+l)N^{-1})|^{2M})
\nonumber\\
&\times\bbI_{\{\sup_{0\leq t\leq T}|n(t\bar\tau N)-tN|\geq(TN)^{\frac 12+\gam}\}}\leq C_{T,\bbZ}(M)
N^{-\gam M+1}\nonumber
\end{eqnarray}
for some $ C_{T,\bbZ}(M)>0$ which do not depend on $N$. Employing similar estimates for other terms in the right hand side 
of (\ref{6.28}) and combining (\ref{6.21}) with (\ref{6.28})--(\ref{6.36}) we conclude that
\begin{equation}\label{6.37}
E\sup_{0\leq t\leq T}|\bbS_N^{ij}(t)-tEF_{ij}-\bbZ_N^{ij}(t)|^{2M}\leq C_{\bbS\bbZ}(M)N^{-\del M}
\end{equation}
for some $\del,\, C_{\bbS\bbZ}(M)>0$ which do not depend on $N$, provided we choose $\gam$ small enough
similarly to Section \ref{subsec6.1}.

Next, we estimate $E\|\bbS_N^{ij}-tEF_{ij}-\bbZ_N^{ij}\|^{2M}_{\frac p2,[0,T]}$. Let
$0=t_0<t_1<...<t_m=T$ and set
\[
\cJ_l=|\bbS_N^{ij}(t_l,t_{l+1})-(t_{l+1}-t_l)EF_{ij}-\bbZ_N^{ij}(t_l,t_{l+1})|.
\]
Then
\begin{equation}\label{6.38}
\sum_{0\leq l<m}\cJ_l^{-p/2}\leq J_1+2^{\frac p2-1}(J_2+J_3)
\end{equation}
where
\begin{eqnarray*}
&J_1=\sum_{l:\, t_{l+1}-t_l>N^{-(1-\al)}}\cJ_l^{p/2},\\
& J_2=\sum_{l:\, t_{l+1}-t_l\leq N^{-(1-\al)}}|\bbS_N^{ij}(t_l,t_{l+1})-(t_{l+1}-t_l)EF_{ij}|^{p/2}, \\
&\mbox{and}\,\,\, J_4=\sum_{l:\, 0<t_{l+1}-t_l\leq N^{-(1-\al)}}|\bbZ_N^{ij}(t_l,t_{l+1})|^{p/2}.
\end{eqnarray*}

Observe that
\[
\bbS_N^{ij}(s,t)=\bbS_N^{ij}(t)-\bbS_N^{ij}(s)-S_N^i(s)(S^j_N(t)-S^j_N(s))
\]
and
\[
\bbZ_N^{ij}(s,t)=\bbZ_N^{ij}(t)-\bbZ_N^{ij}(s)-Z_N^i(s)(Z^j_N(t)-Z^j_N(s)).
\]
This together with (\ref{6.32})--(\ref{6.34}) and (\ref{6.37}) yields  similarly to (\ref{5.12})) that
\begin{equation}\label{6.39}
E\sup_{0\leq t\leq T}|\bbS^{ij}_N(s,t)-E\bbS_N^{ij}(s,t)-Z_N^{ij}(s,t)|^{2M}\leq\tilde C_{\bbS\bbZ}(M)N^{-\del M}
\end{equation}
where $\del,\tilde C_{\bbS\bbZ}(M)>0$ do not depend on $N$.

Now we observe that the sum for $J_1$ contains no more than $TN^{1-\al}$ terms which together with (\ref{6.40}) gives that
\begin{equation}\label{6.40}
E\sup_{0=t_0<t_1<...<t_m=T}J_1^{4M/p}\leq\tilde C_{\bbS\bbZ}(M)T^{\frac {4M}p}N^{(\frac {4M}p-1)(1-\al)-\del M}.
\end{equation}
Choosing $\al$ so that $1>\al>\frac {4M-p-\del pM}{4M-p}$ we obtain that the right hand side of (\ref{6.40}) is bounded by
$C_J^{(1)}(M)N^{-\ve M}$ for some $\ve,C_J^{(1)}(M)>0$ which do not depend on $N$.

Next, taking into account that $\sum_{0\leq l<m}(t_{l+1}-t_l)=T$ and applying Theorem \ref{hoelder}, which is possible in view of (\ref{6.27}),
we obtain similarly to (\ref{4.6}) that
\begin{eqnarray}\label{6.41}
&E\sup_{0=t_0<t_1<...<t_m=T}J_2^{4M/p}\leq T^{4M/p}N^{-2M(1-\al)(1-\gam-\frac 2p)}\\
&\times\sup_{s,t:\, 0<t-s\leq N^{-(1-\al)},\,t\leq T}\big((t-s)^{-2M(1-\gam)}|\bbS_N^{ij}(s,t)-(t-s)EF_{ij}|^{2M}\big)\leq C_J^{(2)}N^{-\del M}\nonumber
\nonumber\end{eqnarray}
for some $\del, C_J^{(2)}(M)>0$ which do not depend on $N$ provided we choose $\al\in (0,1)$ and $\gam\in(0,1-\frac 2p)$.
 Finally, using (\ref{3.4}) of Lemma \ref{lem3.2}, taking into account that $\bbZ_N^{ij}(s,t)=0$ if $|t-s|<N^{-1}$ and using Proposition \ref{hoelder2},
 which is possible in view of (\ref{6.35}), we obtain that
\begin{eqnarray}\label{6.42}
&E\sup_{0=t_0<t_1<...<t_m=T}J_3^{4M/p}\leq T^{4M/p}N^{-2M(1-\al)(1-\gam-\frac 2p)}\\
&\times\sup_{s,t:\, 0<t-s\leq N^{-(1-\al)},\,t\leq T}\big((t-s)^{-2M(1-\gam)}|\bbZ_N^{ij}(s,t)|^{2M}\big)\leq C_J^{(3)}N^{-\del M}\nonumber
\end{eqnarray}
for some $C^{(3)}_J(M)>0$ which does not depend on $N$ where, again, we choose $\al\in (0,1)$ and $\gam\in(0,1-\frac 2p)$.
Now, combining (\ref{6.1}), (\ref{6.38}) and (\ref{6.40})--(\ref{6.42}) we derive Theorem \ref{thm2.3} for $\nu=2$ while its
extension for $\nu>2$ will follow from the arguments of Section \ref{sec7}.

\section{Multiple iterated sums and integrals}\label{sec7}\setcounter{equation}{0}
\subsection{Multiplicativity or high ranks Chen's identities}\label{subsec7.1}
By the recurrence definition (\ref{rec1}),
%\[
%\bbW_N^{(n)}(s,t)=\int_s^t\bbW_N^{(n-1)}(s,v)\otimes dW_N(v)+\int_s^t\bbW_N^{(n-2)}(s,v)\otimes\Gam dv
%\]
%and, recall, $\bbW_N^{(1)}=W_N$ and $\bbW_N^{(2)}$. Let $s\leq u\leq t$. Then for $n=2$,
\begin{eqnarray*}
&\bbW_N(s,t)=\int_s^tW_N(s,v)\otimes dW_N(v)+(t-s)\Gam=\int_s^uW_N(s,v)\otimes dW_N(v)\\
&+(u-s)\Gam+\int_u^tW_N(u,v)\otimes dW_N(v)+(t-u)\Gam +W_N(s,u)\otimes W_N(u,t)\\
&=\bbW_N(s,u)+W_N(s,u)\otimes W_N(u,t)+\bbW_N(u,t).
\end{eqnarray*}
Next, we proceed by induction. Suppose that
\begin{equation}\label{7.1}
\bbW_N^{(m)}(s,t)=\sum_{k=0}^m\bbW_N^{(k)}(s,u)\otimes\bbW_N^{(m-k)}(u,t)
\end{equation}
for all $m=2,...,n-1$, where $\bbW^{(0)}_N(v,w)=1$. Then by the definition (\ref{rec2}) and the induction hypothesis
\begin{eqnarray*}
&\bbW_N^{(n)}(s,t)=\int_s^u\bbW_N^{(n-1)}(s,v)\otimes dW_N(v)+\int_s^u\bbW_N^{(n-2)}(s,v)\otimes\Gam dv\\
&+\int_u^t\bbW_N^{(n-1)}(u,v)\otimes dW_N(v)+\int_u^t\bbW_N^{(n-2)}(u,v)\otimes\Gam dv+\int_u^t(\bbW_N^{(n-1)}(s,v)\\
&-\bbW_N^{(n-1)}(u,v))\otimes dW_N(v)+\int_u^t(\bbW_N^{(n-2)}(s,v)-\bbW_N^{(n-2)}(u,v))\otimes\Gam dv\\
&=\bbW_N^{(n)}(s,u)+\bbW_N^{(n)}(u,t)+\sum_{k=1}^{n-1}\bbW_N^{(k)}(s,u)\otimes\int_u^t\bbW_N^{(n-1-k)}(u,v)\otimes dW_N(v)\\
&+\sum_{k=1}^{n-2}\bbW_N^{(k)}(s,u)\otimes\int_u^t\bbW_N^{(n-2-k)}(u,v)\otimes\Gam dv=\bbW_N^{(n)}(s,u)\\
&+\bbW_N^{(n)}(u,t)+\bbW_N^{(n-1)}(s,u)\otimes W_N(u,t)+\sum_{k=1}^{n-2}\bbW_N^{(k)}(s,u)\otimes\bbW_N^{(n-k)}(u,t)
\end{eqnarray*}
proving (\ref{7.1}) for $m=n$.

Concerning the iterated sums and integrals $\bbS_N^{(n)}(s,t)$ we obtain the multiplicativity identity
\begin{equation}\label{7.2}
\bbS_N^{(m)}(s,t)=\sum_{k=0}^m\bbS_N^{(k)}(s,u)\otimes\bbS_N^{(m-k)}(u,t),\,\, s\leq u\leq t,
\end{equation}
where $\bbS_N^{(0)}\equiv 1$, proceeding by induction as above relying on the recurrence formulas
\[
\bbS_N^{(n)}(s,t)=N^{-1/2}\sum_{[sN]\leq k< [tN]}\bbS_N^{(n-1)}(s,k/N)\otimes\xi(k)
\]
in the discrete time case and
\[
\bbS_N^{(n)}(s,t)=N^{-1/2}\int_{sN}^{tN}\bbS_N^{(n-1)}(s,v/N)\otimes\xi(v)dv
\]
in the continuous time case. In these cases (\ref{7.2}) follows also from Theorem 2.1.2 in \cite{Lyo}.

\subsection{Higher rank estimates}\label{subsec7.2}
We will start with the following assertion which is a slight extenson of second halves of Theorem 2.2.1 from \cite{Lyo}
and Theorem 3.1.2 from \cite{LQ} which dealt with continuous rough paths while we do not assume any continuity here
specifying instead assumptions really needed for the proof.
\begin{proposition}\label{prop7.1}
For $1\leq\ell<\infty$ let $\bbX^{(\nu)}=\bbX^{(\nu)}(s,t),\, 0\leq s\leq t\leq T,\,\nu=1,2,...,\ell$ be two parameter
stochastic processes which form a multiplicative functional in the sense of \cite{Lyo} and \cite{LQ}, i.e. the Chen
relations (\ref{ho1}) hold true and $X^{(\nu)}(s,s)=0$. Let $\phi(s,t),\, 0\leq s\leq t\leq T$ be superadditive two argument function, i.e.
\[
\phi(s,t)\geq\phi(s,u)+\phi(u,t)\quad\mbox{for any}\quad u\in[s,t]
\]
with $\phi(s,s)=0$ for all $s$. Let $p\geq 1$ and assume that
\begin{equation}\label{7.3}
\|\bbX^{(\nu)}(s,t)\|\leq\frac {(\phi(s,t))^{\nu/p}}{\be(\nu/p)!}\,\,\,\mbox{for all}\,\,\,\nu=1,...,[p]\,\,\,\mbox{and}\,\,\, 0\leq s\leq t\leq T
\end{equation}
where $\|\cdot\|$ is a corresponding tensor norm (see Section 3.1 in \cite{LQ}), $\be\geq 2p^2(1+\sum^\infty_{r=3}(\frac 2{r-2})^{([p]+1)/p})$ and $a!$ is the
$\Gam$-function at $a+1$. Suppose that for any sequence of partitions $\cD_n=\{ 0=t_0^{(n)}<t_1^{(n)}<...<t_{m_n}^{(n)}=T\}$ such that
diam$\cD_n=\max_{0\leq j\leq m_n}|t_{j+1}^{(n)}-t_j^{(n)}|\to 0$ we have
\begin{equation}\label{7.4}
\sum_{0\leq j<m_n}\|\bbX^{(\nu)}(t_j^{(n)},t_{j+1}^{(n)})\|\to 0\quad\mbox{for}\quad\nu=[p]+1,...,\ell.
\end{equation}
Then (\ref{7.3}) remains true for $\nu=[p]+1,...,\ell$.
\end{proposition}
\begin{proof}
For any $0\leq s<t\leq T$ and any partition $\cD=\{ s=t_0<t_1<...<t_m=t\}$ we can write by (\ref{ho1}) that
\[
\bbX^{(\nu)}(s,t_{i+1})=\bbX^{(\nu)}(s,t_i)+\bbX^{(\nu)}(t_i,t_{i+1})+\sum_{k=1}^{\nu-1}\bbX^{(k)}(s,t_i)\otimes\bbX^{(\nu-k)}(t_i,t_{i+1}).
\]
Summing this in $i=0,1,...,m-1$ we obtain
\begin{equation}\label{7.5}
\bbX^{(\nu)}(s,t)=\sum_{i=0}^{m-1}\bbX^{(\nu)}(t_i,t_{i+1})+\sum_{i=1}^{m-1}\sum_{k=1}^{\nu-1}\bbX^{(k)}(s,t_i)\otimes\bbX^{(\nu-k)}(t_i,t_{i+1}).
\end{equation}
We will prove (\ref{7.3}) for $\nu=[p]+1,...,\ell$ by induction. For $\nu\leq[p]$ this is given by our assumption. Suppose that (\ref{7.3}) holds
true for $\nu=1,...,n-1$ with $\ell\geq n>[p]$ and prove it for $\nu=n$. Set
\[
X_\cD^{(n)}(s,t)=\sum_{k=1}^{n-1}\sum_{i=1}^{m-1}\bbX^{(k)}(s,t_i)\otimes\bbX^{(n-k)}(t_i,t_{i+1}).
\]
In order to estimate $\| X_\cD^{(n)}(s,t)\|$ we apply the coarsing of partitions procedure from Theorem 2.2.1 in \cite{Lyo} and
Theorem 3.1.2 in \cite{LQ}. Namely, on the first step we choose a point $t_l$ in the partition $\cD$ such that when $m>2$,
\[
\phi(t_{l-1},t_{l+1})\leq\frac {2\phi(s,t)}{m-2}
\]
which exists by a simple superadditivity argument (see Lemma 2.2.1 in \cite{Lyo}). If $m=2$ we do nothing. Considering the new partition
$\cD'=\cD\setminus\{t_l\}$ we see as in \cite{Lyo} and \cite{LQ} that
\[
X_\cD^{(n)}(s,t)-X_{\cD'}^{(n)}(s,t)=\sum_{k=1}^{n-1}\bbX^{(k)}(t_{l-1},t_l)\otimes\bbX^{(n-k)}(t_l,t_{l+1}).
\]
Continuing coarsening of partitions in the same way, using (\ref{7.3}) for $\nu=1,...,n-1$ and applying the (neo-classical) binomial
inequality from Lemma 2.2.2 in \cite{Lyo} and Theorem 3.1.1 in \cite{LQ} we arrive in view of (\ref{7.5}) at the estimate
\[
\|\bbX^{(n)}(s,t)\|\leq\|\sum_{i=1}^{m-1}\bbX^{(n)}(t_i,t_{i+1})\|+\frac {(\phi(s,t))^{n/p}}{\be(n/p)!}.
\]
Finally, replacing $\cD$ by a sequence of partitions $\cD_n$ of $[s,t]$ with diam$\cD_n\to 0$ and taking into account the assumption (\ref{7.4})
we arrive at (\ref{7.3}) for $\nu=n$ completing the induction step and the whole proof of the proposition.
\end{proof}

We will apply Proposition \ref{prop7.1} to $\bbX^{(\nu)}(s,t)=\bbS_N^{(\nu)}(s,t)$ taking into account that $\bbS_N^{(\nu)}(s,t)=0$ if $0\leq t-s<1/N$
and $\nu\geq 2$ so that (\ref{7.4}) is satisfied. It is easy to see that the Chen (multiplicativity) relations hold true for $\bbS_N^{(\nu)},\,\nu=1,2,...$
as was mentioned in Section \ref{subsec7.1}. As for the (control) superadditive function $\phi$ we set
\begin{equation}\label{7.6}
\phi(s,t)=\phi_N(s,t)=\be^p(\| S_N\|^p_{p,[s,t]}+\|\bbS_N\|^{p/2}_{p/2,[s,t]})
\end{equation}
for $t>s$ and $\phi(s,s)=0$. The superadditivity of this $\phi$ is clear from the definition of the variational norm. Thus, Proposition \ref{prop7.1}
yields that for all $0\leq s\leq t\leq T$,
\begin{equation}\label{7.7}
\|\bbS_N^{(\nu)}(s,t)\|\leq\frac {(\phi_N(s,t))^{\nu/p}}{\be(\nu/p)!},\,\nu=1,2,....
\end{equation}
Recall, that here and in what follows $\|\cdot\|$ is the norm defined by (\ref{2.8+}) in the corresponding tensor product space
 while $\|\cdot\|_{p,[s,t]}$ is the variational norm.

Let $\al\in(0,\frac 12-\frac 1p)$ be a small number. It follows from Lemma \ref{lem3.2} together with Proposition \ref{hoelder2}
that there exist random variables  $C_\al,\,\bbC_\al>0$ such that
\begin{equation}\label{7.8}
\| S_N(s,t)\|\leq C_\al|t-s|^{\frac 12-\al}\,\,\mbox{when}\,\, t-s\geq 1/N,\,\|\bbS_N(s,t)\|\leq\bbC_\al|t-s|^{1-2\al}\,\,\mbox{for all}\,\, t\geq s
\end{equation}
and
\begin{equation}\label{7.9}
EC_\al^{4M}<\infty\quad\mbox{and}\quad E\bbC_\al^{2M}<\infty.
\end{equation}
Observe also that the first inequality in (\ref{7.8}) may not hold true only if $0<t-s<1/N$ and $Ns\leq k<Nt$ for some integer $k$ in which case
$S_N(s,t)=N^{-1/2}\xi(k)$. Hence, for any partition $\cD=\{ s=t_0<t_1<t_1<...<t_m=t\}$ there may exist no more than $N(t-s)+1$ intervals $[t_i,t_{i+1}]$
such that $t_{i+1}-t_i<1/N$ and $S_N(t_i,t_{i+1})=N^{-1/2}\xi(k_i)$ for some integer $k_i$, so that $\| S_N(t_i,t_{i+1})\|=N^{-1/2}|\xi(k_i)|$ in which
case $\| S_N(t_i,t_{i+1})\|$ may not be bounded by $C_\al|t_{i+1}-t_i|^{\frac 12-\al}$. Note also that if $0\leq t-s<1/N$ and there is no integer $k$ such that
$Ns\leq k<Nt$ then $S_N(s,t)=0$, and so the first inequality in (\ref{7.8}) trivially holds true. Taking into account that $\sum_{i=0}^{m-1}(t_{i+1}-t_i)
=t-s$ and that $p(\frac 12-\al)>1$ we conclude that
\begin{eqnarray*}
&\sum_{i=0}^{m-1}\| S_N(t_i,t_{i+1})\|^p\leq C^p_\al\sum_{i\in\Gam_1(s,t)}|t_{i+1}-t_i|^{\frac p2-p\al}\\
&+N^{-p/2}\sum_{i\in\Gam_2(s,t)}|\xi(k_i)|^p\leq C_\al^p|t-s|^{\frac p2-p\al}+N^{-p/2}\Psi_{\Gam_2(s,t)}
\end{eqnarray*}
where $\Gam_1(s,t)=\{ i:\, t_{i+1}-t_i\geq 1/N\}$, $\Gam_2(s,t)=\{ i:\, t_{i+1}-t_i<1/N\},\,$ there exists an integer $k_i$ such that $Nt_i\leq k_i<Nt_{i+1}$,
the cardinality of $\Gam_2(s,t)$ does not exceed $N(t-s)+1$ and $\Psi_{\Gam_2(s,t)}=\sum_{i\in\Gam_2(s,t)}|\xi(k_i)|^p$. Hence,
\begin{equation}\label{7.10}
\| S_N\|^p_{p,[s,t]}\leq C^p_\al|t-s|^{p(\frac 12-\al)}+N^{-\frac p2}\Psi_{\Gam_2(s,t)}.
\end{equation}
Since the second inequality in (\ref{7.8}) always holds true we obtain that
\[
\sum_{i=1}^{m-1}\|\bbS_N(t_i,t_{i+1})\|^{p/2}\leq\bbC_\al^{p/2}\sum_{i=0}^{m-1}|t_{i+1}-t_i|^{p(\frac 12-\al)}\leq\bbC_\al^{p/2}|t-s|^{p(\frac 12-\al)},
\]
and so
\begin{equation}\label{7.11}
\|\bbS_N\|^{p/2}_{p/2,[s,t]}\leq\bbC_\al^{p/2}|t-s|^{p(\frac 12-\al)}.
\end{equation}
Finally, from (\ref{7.10}) and (\ref{7.11}),
\begin{equation}\label{7.12}
\phi_N(s,t)\leq\be^p((C_\al^p+\bbC_\al^{p/2})|t-s|^{p(\frac 12-\al)}+N^{-p/2}\Psi_{\Gam_2(s,t)}).
\end{equation}
Note also that by (\ref{7.7}),
\begin{equation}\label{7.13}
\|\bbS_N^{(\nu)}\|^{p/\nu}_{p/\nu,[0,T]}=\sup_\cD\sum_{i=0}^{m-1}\|\bbS_N^{(\nu)}(t_i,t_{i+1})\|^{p/\nu}\leq (\be(\nu/p)!)^{-p/\nu}\phi_N(0,T)
\end{equation}
where the supremum is taken over all partitions $\cD=\{ 0=t_0<t_1<...<t_m=T\}$ of $[0,T]$.

\subsection{Completing proofs of main results}\label{subsec7.3}

Let $\cD=\cD_{0T}=\{ 0=t_0<t_1<...<t_m=T\}$ be a finite partition of the interval $[0,T]$ and $\cD_{t_it_{i+1}}=\{ t_i=\tau_{i0}<\tau_{i1}<...
<\tau_{im_i}=t_{i+1}\}$ be partitions of $[t_i,t_{i+1}],\, i=0,1,...,m-1$ such that $N^{-\al}\leq\tau_{i,j+1}-\tau_{ij}<2N^{-\al}$ if $t_{i+1}
-t_i\geq 2N^{-\al}$ and if $t_{i+1}-t_i<2N^{-\al}$ then we take $m_i=1$ in which case $\cD_{t_it_{i+1}}$ consists of only one interval $[t_i,t_{i+1}]$.
Here, again, $0<\al<\frac 12-\frac 1p$ is a small number. If $m_i>1$ then by the Chen relations (\ref{7.1}) and (\ref{7.2}),
\begin{eqnarray*}
&\bbW_N^{(n)}(t_i,\tau_{i,r+1})=\bbW_N^{(n)}(t_i,\tau_{i,r})+\bbW_N^{(n)}(\tau_{ir},\tau_{i,r+1})\\
&+\sum_{k=1}^{n-1}\bbW^{(k)}(t_i,\tau_{ir})\otimes\bbW^{(n-k)}(\tau_{ir},\tau_{i,r+1})
\end{eqnarray*}
and
\begin{eqnarray*}
&\bbS_N^{(n)}(t_i,\tau_{i,r+1})=\bbS_N^{(n)}(t_i,\tau_{i,r})+\bbS_N^{(n)}(\tau_{ir},\tau_{i,r+1})\\
&+\sum_{k=1}^{n-1}\bbS^{(k)}(t_i,\tau_{ir})\otimes\bbS^{(n-k)}(\tau_{ir},\tau_{i,r+1}).
\end{eqnarray*}
Summing these in $r=0,1,...,m_i-1$ we obtain
\begin{eqnarray}\label{7.14}
&\bbW_N^{(n)}(t_i,t_{i+1})=\sum_{r=0}^{m_i-1}\bbW_N^{(n)}(\tau_{ir},\tau_{i,r+1})\\
&+\sum_{r=1}^{m_i-1}\sum_{k=1}^{n-1}\bbW_N^{(k)}(t_i,\tau_{ir})\otimes\bbW_N^{(n-k)}(\tau_{ir},\tau_{i,r+1})\nonumber
\end{eqnarray}
and
\begin{eqnarray}\label{7.15}
&\bbS_N^{(n)}(t_i,t_{i+1})=\sum_{r=0}^{m_i-1}\bbS_N^{(n)}(\tau_{ir},\tau_{i,r+1})\\
&+\sum_{r=1}^{m_i-1}\sum_{k=1}^{n-1}\bbS_N^{(k)}(t_i,\tau_{ir})\otimes\bbS_N^{(n-k)}(\tau_{ir},\tau_{i,r+1}).\nonumber
\end{eqnarray}
Then for $n\geq 3$,
\begin{equation}\label{7.16}
\sum_{i=0}^{m-1}\|\bbS_N^{(n)}(t_i,t_{i+1})-\bbW_N^{(n)}(t_i,t_{i+1})\|^{p/n}\leq\cI_\cD^{(1)}+\cI_\cD^{(2)}+\cI_\cD^{(3)}
\end{equation}
where
\[
\cI_\cD^{(1)}=\sum_{i=0}^{m-1}\|\sum_{j=0}^{m_i-1}\bbW_N^{(n)}(\tau_{ij},\tau_{i,j+1})\|^{p/n},\,\,
\cI_\cD^{(2)}=\sum_{i=0}^{m-1}\|\sum_{j=0}^{m_i-1}\bbS_N^{(n)}(\tau_{ij},\tau_{i,j+1})\|^{p/n}
\]
and
\begin{eqnarray*}
&\cI_\cD^{(3)}=\sum_{0\leq i<m,\, m_i>1}\|\sum_{j=1}^{m_i-1}\sum_{k=1}^{n-1}(\bbS_N^{(k)}(t_i,\tau_{ij})\otimes\bbS_N^{(n-k)}(\tau_{ij},\tau_{i,j+1})\\
&-\bbW_N^{(k)}(t_i,\tau_{ij})\otimes\bbW_N^{(n-k)}(\tau_{ij},\tau_{i,j+1})\|^{p/n}.
\end{eqnarray*}
Taking into account that $p/n<1$, $N^{-\al}\leq\tau_{i,j+1}-\tau_{ij}<2N^{-\al}$ for all $i,j$, $\frac p2(1-\al)>1$ and $\sum_{i=0}^{m-1}\sum_{j=0}^{m_i-1}
|\tau_{i,j+1}-\tau_{ij}|=T$ we have (similarly to Section \ref{subsec4.3}),
\begin{eqnarray}\label{7.17}
&\cI_\cD^{(1)}\leq\sum_{i=0}^{m-1}\sum_{j=0}^{m_i-1}\|\bbW_N^{(n)}(\tau_{ij},\tau_{i,j+1})\|^{p/n}\\
&\leq(\sum_{i=0}^{m-1}\sum_{j=0}^{m_i-1}|\tau_{i,j+1}-\tau_{ij}|^{\frac p2(1-\al)})\sup_{0\leq u<v\leq T}\frac {\|\bbW_N^{(n)}(u,v)\|^{p/n}}{|v-u|^{\frac p2(1-\al)}}\nonumber\\
&\leq T2^{\frac p2(1-\al)}N^{-\al(\frac p2(1-\al)-1)}\sup_{0\leq u<v\leq T}\frac {\|\bbW_N^{(n)}(u,v)\|^{p/n}}{|v-u|^{\frac p2(1-\al)}}.\nonumber
\end{eqnarray}
By (\ref{7.7}) and (\ref{7.12}),
\begin{eqnarray}\label{7.18}
&\cI_\cD^{(2)}\leq\sum_{i=0}^{m-1}\sum_{j=0}^{m_i-1}\|\bbS_N^{(n)}(\tau_{ij},\tau_{i,j+1})\|^{p/n}\\
&\leq (\be(n/p)!)^{-p/n}\sum_{i=0}^{m-1}\sum_{j=0}^{m_i-1}\phi_N(\tau_{ij},\tau_{i,j+1})\nonumber\\
&\leq\be^p(\be(n/p)!)^{-p/n}(T2^{ p(\frac 12-\al)}N^{-\al(p(\frac 12-\al)-1)}(C^p_\al+\bbC_\al^{p/2})+N^{-p/2}\Psi_{\Gam_2(0,T)}).\nonumber
\end{eqnarray}
Now, assuming that $n\geq 3$ and applying twice the H\" older inequality for sums of products we obtain,
\begin{eqnarray}\label{7.19}
&\cI_\cD^{(3)}\leq\sum_{k=1}^{n-1}\sum_{0\leq i<m,\, m_i>1}\sum_{j=1}^{m_i-1}\big( \|\bbS_N^{(k)}(t_i,\tau_{ij})\|^{p/n}\\
&\times\|\bbS_N^{(n-k)}(\tau_{ij},\tau_{i,j+1})-\bbW_N^{(n-k)}(\tau_{ij},\tau_{i,j+1})\|^{\frac pn}+\|\bbS_N^{(k)}(t_i,\tau_{ij})-\bbW_N^{(k)}(t_i,\tau_{ij})\|^{p/n}\nonumber\\
&\times\|\bbW_N^{(n-k)}(\tau_{ij},\tau_{i,j+1})\|^{\frac pn}\big)\leq\sum_{k=1}^{n-1}\sum_{0\leq i<m,\, m_i>1}\big((\sum_{j=1}^{m_i-1}\|\bbS_N^{(k)}(t_i,\tau_{ij})\|^{p/k})^{k/n}\nonumber\\
&\times(\sum_{j=1}^{m_i-1}\|\bbS_N^{(n-k)}(\tau_{ij},\tau_{i,j+1})-\bbW_N^{(n-k)}(\tau_{ij},\tau_{i,j+1})\|^{\frac p{n-k}})^{\frac {n-k}n}\nonumber\\
&+(\sum_{j=1}^{m_i-1}\|\bbS_N^{(k)}(t_i,\tau_{ij})-\bbW_N^{(k)}(t_i,\tau_{ij})\|^{p/k})^{k/n}\nonumber\\
&\times(\sum_{j=1}^{m_i-1}\|\bbW_N^{(n-k)}(\tau_{ij},\tau_{i,j+1})\|^{\frac p{n-k}})^{\frac {n-k}n}\big)\leq(\max_{0\leq i<m,m_i>1}m_i)\nonumber\\
&\times\sum_{k=1}^{n-1}\sum_{0\leq i<m,m_i>1}(\|\bbS_N^{(k)}\|^{p/n}_{p/k,[t_i,t_{i+1}]}\|\bbS_N^{(n-k)}-\bbW_N^{(n-k)}\|^{p/n}_{\frac p{n-k},[t_i,t_{i+1}]}\nonumber\\
&+\|\bbS_N^{(k)}-\bbW_N^{(k)}\|^{p/n}_{p/k,[t_i,t_{i+1}]}\|\bbW_N^{(n-k)}\|^{p/n}_{\frac p{n-k},[t_i,t_{i+1}]})\nonumber\\
&\leq(\max_{0\leq i<m,m_i>1}m_i)\sum_{k=1}^{n-1}\big((\sum_{0\leq i<m,m_i>1}\|\bbS_N^{(k)}\|^{p/k}_{p/k,[t_i,t_{i+1}]})^{k/n}\nonumber\\
&\times(\sum_{0\leq i<m,m_i>1}\|\bbS_N^{(n-k)}-\bbW_N^{(n-k)}\|^{\frac p{n-k}}_{\frac p{n-k},[t_i,t_{i+1}]})^{\frac {n-k}n}\nonumber\\
&+(\sum_{0\leq i<m,m_i>1}\|\bbS_N^{(k)}-\bbW_N^{(k)}\|^{p/k}_{p/k,[t_i,t_{i+1}]})^{k/n}\nonumber\\
&\times(\sum_{0\leq i<m,m_i>1}\|\bbW_N^{(n-k)}\|^{\frac p{n-k}}_{\frac p{n-k},[t_i,t_{i+1}]})^{\frac {n-k}n}\big)\nonumber\\
&\leq TN^\al\sum_{k=1}^{n-1}(\|\bbS_N^{(k)}\|^{p/n}_{p/k,[0,T]}\|\bbS_N^{(n-k)}-\bbW_N^{(n-k)}\|^{p/n}_{\frac p{n-k},[0,T]}\nonumber\\
&+\|\bbS_N^{(k)}-\bbW_N^{(k)}\|^{p/n}_{p/k,[0,T]}\|\bbW_N^{(n-k)}\|^{p/n}_{\frac p{n-k},[0,T]})\nonumber
\end{eqnarray}
where we used that when $m_i>1$ then $\tau_{i,j+1}-\tau_{ij}\geq N^{-\al}$ for each $j=0,1,...,m_i$, and so $m_i\leq (t_{i+1}-t_i)N^\al\leq TN^\al$.

Next, observe that for any integers $\nu,D\geq 1$ and $0\leq s\leq t\leq T$,
\begin{equation}\label{7.20}
E\|\bbW_N^{(\nu)}(s,t)\|^{2D}\leq C_{D,\nu}|t-s|^{D\nu}
\end{equation}
for some $C_{D}>0$ which do not depend on $N,s,t$. Indeed, proceeding by induction we note first that for $\nu=1$ we have
$\bbW_N^{(1)}(s,t)=W_N(s,t)$. Since each coordinate process $W_N^i$ of $W_N$ is the one-dimensional Brownian motion with the variance $\Gam^{ii}$ at time 1, we obtain that
\[
E|W_N(s,t)|^{2D}\leq d^{2D-1}\sum_{1\leq i\leq d}E|W_N^{i}(t-s)|^{2D}\leq d^{2D-1}(t-s)^D\sum_{1\leq i\leq d}|\Gam^{ii}|^{d}\prod_{k=1}^D(2k-1),
\]
and so (7.20) holds true in this case. For $\nu=2$ we have by (\ref{2.7}) and (\ref{2.8+}),
\begin{eqnarray*}
&E\|\bbW_N^{(2)}(s,t)\|^{2D}\leq E\big(\sum_{1\leq i,j\leq d}|\int (W_N^i(v)-W^i_N(s))dW^j_N(v)+(t-s)\Gam^{ij}|\big)^{2D}\\
&\leq (2d^2)^{2D-1}\sum_{1\leq i,j\leq d}\big(E|\int (W_N^i(v)-W^i_N(s))dW^j_N(v)|^{2D}+(t-s)^{2D}|\Gam^{ij}|^{2D}\big).
\end{eqnarray*}
By the moment estimates of stochastic integrals (see, for instance, \cite{Mao}, Section 1.7),
\begin{eqnarray*}
&E|\int (W_N^i(v)-W^i_N(s))dW^j_N(v)|^{2D}\\
&\leq D^D(2D-1)^D(t-s)^{D-1}|\Gam^{jj}|^D\int_s^tE|W_N^i(v)-W_N^i(s)|^{2D}dv=C(t-s)^{2D},
\end{eqnarray*}
where $C=D^D(2D-1)^D|\Gam^{ii}|^D|\Gam^{jj}|^D\prod_{k=1}^D(2k-1)$, and (\ref{7.20}) follows.
 Now suppose that (\ref{7.20}) is valid for $\nu=1,2,...,n-1$ and consider $\bbW_N^{(n)}(s,t)$ given by the recurrence formula (\ref{rec2}).
 By the standard moment estimates for stochastic integrals (see, for instance, \cite{Mao}, Section 1.7), the H\" older inequality (in the form 
 $(\int_s^tg(v)dv)^{2D}\leq (t-s)^{2D-1}\int_s^t|g(v)|^{2D}dv$) and the induction hypothesis we obtain
\begin{eqnarray}\label{7.21}
&E\|\bbW_N^{(n)}(s,t)\|^{2D}\leq E\big(\sum_{1\leq i_1,...,i_n\leq d}|\int_s^t\bbW_N^{i_1,...,i_{n-1}}(s,v)dW_N^{i_n}(v)\\
&+\int_s^t\bbW_N^{i_1,...,i_{n-2}}(s,v)\Gam^{i_{n-1}i_n}dv\big)^D\nonumber\\
&\leq(2d^n)^{2D-1}\sum_{1\leq i_1,...,i_n\leq d} \big(E|\int_s^t\bbW_N^{i_1,...,i_{n-1}}(s,v)dW_N^{i_n}(v)|^D\nonumber\\
&+|\Gam^{i_{n-1}i_n}|^{2D}E|\int_s^t\bbW_N^{i_1,...,i_{n-2}}(s,v)dv|^D\big)\nonumber\\
&\leq (2d^n)^{2D-1}D^D(2D-1)^Dd^{2D-1}(\sum_{i=1}^d|\Gam^{ii})|^D)\nonumber\\
&\times\sum_{1\leq i_1,...,i_{n-1}\leq d}\big(\int_s^tE|\bbW_N^{i_1,...,i_{n-1}}(s,v)|^{2D}dv\nonumber\\
&+|\Gam^{i_{n-1}i_n}|^{2D}(t-s)^{D-1)}\int_s^tE|\bbW_N^{i_1,...,i_{n-2}}(s,v)|^{2D}dv\nonumber\\
&\leq C(t-s)^{D-1}\int_s^t(v-s)^{D(n-1)}dv+\tilde C(t-s)^{2D-1}\int_s^t(v-s)^{D(n-2)}dv\nonumber\\
&\leq (\frac C{D(n-1)+1}+\frac {\tilde C}{D(n-2)+1})(t-s)^{Dn},
\nonumber\end{eqnarray}
where $C,\tilde C>0$ are some constants. This completes the induction step
and shows that (\ref{7.20}) holds true for any $\nu\geq 1$. Now (\ref{7.20}) taken with $D=2M$ together with
Theorem \ref{hoelder} considered with $\be=1/2$ (which has nothing to do with $\be$ in Proposition \ref{prop7.1}
and in estimates of the present section) yields that
\begin{equation}\label{7.22}
E\sup_{0<u<v\leq T}\frac {\|\bbW_N^{(\nu)}(u,v)\|^{4M/\nu}}{|v-u|^{2M(1-\al)}}\leq\bbC_{\bbW,\al}(M,\nu)<\infty
\end{equation}
where $\bbC_{\bbW,\al}(M,\nu)>0$ does not depend on $N$ and $\al>0$ is arbitrarily small.

Next, let $0=t_0<t_1<...<t_m=T$. Then for $0<\al<\frac 12-\frac 1p$,
 \begin{eqnarray*}
& \sum_{r=0}^{m-1}\|\bbW_N^{(\nu)}(t_r,t_{r+1}\|^{p/\nu}\leq(\sum_{r=0}^{m-1}(t_{r+1}-t_r)^{\frac p2(1-\al)})\\
&\times\sup_{0\leq u<v\leq T}\big(\frac {\|\bbW_N^{(\nu)}(u,v)\|}{|v-u|^{\frac \nu 2(1-\al)}}\big)^{p/\nu} 
\leq T^{\frac p2(1-\al)}\sup_{0\leq u<v\leq T}\big(\frac {\|\bbW_N^{(\nu)}(u,v)\|}{|v-u|^{\frac \nu 2(1-\al)}}\big)^{p/\nu}.
\end{eqnarray*}
Taking the supremum over all such partitions of $[0,T]$ and then $\frac {4M}p$-th moment we derive from (\ref{7.22}) that
\begin{equation}\label{7.23}
E\|\bbW_N^{(\nu)}\|^{4M/\nu}_{p/\nu,[0,T]}\leq T^{2M}\bbC_{\bbW,\al}(M,\nu)<\infty\,\,\mbox{for}\,\,\nu=1,2,...,4M.
\end{equation}

%By (\ref{7.12}), (\ref{7.13}) and Proposition \ref{hoelder2},
%\begin{equation}\label{7.24}
%E\|\bbS_N^{(\nu)}\|^{4M/\nu}_{p/\nu,[0,T]}\leq T^{2M}\bbC_{\bbS}(M,\nu)<\infty\,\,\mbox{for}\,\,\nu=1,2,...,4M.
%\end{equation}
%In view of (\ref{7.23}) the following triangle inequality estimate
%\begin{equation*}
%E\|\bbS_N^{(k)}\|^{4M/k}_{p/k,[0,T]}\leq 2^{4M/k}(E\|\bbS_N^{(k)}-\bbW_N^{(k)}\|^{4M/k}_{p/k,[0,T]}
%+E\|\bbW_N^{(k)}\|^{4M/k}_{p/k,[0,T]})
%\end{equation*}
%also suffices for our purposes.

By (\ref{7.12}) and (\ref{7.13}),
\begin{eqnarray}\label{7.24}
&E\|\bbS_N^{(\nu)}\|^{4M/\nu}_{p/\nu,[0,T]}\leq(\be(\nu/p)!)^{-4M/\nu}E(\phi_N(0,T))^{4M/p}\\
&\leq(\be(\nu/p)!)^{-4M/\nu}\be^{4M}2^{4M/p}\big((C^p_\al+\bbC_\al^{p/2})^{4M/p}T^{4M(\frac 12-\al)}+N^{-2M}E\Psi^{4M/p}_{\Gam_2(0,T)}\big).
\nonumber\end{eqnarray}
The last term above we estimate by
\begin{eqnarray}\label{7.24+}
&N^{-2M}E\Psi^{4M/p}_{\Gam_2(0,T)}\leq N^{-2M}E(\sum_{0\leq k<[NT]}|\xi(k)|^p)^{4M/p}\\
&\leq N^{-2M}(NT)^{\frac {4M}p-1}\sum_{0\leq k<[NT]}E|\xi(k)|^{4M}\leq T^{4M/p}N^{-2M(1-\frac 2p)}E|\xi(0)|^{4M}.\nonumber
\end{eqnarray}

Taking the supremum over all partitions $\cD$ of the interval $[0,T]$ we obtain from (\ref{7.17})--(\ref{7.19}) and (\ref{7.22})--(\ref{7.24+}) that
there exist constants $\bbC_{\bbW,\al}(M,T)$ and $\bbC_{\bbS,\al}(M,T)$ which do not depend on $N$ and such that
\begin{equation}\label{7.25}
E\sup_\cD(\cI_\cD^{(1)})^{4M/p}\leq \bbC_{\bbW,\al}(M,T)N^{-2M\al(1-\al-\frac 2p)}\quad\mbox{and}
\end{equation}
\begin{equation}\label{7.26}
E\sup_\cD(\cI_\cD^{(2)})^{4M/p}\leq \bbC_{\bbS,\al}(M,T)N^{-2M\al(1-2\al-\frac 2p)}
\end{equation}
where we took into account Theorem \ref{hoelder}. Using, in addition, the H\" older inequality for expectations of products we obtain
from (\ref{7.23})--(\ref{7.24+}) that
\begin{eqnarray}\label{7.27}
&E\sup_\cD(\cI_\cD^{(3)})^{4M/p}\leq T^{4M/p}N^{4M\al/p}(2n)^{\frac {4M}p}\\
&\times\sum_{k=1}^{n-1}\big((E\|\bbS_N^{(k)}\|^{4M/k}_{p/k,[0,T]})^{k/n}(E\|\bbS_N^{(n-k)}-
\bbW_N^{(n-k)}\|^{\frac {4M}{n-k}}_{\frac p{n-k},[0,T]})^{(n-k)/n}\nonumber\\
&+(E\|\bbS_N^{(k)}-\bbW_N^{(k)}\|^{4M/k}_{p/k,[0,T]})^{k/n}(E\|\bbW_N^{(n-k)}\|^{\frac {4M}{n-k}}_{\frac p{n-k},[0,T]})^{(n-k)/n}\big)\nonumber\\
&\leq\bbC_{\bbS,\bbW,\al}(M,T)N^{4M\al/p}\sum_{k=1}^{n-1}(E\|\bbS_N^{(k)}-\bbW_N^{(k)}\|^{4M/k}_{p/k,[0,T]})^{k/n},\nonumber
\end{eqnarray}
where $\bbC_{\bbS,\bbW,\al}(M,T)>0$ does not depend on $N$, and so by (\ref{7.16}), (\ref{7.25})--(\ref{7.27}),
\begin{eqnarray}\label{7.28}
&E\|\bbS_N^{(n)}-\bbW_N^{(n)}\|^{4M/n}_{p/n,[0,T]}\leq 3^{4M/p}(\bbC_{\bbW,\al}(M,T)N^{-2M\al(1-\al-\frac 2p)}\nonumber\\
&+\bbC_{\bbS,\al}(M,T)N^{-2M\al(1-2\al-\frac 2p)}\nonumber\\
&+\bbC_{\bbS,\bbW,\al}(M,T)N^{4M\al/p}\sum_{k=1}^{n-1}(E\|\bbS_N^{(k)}-\bbW_N^{(k)}\|^{4M/k}_{p/k,[0,T]})^{k/n}.\nonumber
\end{eqnarray}
Now we estimate $E\|\bbS_N^{(\nu)}-\bbW_N^{(\nu)}\|^{4M/\nu}_{p/\nu,[0,T]}$ by induction. By Theorem \ref{thm2.1} for $\nu=1$ and $\nu=2$, which
is already proved, we have
\begin{equation*}
E\|S_N-W_N\|^{4M}_{p,[0,T]}+E\|\bbS_N-\bbW_N\|^{2M}_{p/2,[0,T]}\leq C_{S,\bbS,W,\bbW}(M)N^{-\del_4}
\end{equation*}
for some $C_{S,\bbS,W,\bbW}(M)>0$ and $\del_4\in(0,1)$ which do not depend on $N$. Suppose that for all $\nu<n,\,\nu\geq 2$ we have
\begin{equation}\label{7.29}
E\|\bbS_N^{(\nu)}-\bbW_N^{(\nu)}\|^{4M/\nu}_{p/\nu,[0,T]}\leq C(M,\nu)N^{-\ka_\nu}
\end{equation}
for $\ka_\nu=\del_4(\frac {p-2}{2p})^\nu\frac {(\nu-1)!}{n^\nu}$ and a constant $C(M,\nu)>0$ which do not depend on $N$.
 Then taking $\al=\frac {n-1}{2Mn}\ka_{n-1}$ and $\del_4\leq 1/2$ we obtain (\ref{7.29}) for $\nu=n$ from (\ref{7.28}) completing the induction
step and the whole proof of Theorem \ref{thm2.1}.

In the continuous time setup of Theorems \ref{thm2.2} and \ref{thm2.3} we can obtain the required estimates in
the same way as above but Proposition \ref{prop7.1} is not needed here and we can apply directly Theorems 2.2.1 from \cite{Lyo}
or Theorems 3.1.2 from \cite{LQ} since now unlike the case of sums (where the corresponding
rough paths are not continuous in time), the normalized iterated integrals $\bbS_N^{(n)},\, n\geq 3$ already serve as Lyons' extensions of
the pair $\bbS^{(1)}_N,\bbS^{(2)}_N$. Indeed, the multiplicativity Chen relation for
$\bbS_N^{(n)},\, n\geq 1$ is easy to verify directly (see Theorem 2.1.2 in \cite{Lyo}). Then the equality
(\ref{7.5}) holds true. But for our iterated integrals for each $n\geq 2$,
\[
\lim_{\mbox{diam}(\cD_{st})\to 0}\sum_{0\leq r<m}\bbS_N^{(n)}(\tau_r,\tau_{r+1})=0
\]
where $\cD_{st}=\{ s=\tau_0<\tau_1<...<\tau_m=t\}$. Indeed, in our straightforward continuous time setup, by Jensen's inequality,
\[
\|\bbS_N^{(n)}(\tau_r,\tau_{r+1})\|\leq N^{-n/2}(\int_{N\tau_r}^{N\tau_{r+1}}|\xi(u)|du)^n\leq N^{\frac n2-1}|\tau_{r+1}-
\tau_r|^{n-1}\int_{N\tau_r}^{N\tau_{r+1}}|\xi(u)|^ndu,
\]
and so when $n\geq 2$,
\[
\sum_{0\leq r<m}\|\bbS_N^{(n)}(\tau_r,\tau_{r+1})\|\leq N^{\frac n2-1}\max_{0\leq r<m}|\tau_{r+1}-\tau_r|^{n-1}
\int_{Ns}^{Nt}|\xi(u)|^ndu\to 0
\]
as diam$(\cD_{st})\to 0$. In the continuous time suspension setup we have similarly to the above,
\[
\sum_{0\leq r<m}\|\bbS_N^{(n)}(\tau_r,\tau_{r+1})\|\leq N^{\frac n2-1}\bar\tau^{n-1}\max_{0\leq r<m}|\tau_{r+1}-\tau_r|^{n-1}
\int_{N\bar\tau s}^{N\bar\tau t}|\xi(u)|^ndu\to 0
\]
as diam$(\cD_{st})\to 0$. Hence, in both cases,
\[
\bbS_N^{(n)}(s,t)=\lim_{\mbox{diam}(\cD_{st})\to 0}\sum_{r=1}^{m-1}\sum_{k=1}^{n-1}\bbS_N^{(k)}(s,\tau_r)\otimes\bbS_N^{(n-k)}(\tau_r,\tau_{r+1}),
\]
and so $\bbS_N^{(n)}$'s coincide with Lyons' extension constructed in Theorem 2.2.1 from \cite{Lyo} and
in Theorem 3.1.2 from \cite{LQ}. The required estimates for
$\| S_N-W_N\|_{p,[s,t]}$ and $\|\bbS_N-\bbW_N\|_{p/2,[s,t]}$ where obtained in Sections \ref{sec5} and \ref{sec6},
 and so Theorems \ref{thm2.2} and \ref{thm2.3} follow from Theorem 2.2.1 in \cite{Lyo} or Theorem 3.1.2 in \cite{LQ}
 together with the arguments of the present subsection taking into account that in the continuous time case the control function
 $\phi$ given by (\ref{7.6}) can be estimated by (\ref{7.12}) without the last term there which together with the moment estimates
 of $\phi$ follow from Theorem \ref{hoelder} and moment estimates of Sections \ref{sec5} and \ref{sec6}.

 \begin{remark}\label{rem7.2}
 It does not seem possible, even in the continuous time case, to rely on a version of the Lyons continuity theorem
 (Theorem 2.2.2 in \cite{Lyo} and Theorem 3.1.3 in \cite{LQ}) in place of the arguments of the last subsection leading
 to the estimate (\ref{7.29}). Indeed, in order to complete the proof of Theorems \ref{thm2.2} and \ref{thm2.3} relying
 on the above theorems we need that
 \[
 \| S_N\|^p_{p,[s,t]},\,\| W_N\|^p_{p,[s,t]}\leq\be^{-p}\phi_N(s,t)\,\,\mbox{and}\,\,\|\bbS_N\|^{p/2}_{p/2,[s,t]},\,\|\bbW_N\|^{p/2}_{p/2,[s,t]}
 \leq\be^{-p/2}\phi_N(s,t)
 \]
 which does not cause a problem if we take, for instance,
 \[
 \phi_N(s,t)=\be^p(\| S_N\|^p_{p,[s,t]}+\| W_N\|^p_{p,[s,t]}+\|\bbS_N\|^{p/2}_{p/2,[s,t]}+\|\bbW_N\|^{p/2}_{p/2,[s,t]}).
 \]
 Still, in order to use these theorems from \cite{Lyo} and \cite{LQ} for our Theorems \ref{thm2.2} and \ref{thm2.3} we would need also that
 for any $0\leq s<t\leq T$ and some $\al>0$,
 \[
 \| S_N-W_N\|^p_{p.[s,t]}\leq N^{-\al}\phi_N(s,t)\,\,\mbox{and}\,\,\|\bbS_N-\bbW_N\|^{p/2}_{p/2.[s,t]}\leq N^{-\al}\phi_N(s,t).
 \]
 On the other hand, we only proved that
 \[
 \| S_N-W_N\|^p_{p.[0,T]}\leq N^{-\al}\,\,\mbox{and}\,\,\|\bbS_N-\bbW_N\|^{p/2}_{p/2.[s,t]}\leq N^{-\al}
 \]
 which do not imply better estimates for small $t-s$ required above.
 \end{remark}

%\bibliography{matz_nonarticles,matz_articles}

\begin{thebibliography}{Bow75}

\itemsep=\smallskipamount

\bibitem{Bor} Y. Borovskikh, {\em $U$-statistics in Banach Spaces}, Acta Appl. Math. 35 (1994), 213--297.

%\bibitem{Ber} M.S. Berger, {\em Nonlinearity and Functional Analysis}, Acad. Press,
%New York, 1977.

%\bibitem{Bil} P. Billingsley, {\em Convergence of Probability Measures}, 2nd ed., J.Willey,
%New York, 1999.


%\bibitem{Bor} A.N. Borodin, {\em A limit theorem for solutions of differential equations with
%random right-hand side}, Theory Probab. Appl. 22 (1977), 482--497.


\bibitem{Bow}
R. Bowen, {\em Equilibrium States and the Ergodic Theory of Anosov
Diffeomorphisms}, Lecture Notes in Math. 470, Springer--Verlag, Berlin, 1975.


\bibitem{Bra} R.C. Bradley, {\em Introduction to Strong Mixing Conditions,}
Kendrick Press, Heber City, 2007.

%\bibitem{BF} A.N. Borodin and M.I. Freidlin, {\em Fast oscillating random perturbations of dynamical
%systems with conservation laws}, Annales de l'I.H.P., sec. B, 31 (1995), 485--525.


%\bibitem{BDG} E. Bayraktar, Ya. Dolinsky and J. Guo, {\em Recombining tree approximations
%for optimal stopping for diffusions}, SIAM J. Financial Math. 9 (2018), 602--633.



%\bibitem{BK} V. Bakhtin and Yu. Kifer, {\em Diffusion approximation for slow motion in
%fully coupled averaging}, Probab. Th. Relat. Fields 129 (2004), 157--181.



\bibitem{BP} I. Berkes and W. Philipp, {\em Approximation theorems for independent and
weakly dependent random vectors}, Annals Probab. 7 (1979), 29--54.



\bibitem{BBK} P. Baldi, G. Ben Arous, G. Kerkyacharian, {\em Large deviations and the Strassen theorem in
H\" older norm}, Stoch. Proc. Appl. 42 (1992), 171--180.



%\bibitem{BM} N.N. Bogolyubov and Yu.A. Mitropol'skii, {\em Asymptotic Methods in the
%Theory of Nonlinear Oscillations}, Hindustan Publ. Co., 1961.


\bibitem{BR} R. Bowen and D. Ruelle, {\em The ergodic theory of Axiom A
flows,} Invent. Math. {\bf 29} (1975), 181--202.


\bibitem{BV} I.S.Borisov and N.V. Volodko, {\em Limit theorems and exponential inequalities for canonical 
$U$- and $V$- statistics of dependent data}, IMS Collections: High Dimensional Probability, vol. 5 (2009),
108--130.




%\bibitem{Chu} K.-L. Chung, {\em A Course in Probability}, 3d edition, Acad. Press,
%San Diego, Ca., 2001.


\bibitem{Cla} D. S. Clark, {\em A short proof of a discrete Gronwall inequality},
Discrete Appl. Math. 16 (1987), 279--281.



%\bibitem{CE} R. Cogburn and J.A. Ellison, {\em A stochastic theory of adiabatic invariance},
%Commun. Math. Phys. 149 (1992), 97--126.



\bibitem{CFKMZ} I. Chevyrev, P.K. Friz, A. Korepanov, I. Melbourne, H. Zhang,
{\em Deterministic homogenization under optimal moment assumptions for fast-slow systems.
Part 2}, Ann. l'Inst. H. Poincar\' e, Prob. Stat. 58 (2022), 1328--1350.


%\bibitem{Dol} Y. Dolinsky, {\em Applications of weak convergence for hedging of game options},
%Ann. Appl. Probab. 20 (2010), 1891--1906.



%\bibitem{DZ} G. Da Prato and J. Zabczyk, {\em Stochastic Equations in Infinite Dimensions}, Cambridge
%Univ. Press, Cambridge, 1992.



\bibitem{DET} J. Diehl, K. Ebrahimi-Fard and Nicolas Tapia, {\em Generalized iterated-sums signatores},
J. Algebra 632(2023), 801--824.


 \bibitem{DDP} H. Dehling, M. Denker and W. Philipp, {\em Invariance principles for von Mises and $U$-statistics},
 Z. Wahr. verw. Geb. 67 (1984), 139--167.
 
 
 \bibitem{DK} M. Denker and G. Keller, {\em On $U$-statistics and v. Mises' statistics for weakly
 dependent processes}, Z. Wahrschein. Geb. 64 (1983), 505--522.


\bibitem{DS}
N. Dunford and J.T. Schwartz, Linear Operators, Part I, Wiley, New York, 1958.


\bibitem{DP} H. Dehling and W. Philipp, {\em Empirical process technique for dependent
data}, In: H.G. Dehling, T. Mikosch and MSorenson (Eds.), {\em Empirical Process Technique for
 Dependent Data}, p.p. 3--113, Birkh\" auser, Boston, 2002.


%\bibitem{DeP} M. Denker and W. Philipp, {\em Approximation by Brownian motion for Gibbs
%measures and flows under a function}, Ergod. Th. \& Dynam. Sys. 4 (1984), 541--552.



\bibitem{DR} J. Diehl and J. Reizenstein, {\em Invariants of multidimensional time series based
 on their iterated-integral signature}, Acta Appl. Math. 164 (2019), 83--122.


%\bibitem{Ebe} E. Eberlein, {\em Strong approximation of continuous time stochastic processes},
%J. Multivar. Anal. 31 (1989), 220--235.


%\bibitem{Fle} W.H. Fleming, {\em Functions of Several Variables}, Springer,
%New York, 1977.


%\bibitem{Fre} M.I. Freidlin, {\em On the factorization of non-negative definite matrices},
%Theory Probab. Appl. 13 (1968), 354--356.



\bibitem{FH} P.K. Friz and M. Hairer, {\em A Course on Rough Paths}, 2nd ed., Springer, Switzerland, 2020.


\bibitem{FK} P.K. Friz and Yu. Kifer, {\em Almost sure diffusion approximation in averaging via
rough paths theory}, Electron. J. Probab. 29 (2024), no.111, 1-56.


\bibitem{FZ} P.K. Friz and H. Zhang, {\em Differential equations driven by rough paths with jumps},
J. Diff. Equat. 264 (2018), 6226--6301.



\bibitem{Ga}
D.J.H. Garling, {\em Inequalities: a Journey into Linear Analysis}, Cambridge
Univ. Press, Cambridge (2007).


\bibitem{GM} M.Galton and I. Melbourne, Iterated invariance principle for slowly mixing
dynamical systems, Ann. Inst. Henri Poincar\' e, Probab. Stat. 58 (2022), 1284--1304.


\bibitem{GS} A.L. Gibbs and F.E. Su, {\em On choosing and bounding probability metrics}, Internat.
Stat. Review (2002), 419--435.


%\bibitem{GS} U. Gruber and M. Schweizer, {\em A diffusion limit for generalized correlated
%random walks}, J. Appl. Probab. 43 (2006), 60--73.


%\bibitem{He} H. He, {\em Convergence from discrete-to continuous-time contingent claims
%prices}, Review Financial Studies 3 (1990), 523--546.


%\bibitem{GM} G.A. Gottwald and I. Melbourne, {\em Homogenization for deterministic maps
%and multiplicative noise}, Proc. Royal Soc. A 469:20130201.



%\bibitem{Had} J. Hadamard, Sur les transformations ponctuelles, Bull. Math. France 34 %(1904), 71--84.



\bibitem{Hei}
L. Heinrich, {\em Mixing properties and central limit theorem for a class of
non-identical piecewise monotonic $C^2$-transformations},
Mathematische Nachricht. 181 (1996), 185--214.




\bibitem{HK} Ye. Hafouta and  Yu. Kifer, {\em Nonconventional Limit Theorems and
 Random Dynamics}, World Scientific, Singapore, 2018.


 \bibitem{HL} B. Humbly and T. Lyons, {\em Uniqueness for the signature of a path of bounded variation
 and the reduced path group}, Ann. Math. 171 (2010), 109--167.



%\bibitem{IW} N. Ikeda and S. Watanabe, {\em Stochastic Differential Equations
%and Diffusion Processes 2nd. ed.}, North-Holland, Amsterdam, 1989.



%\bibitem{Kha66} R.Z. Khasminskii, {\em On stochastic processes defined by differential
%equations with a small parameter}, Theory Probab. Appl. 11 (1966), 211--228.


%\bibitem{Ki95} Yu. Kifer,  {\em Limit theorems in averaging for dynamical systems},
%Ergod. Th. \& Dynam. Sys. 15 (1995), 1143--1172.


%\bibitem{Ki03} Yu. Kifer,  {\em $L^2$ diffusion approximation for slow motion in averaging},
%yStochastics and Dynam. 3 (2003), 213--246.


%\bibitem{Ki95} Yu. Kifer, {\em Limit theorems in averaging for dynamical systems},
%Ergod. Th. \& Dynam. Sys. (1995), 1143--1172.


%\bibitem{Ki19} Yu. Kifer, {\em Lectures on Mathematical Finance and Related Topics},
%World Scientific, Singapore, 2020.


\bibitem{Ki20} Yu. Kifer, {\em Strong diffusion approximation in averaging and value
computation in Dynkin's games}, Ann. Appl. Probab., 34 (2024), 103--147.

\bibitem{Ki22} Yu. Kifer, {\em Strong diffusion approximation in averaging with
dynamical systems fast motions}, Israel J. Math. 251 (2022), 595--634.

\bibitem{Ki25} Yu. Kifer, {\em Almost sure approximations and laws of iterated logarithm for signatures},
arXiv: 2310.02665.

%\bibitem{KV} Y. Kifer and S.R.S. Varadhan, {Nonconventional limit theorems in discrete and continuous time
%via martingales}, Ann. Probab. 42 (2014), 649--688.


%\bibitem{KM} D. Kelly and I. Melbourne, {\em Smooth approximation of stochastic
%differential equations}, Ann. Probab. 44 (2016), 479--520.


\bibitem{KP} J. Kuelbs and W. Philipp, {\em Almost sure invariance principles for partial sums of mixing $B$-valued random variables}, Annals Probab. 8
    (1980), 1003--1036.
    
    
 \bibitem{KY} S. Kanagawa and K. Yoshihara, {\em The almost sure invariance principles of degenerate
$U$-statistics of degree two for stationary variables}, Stoch. Proc. Appl. 49 (1994), 347--356.

\bibitem{Lyo} T. Lyons, {\em Differential equations driven by rough signals}, Revista Mat. Iberoamericana 14.2 (1998): 215--310


\bibitem{LQ} T. Lyons and Z. Qian, {\em System Control and Rough Paths}, Clarendon Press, Oxford, 2002.



\bibitem{LN} A. Leucht and M.H. Neumann, {\em Dependent wild bootstrap for degenerate $U$- and $V$-statistics}, J. Multivar. Anal.
117 (2013), 257--280.


\bibitem{Mao} X. Mao, {\em Stochastic Differential Equations and Applications}, 2nd. ed.,
Woodhead, Oxford, 2010.


\bibitem{MN}
I. Melbourne and M. Nicol, {\em Almost sure invariance principle for nonuniformly
hyperbolic systems}, Commun. Math. Phys. 260 (2005), 131--146.


%\bibitem{Ni} L. Nirenberg, {\em Topics in Nonlinear Functional Analysis}, Courant
%Lect. Notes Math. 6, New York University, Courant Inst., New York 2001.


%\bibitem{MP1} D. Monrad and W. Philipp, {\em Nearby variables with nearby laws and a strong
%approximation theorem for Hilbert space valued martingales}, Probab. Th. Rel. Fields 88
%(1991), 381--404.



%\bibitem{SP} R.L. Schilling and L. Partzsch, {\em Brownian Motion. An Introduction to
%Stochastic Processes}, De Gruyter, Berlin, 2012.


%\bibitem{MP2} D. Monrad and W. Philipp, {\em The problem of embedding vector-valued martingales in
%a Gaussian process}, Theory Probab. Appl. 35 (1991), 374--377.

%\bibitem{PK} G. C. Papanicolaou and W. Kohler, {\em Asymptotic theory of mixing
%stochastic ordinary differential equations}, Comm. Pure Appl. Math. 27 (1974),
%641--668.


%\bibitem{PS} W. Philipp and W. Stout, {\em Almost sure invariance principles for partial
%sums of weakly dependent random variables}, Memoir AMS 161, 1975.


%\bibitem{Str} V. Strassen, {\em Almost sure behavior of sums of independent random variables
%and martingales}, Proc. Fifth Berkeley Symp. Math. Stat. Probab., II, Part 1, 315--343.


%\bibitem{SV} D.W. Stroock and S.R.S. Varadhan, {\em Multidimensional Diffusion processes},
%Springer-Verlag, Berlin, 1997.


%\bibitem{Zai} A. Yu. Zaitsev, Multidimensional version of a result of Sakharenko
%in the invariance principle for vectors with finite exponential moments, I--III,
%Theory Probab. Appl. 45 (2001), 624--641; 46 (2002), 490--514, 676--698.

\end{thebibliography}
%\bibliographystyle{alpha}

\end{document}